\documentclass[a4paper,reqno,11pt]{amsart}
\usepackage{mathrsfs}
\usepackage{txfonts}
\usepackage{amssymb}
\usepackage{amsfonts,amsmath,amsthm,amssymb,stmaryrd,esint,extarrows,chemarrow,esint,color,xcolor}
\usepackage[numbers,sort&compress]{natbib}
\usepackage{booktabs,tabularx}
\usepackage{hyperref}

\newtheorem{Theorem}{Theorem}[section]
\newtheorem{Lemma}{Lemma}[section]
\newtheorem{Proposition}{Proposition}[section]
\newtheorem{Remark}{Remark}[section]
\newtheorem{Corollary}{Corollary}[section]

\numberwithin{equation}{section}
\usepackage[left=1 in, right=1 in,top=1 in, bottom=1 in]{geometry}

\def\XXint#1#2#3{{\setbox0=\hbox{$#1{#2#3}{\int}$ }
\vcenter{\hbox{$#2#3$ }}\kern-.6\wd0}}

\def\r3{\mathbb{R}^3}

\makeatletter
\@namedef{subjclassname@2020}{\textup{2020} Mathematics Subject Classification}
\makeatother

\allowdisplaybreaks
\begin{document}
\bibliographystyle{plain}

\title[nematic liquid crystal flows]{Optimal decay of global strong solutions to nematic liquid crystal flows in the half-space}

\author{Haokun Chen}
\address{School of Mathematical Sciences, South China Normal University, Guangzhou, Guangdong 510631, China.}
\email[H. Chen]{chenhaokunn@163.com}

\author{Yong Wang}
\address{School of Mathematical Sciences, South China Normal University, Guangzhou, Guangdong 510631, China.}
\email[Y. Wang]{wangyongxmu@163.com}
\thanks{Corresponding author: Yong Wang (wangyongxmu@163.com).}

\begin{abstract}
We study asymptotic behaviors of the higher-order spatial derivatives and the first-order time derivatives for the strong solution to nematic liquid crystal flows in the half-space $\mathbb{R}_+^3$.
Furthermore, when the initial data lie in an appropriately weighted Sobolev space, we obtain the decay rates that are faster than the heat kernel.
The main tools employed in this paper are the $L^p-L^q$ estimates of the Stokes semigroup, the a priori estimates of the steady Stokes system in $\mathbb{R}_+^3$, and the representation formula of the Leray projection operator.
\end{abstract}

\maketitle

%
%
%

\section{Introduction}

	 In this paper, we consider the large time behaviors of global strong solutions to the nematic liquid crystal flows in the half-space $\mathbb{R}_{+}^{3}$ (see \cite{Lin-1989,Lin-1995}) :
\begin{align}\label{1.1-main}
	\left\{
	\begin{array}{lll}
		u_t+u\cdot\nabla u+\nabla p= \mu \,\Delta u-\lambda  \,\nabla \cdot(\nabla d\odot\nabla d),\\
d_t+u\cdot\nabla d= \theta\, (\Delta d+|\nabla d|^2d), \\
		\nabla\cdot u=0,\quad\quad\quad\quad\quad(x,t)\in \mathbb{R}_+^3\times\mathbb{R}^+,\\
	\end{array}
	\right.
\end{align}
with the initial and boundary conditions
\begin{align}\label{1.2}
	\left\{
	\begin{array}{lll}
		u=\frac{\partial d}{\partial x_3}=0,&\text{on}\quad\partial\mathbb{R}_+^3\times\mathbb{R}^+,\\
(u,d)\to(0,e_3),&\text{as} \quad|x|\to\infty,\\
		(u(x, 0), d(x,0))=(u_0(x), d_0(x)),&\text{in}\quad\mathbb{R}_+^3.
	\end{array}
	\right.
\end{align}
Here, $\mathbb{R}_{+}^{3}=\{x=(x_{1},x_{2},x_{3})\in\mathbb{R}^{3}: x_{3}>0\}$ with the boundary $\partial\mathbb{R}_{+}^{3}=\{x=(x_{1},x_{2},x_{3})\in\mathbb{R}^{3}: x_{3}=0\}$ and $e_3=(0,0,1)$. The unknown variables $u=(u_1,u_2,u_3)(x,t)$ and $p=p(x,t)$ denote the velocity field of the fluid and the pressure function, respectively, and $d=(d_1,d_2,d_3)(x,t)$ is the unit vector field representing the macroscopic orientation field of liquid crystal molecules. The notation $\nabla d\odot\nabla d$ denotes a $3\times3$ matrix whose $(i,j)-\text{th}$ entry is given by $\partial_id\cdot\partial_jd$ for $1\leq i,j\leq3$. The constants $\mu$, $\lambda$, $\theta>0$ represent the viscosity of the fluid, the competition between kinetic and potential energy, and the microscopic elastic relaxation time for the molecular orientation field, respectively. Without loss of generality, we assume $\mu=\lambda=\theta=1$ in this paper.

If the unknown function $d$ is vanishing in the system \eqref{1.1-main}, then one has the half-space problem of the classical incompressible Navier-Stokes equations:
\begin{align*}
	\left\{
	\begin{array}{lll}
		u_t+u\cdot\nabla u+\nabla p=\mu\,\Delta u &\text{in}\quad\mathbb{R}_+^n\times\mathbb{R}^+,\\
		\nabla\cdot u=0 &\text{in}\quad\mathbb{R}_+^n\times\mathbb{R}^+,\\
		u(x,t)=0&\text{on}\quad\partial\mathbb{R}_+^n\times\mathbb{R}^+,\\
		u(x, 0)=u_0(x)&\text{in}\quad\mathbb{R}_+^n,
	\end{array}
	\right.
\end{align*}
where $n\geq2$ is the spatial dimension. There are lots of literature on the existence, uniqueness and large time behaviors of solutions of the Cauchy problem and the initial-boundary value problem for the incompressible Navier-Stokes equations.
 For the Cauchy problem, Leray \cite{Leray-1934} proved the existence of global weak solutions for the incompressible Navier-Stokes equations in dimension three. Significantly, Leray posed the problem of determining whether or not weak solutions decay to zero in the $L^2$-norm as time tends to infinity. This problem remains unsolved for a long time until the emergence of Kato's seminal work \cite{Kato-1984}. Specifically, Kato proved the existence of Leray weak solutions in spatial dimensions $n\leq4$ and established the asymptotic behavior: $\|u(t)\|_{L^2}\to0$ as $t\to\infty$.
 Later, Schonbek \cite{Schonbek-1985,Schonbek-1986} introduced the Fourier splitting method and obtained the algebraic decay for weak solutions in $\mathbb{R}^3$. More precisely, under the condition that the initial velocity satisfies $u_0\in L_\sigma^2(\mathbb{R}^3)\cap L^r(\mathbb{R}^3)\ (1\leq r<2)$, Schonbek \cite{Schonbek-1985,Schonbek-1986} deduced the explicit decay rate:
 $$\| u(t)\|_{L^2(\mathbb{R}^3)}\leq C(1+t)^{-\frac 32(\frac1r-\frac12)},\ \forall\ t>0.$$
 Kajikiya-Miyakawa \cite{Kajikiya-Miyakawa-1986} generalized these results in \cite{Schonbek-1985,Schonbek-1986} to the case $n\geq2$ by using the spectral theory of self-adjoint operator in Hilbert spaces. Wiegner \cite{Wiegner-1987} also studied the $L^2$-decay problem in $\mathbb{R}^n\,(n\geq2)$ and obtained the most general results in this direction. However, most of the techniques in the whole space rely on the Fourier transform theory to deal with the nonlinear term $u\cdot\nabla u$, so the arguments developed there cannot be directly applied to the case of unbounded fields with nonempty boundaries.
 On the upper half-space $\mathbb{R}_+^n$, using the estimates of the fractional powers of the Stokes operator, Borchers-Miyakawa \cite{Borchers-Miyakawa1988} obtained the time decay rates in $L^2$ for the Navier-Stokes equations and established the results similar to the case of the Cauchy problem mentioned before. Later, Bae-Choe \cite{Bae-Choe-2001} improved their results by virtue of Ukai's solution formula (see \cite{Ukai-1987}) for the Stokes equations.
Fujigaki-Miyakawa \cite{Fujigaki-Miyakawa-2001} considered the long time behavior of strong solutions for Navier-Stokes flows in $L^q(\mathbb{R}_+^n)\ (1<q<\infty)$. Specifically, Fujigaki-Miyakawa \cite{Fujigaki-Miyakawa-2001} verified that if $u_0\in L_\sigma^q(\mathbb{R}_+^n)\cap L^1(\mathbb{R}_+^n)\ (n\geq2)$ for all $1<q<\infty$ and $\|x_n u_0(x)\|_{L^1(\mathbb{R}_+^n)}<\infty$, then it holds for any $t>0$,
$$\|\nabla ^ku(t)\|_{L^q(\mathbb{R}_+^n)}\leq C(1+t)^{-\frac {k+1}2-\frac n2(1-\frac1q)},$$
where $k=0,1$.
Han \cite{Han-2010,Han-2011,Han-2012,Han-2014*} derived the decay rates for the first-order and second-order derivatives of the Navier-Stokes flows in $L^r(\mathbb{R}_+^n)\ (1\leq r\leq\infty)$ by using $L^p-L^q$ estimates and a clever analysis on the fractional powers of the Stokes operator. In particular, the estimates
\begin{align*}
\|\nabla u(t)\|_{L^1(\mathbb{R}_+^n)}\leq Ct^{-\frac12},\quad \|\nabla u(t)\|_{L^\infty(\mathbb{R}_+^n)}\leq& Ct^{-\frac12-\frac n2},\\
\|\nabla^2  u(t)\|_{L^q(\mathbb{R}_+^n)}+\|\partial_tu(t)\|_{{L^{q}(\mathbb{R}_{+}^{n})}}+\|\nabla p(t)\|_{{L^{q}(\mathbb{R}_{+}^{n})}}\leq &Ct^{-1-\frac n2(1-\frac1q)},\quad 1<q<\infty
\end{align*}
were given in \cite{Han-2011} for any $t>0$ provided $u_0\in L_\sigma^q(\mathbb{R}_+^n)\cap L^1(\mathbb{R}_+^n)\ (n\geq2)$ for all $1<q<\infty$. Under some constraint conditions on the initial data, Han \cite{Han-2012,Han-2014} gave various weighted estimates of $\|\nabla^k u(t)\|_{L^r(\mathbb{R}_+^n)}$ for the Navier-Stokes flows, where $k=0,1,2$, $1\leq r\leq\infty$, and $r\neq1$ if $k=0$.  Applying the regularity estimates of the steady Stokes system, Han \cite{Han-2016} gave the decay results for the higher-order spatial derivatives of the Navier-Stokes flows, i.e., it was shown in \cite{Han-2016} that
\begin{align*}
\|\nabla ^{k+2}u(t)\|_{L^r(\mathbb{R}_+^n)}\leq
\left\{\begin{array}{ll}
\displaystyle		C_\varepsilon t^{\varepsilon-\frac{2+k}2-(1-\frac1r)},&\mathrm{~if~}n=2,\smallskip\\
\displaystyle Ct^{-\frac{2+k}2-\frac n2(1-\frac1r)},&\mathrm{~if~}n\geq3  \smallskip
\end{array}\right.
\end{align*}
as $t\to\infty$ for $k\geq1$ and $1<r\leq\infty$, provided that $u_0$ also belongs to $L_\sigma^q(\mathbb{R}_+^n)\cap L^1(\mathbb{R}_+^n)\ (n\geq2)$ for all $1<q<\infty$. For further details, the interested readers may refer to \cite{Schonbek-1991,Schonbek-1992,Schonbek-1995,Han-2018,Han-2024,Kozono-Ogawa-1992,he-wang-2011,Desch-2001,kozono-1989} and the references cited therein.

There is an extensive literature devoted to nematic liquid crystal flows.
The hydrodynamic theory of liquid crystals was established by Ericksen and Leslie in 1960s (see \cite{Ericksen-1961,Ericksen-1962,Leslie-1968}). In 1989, Lin \cite{Lin-1989} derived a simplified Ericksen-Leslie system $(\ref{1.1-main})$ modeling incompressible viscous liquid crystal flows. From the mathematical point of view, $(\ref{1.1-main})$ is a nonlinear coupling of the Navier-Stokes equation and the transported heat flow of harmonic maps to $\mathbb{S}^2\,$(the unit sphere in $\mathbb{R}^3$). Lin-Liu \cite{Lin-1995,Lin-1996,Lin-2000} initiated the mathematical analysis of $(\ref{1.1-main})$ by considering its Ginzburg-Landau approximation,
namely, $(\ref{1.1-main})_2$ is replaced by
$$d_t+u\cdot\nabla d= \Delta d+\frac{1}{\varepsilon^{2}}(1-|d|^{2})d.$$
More precisely, Lin-Liu \cite{Lin-1995} proved the existence of global weak solutions to system $(\ref{1.1-main})$ in dimensions two and three with Dirichlet boundary conditions. In \cite{Lin-1996}, the authors proved that the one dimensional space-time Hausdorff measure of the singular set of the suitable weak solutions is zero. For more related topics, readers can refer to  \cite{Wen-Ding-2011,Lin-2016,Hong-Xin-2012,hong-Li-xin-2014,Lei-Li-Zhang-2014,Wang-Wang-2014} and the references cited therein. At the same time, there are many works on decay properties of the liquid crystal flows. Based on the Fourier splitting method introduced by Schonbek \cite{Schonbek-1985,Schonbek-1995-}, Liu \cite{Liu-2016} studied the temporal decay properties for weak solutions in $\mathbb{R}^2$. Dai and her collaborators \cite{Dai-2012,Dai-2014} obtained the optimal decay rates of the strong solution to $(\ref{1.1-main})$ in $H^m(m\geq0)$ provided $\|u_0\|_{H^1(\mathbb{R}^3)}+\|d_0-e_3\|_{H^2(\mathbb{R}^3)}$ is sufficiently small.
Liu-Xu \cite{Liu-Xu-2015} proved that if the initial data $(u_0,d_0)\in H^m(\mathbb{R}^3)\times H^{m+1}(\mathbb{R}^3)\,(m\geq3)$ have sufficiently small $\|(u_0,\nabla d_0)\|_{L^2(\mathbb{R}^3)}-$norm, then system $(\ref{1.1-main})$ admits a unique global-in-time smooth solution $(u,d)$ and satisfies
$$\|\nabla^\ell u\|_{L^2}^2+\|\nabla^\ell\nabla d\|_{L^2}^2\leq C(1+t)^{-\ell-3/2},\quad \ell=0,1,\ldots,m,$$
provided that $(u_0,\nabla d_0)\in L^1(\mathbb{R}^3)\times L^1(\mathbb{R}^3)$.
In the case of the half-space, the Fourier transform method does not work well. In addition, the projection operator $\mathbb{P}:{L}^{r}(\mathbb{R}_{+}^{n})\to L_{\sigma}^{r}(\mathbb{R}_{+}^{n})$ is unbounded for $r=1,\infty$, and $\mathbb{P}\partial_n\ne \partial_n \mathbb{P}$, due to the noncompact boundary $\partial\mathbb{R}_{+}^{n}$, which leads to many difficulties in dealing with the decay problems.
Recently, Huang-Wang-Wen \cite{huang-wang-wen-2019} obtained the optimal time decay rates in $L^r(\mathbb{R}_+^3)$ for $r\geq1$ of global strong solutions to $(\ref{1.1-main})$, provided that $\|u_0\|_{L^3(\mathbb{R}_+^3)}+\|\nabla d_0\|_{L^3(\mathbb{R}_+^3)}$ is sufficiently small. However, in comparison to the heat equation and the Navier-Stokes equations, further research is warranted regarding the decay rates of liquid crystal system.

In this paper, we intend to establish the algebraic decay estimates of the global strong solution, including the $L^1$-decay rates of second-order derivatives and the $L^\infty$-decay rates of higher-order derivatives with respect to spatial variables, to the nematic liquid crystal flows in $L^r(\mathbb{R}_+^3)$.
To solve this problem, we usually transform the system $(\ref{1.1-main})$ into the following integral forms by $\text{Duhamel}$'s formula:
\begin{align}\label{uB1}
\left\{
\begin{array}{lll}
	\displaystyle	u(t)=e^{-t\mathbb{A}}u_0-\int_0^te^{-(t-s)\mathbb{A}}\mathbb{P}\Big(u(s)\cdot\nabla u(s)+\nabla \cdot(\nabla d(s)\odot\nabla d(s))\Big)\,{d}s,\smallskip\\
	\displaystyle	(d-e_3)(t)=e^{t\Delta}({d_0-e_3})-\int_0^te^{(t-s)\Delta}\Big(u(s)\cdot\nabla d(s)-|\nabla d(s)|^2d(s)\Big)\,{d}s,
\end{array}
\right.
\end{align}
and
\begin{align}\label{uB2}
\left\{
\begin{array}{lll}
	\displaystyle	u(t)=e^{-\frac t2\mathbb{A}}u(\frac t2)-\int_{\frac t2}^{t}e^{-(t-s)\mathbb{A}}\mathbb{P}\Big(u(s)\cdot\nabla u(s)+\nabla \cdot(\nabla d(s)\odot\nabla d(s))\Big)\,{d}s,\smallskip\\
	\displaystyle	(d-e_3)(t)=e^{\frac t2\Delta}({d-e_3})(\frac t2)-\int_\frac t2^te^{(t-s)\Delta}\Big(u(s)\cdot\nabla d(s)-|\nabla d(s)|^2d(s)\Big)\,{d}s,
\end{array}
\right.
\end{align}
where $\mathbb{A}=-\mathbb{P}\Delta$ is the Stokes operator.
The decay rates of the higher-order derivatives of the strong solution $(u,d)$ are not so easy to obtain. The difficulty lies in the strong singularity of the integral when calculating $\|(\nabla^m u,\nabla^m d)\|_{L^r(\mathbb{R}_+^3)} $ with $m\geq2$ by applying the $L^p-L^q$ estimates of the Stokes flows. For example,
$$\int_{\frac t2}^t(t-s)^{-\frac m2}\|u(s)\|_{L^2(\mathbb{R}_+^3)}^2\,ds\geq \int_{\frac t2}^t(t-s)^{-\frac m2}\,ds\,\|u(t)\|_{L^2(\mathbb{R}_+^3)}^2=+\infty,\quad \forall\ t>0 ,\  m\geq2. $$
To address this issue, the equation of $d$ is appropriately modified (see $(\ref{3.7'})$), and a continuity argument is employed to derive the $L^r$-decay estimate of $\nabla^2 d$, inspired by the work in \cite{huang-wang-wen-2019}.
By conducting a sophisticated analysis of the fractional powers of the Stokes operator $\mathbb{A}$, the strong singularity arising from $\|\nabla^2 u\|_{L^r(\mathbb{R}_+^3)} $ can be effectively avoided, see Lemma $\ref{le3.4}$. However, the method employed in Lemma $\ref{le3.4}$ is not applicable for $m \geq 3$ because the singularity is really too strong. Motivated by the work in \cite{Han-2016}, we aim to apply the regularity estimates of the steady Stokes system to overcome the difficulty appearing from $\|\nabla^m u \|_{L^r(\mathbb{R}_+^3)} $ with $m\geq3$. Besides, the standard elliptic estimates are employed to derive the decay rates of $\|\nabla^m d \|_{L^r(\mathbb{R}_+^3)}\, (m\geq3)$. Since $\nabla d$ is not divergence free, the required estimates on $\nabla \cdot(\nabla d\odot\nabla d)$ is more delicate than the nonlinear term $u\cdot\nabla u$, as stated in \cite{huang-wang-wen-2019}. In fact, the third-order and forth-order derivatives of $d$ emerge in the estimate of $\|\nabla \mathbb{P}(u\cdot\nabla u+\nabla \cdot(\nabla d\odot\nabla d))\|_{L^r(\mathbb{R}_+^3)}$, see $(\ref{3.50})$. Consequently, to formally obtain the decay rates of $\|\nabla^{k}u\|_{L^r(\mathbb{R}_+^3)}$, it is necessary to first determine the decay rates of $\|\nabla^{k+1}d\|_{L^r(\mathbb{R}_+^3)}$. To achieve this, we need to jointly solve for the decay rates of the higher derivatives of $(u,d)$, see $(\ref{3.39})$ and $(\ref{4.19})$ below. This approach differs from that used for the Navier-Stokes equations (see \cite{Han-2016}) and the magnetohydrodynamic equations (see \cite{Liu-2019}) due to the complexity of the nonlinear terms with respect to the director field $d$.

\medskip

\begin{flushleft}
\textbf{Notations.} We denote the usual $L^p$ spaces by
\end{flushleft}
$$L^p(\mathbb{R}_+^3)=\{f(x)\mid\|f(x)\|_{L^p(\mathbb{R}_+^3)}<\infty\},\quad1\leq p\leq\infty,$$
where
\begin{align*}
\|f(x)\|_{L^p(\mathbb{R}_+^3)}:=
\left\{\begin{array}{ll}
\displaystyle		\Big(\int_{\mathbb{R}_+^3}|f(x)|^p \,dx\Big)^{\frac1p},&1\leq p<\infty,\smallskip\\
\displaystyle \inf_{\substack{|E|=0\\E\subset{\mathbb{R}_+^3}}}\{\sup_{{\mathbb{R}_+^3}\setminus E}|f(x)|\},&p=\infty. \smallskip
\end{array}\right.
\end{align*}
We denote the usual Sobolev spaces by
$$W^{k,p}({\mathbb{R}_+^3})=\{f(x)\mid\|f(x)\|_{W^{k,p}(\mathbb{R}_+^3)}<\infty\},\quad k=0,1,\ldots,\quad1\leq p\leq\infty,$$
where
\begin{align*}
\|f(x)\|_{W^{k,p}(\mathbb{R}_+^3)}:=
\left\{\begin{array}{ll}
\displaystyle		\Big(\sum_{0\leq l\leq k}\|\nabla^lf(x)\|_{L^p(\mathbb{R}_+^3)}^p\Big)^{\frac1p},&1\leq p<\infty,\smallskip\\
\displaystyle \sum_{0\leq l\leq k}\|\nabla^{l}f(x)\|_{L^{\infty}(\mathbb{R}_+^3)},&p=\infty. \smallskip
\end{array}\right.
\end{align*}
For simplicity, we denote $H^k(\mathbb{R}_{+}^{3})=W^{k,2}(\mathbb{R}_{+}^{3})\ (k=1,2,\ldots)$.
And we denote by $C_{0,\sigma}^\infty(\mathbb{R}_+^3)$ the set of all $C^\infty$ divergence-free vector fields with compact support in $\mathbb{R}_+^3$. The space $L_\sigma^r(\mathbb{R}_+^3) \ (1<r<\infty)$ is the closure of $C_{0,\sigma}^\infty(\mathbb{R}_+^3)$ with respect to $\|\cdot\|_{L^r(\mathbb{R}_+^3)}$, namely,
$$L_\sigma^r(\mathbb{R}_+^3)=\left\{u=(u_1,u_2,u_3)\in L^r(\mathbb{R}_+^3)\mid\nabla\cdot u=0, \ u_3|_{\partial\mathbb{R}_+^3}=0\right\}.$$
Set
$$L_\pi^r(\mathbb{R}_+^3)=\left\{\nabla p\in L^r(\mathbb{R}_+^3)\mid p\in L_{loc}^r(\overline{\mathbb{R}_+^3})\right\}.$$
Then the Helmholtz decomposition holds (see \cite{Borchers-Miyakawa1988}):
$$L^r(\mathbb{R}_+^3)=L_\sigma^r(\mathbb{R}_+^3)\oplus L_\pi^r(\mathbb{R}_+^3),\quad1<r<\infty.$$
In addition, we denote
$$D^{k,r}(\mathbb{R}_+^3)=\left\{v\in L_{\mathrm{loc}}^1(\mathbb{R}_+^3)\mid \|\nabla^kv\|_{L^r(\mathbb{R}_+^3)}<\infty\right\},$$
and $D^k(\mathbb{R}_+^3)=D^{k,2}(\mathbb{R}_+^3).$

For completeness, we first recall the following proposition on the global existence of strong solutions to the system $(\ref{1.1-main})$-$(\ref{1.2})$, see Theorem 1.1 in \cite{huang-wang-wen-2019}.

\begin{Proposition}\label{solu1}
There exists an $\varepsilon_0>0$ such that if $u_0\in L_{\sigma}^{r}(\mathbb{R}_{+}^{3})$ and $d_0\in D^{1,r}(\mathbb{R}_+^3)$ for $r=2,3$ satisfies
$$\|u_0\|_{L^3(\mathbb{R}_+^3)}+\|\nabla d_0\|_{L^3(\mathbb{R}_+^3)}\leq\varepsilon_0,$$
then the half-space problem $(\ref{1.1-main})$-$(\ref{1.2})$ admits a unique global strong solution (u,d) such that for any $\tau>0$, the following properties hold:
\begin{align*}
\left\{
\begin{array}{lll}
	\displaystyle	u\in C([0,\infty),L^{2}(\mathbb{R}_{+}^{3})\cap L^{2}([0,\infty),D^{1}(\mathbb{R}_{+}^{3})),\\
	\displaystyle	u\in C(\mathbb{R}^+, D^2(\mathbb{R}_+^3))\cap L^2([\tau,+\infty), D^3(\mathbb{R}_+^3)),\\
\displaystyle	u_{t}\in C(\mathbb{R}^{+},L^{2}(\mathbb{R}_{+}^{3}))\cap L^{2}([\tau,{+}\infty),W^{1,2}(\mathbb{R}_{+}^{3})),\\
	\displaystyle	u\in L^{\infty}([0,\infty),L^{3}(\mathbb{R}_{+}^{3})), \nabla d\in L^{\infty}([0,\infty),L^{3}(\mathbb{R}_{+}^{3})),\\
\displaystyle	\big(|u|^{\frac32},|\nabla d|^{\frac32}\big)\in L^2([0,\infty),D^1(\mathbb{R}_+^3)),\\
	\displaystyle	\nabla d\in C([0,\infty),L^2(\mathbb{R}_+^3))\cap L^2([0,\infty),D^1(\mathbb{R}_+^3)),\\
	\displaystyle	\nabla d\in C(\mathbb{R}^+,D^2(\mathbb{R}_+^3))\cap L^2([\tau,+\infty),D^3(\mathbb{R}_+^3)),\\
\displaystyle	d_t\in L^2([0,\infty),L^2(\mathbb{R}_+^3)),\\
\displaystyle d_t\in C(\mathbb{R}^+,W^{1,2}(\mathbb{R}_+^3))\cap L^2([\tau,+\infty),D^2(\mathbb{R}_+^3)).
\end{array}
\right.
\end{align*}
\end{Proposition}

Our first theorem considers the decay rates of $\| (\nabla^k u,\nabla^k d)(t)\|_{L^p(\mathbb{R}_+^3)}(k=1,2)$ with $1\leq p\leq\infty$, where $(u,d)$ is the strong solution obtained in Proposition $\ref{solu1}$.
\begin{Theorem}\label{th1.2}
Assume $u_0, d_0-e_3\in L^1(\mathbb{R}_+^{3})$, if $(u,d,p)$ is the global strong solution obtained by Proposition $\ref{solu1}$, then the following estimates
\begin{align*}
\left\{
\begin{array}{lll}
	\displaystyle	\|\nabla u(t)\|_{L^p(\mathbb{R}_+^3)}\leq Ct^{-\frac12-\frac32(1-\frac1p)},\\
	\displaystyle	\|\nabla^2d(t)\|_{L^p(\mathbb{R}_+^3)}\leq Ct^{-1-\frac32(1-\frac1p)},\\
     \displaystyle \|\nabla^2  u(t)\|_{L^1(\mathbb{R}_+^3)}\leq Ct^{-1},\\
     \displaystyle \|\nabla^2  u(t)\|_{L^\infty(\mathbb{R}_+^3)}\leq Ct^{-\frac52},\\
     \displaystyle \|\nabla^2  u(t)\|_{L^r(\mathbb{R}_+^3)}+\|\partial_tu(t)\|_{{L^{r}(\mathbb{R}_{+}^{3})}}+\|\nabla p(t)\|_{{L^{r}(\mathbb{R}_{+}^{3})}}+\| \partial_td(t)\|_{L^r(\mathbb{R}_+^3)}\leq Ct^{-1-\frac32(1-\frac1r)}
\end{array}
\right.
\end{align*}
hold for any $t>0$, where $p\in [1,\infty]$, and $r\in (1,\infty)$.
\end{Theorem}
\begin{Remark}

 Huang-Wang-Wen \cite{huang-wang-wen-2019} obtained the decay rates of $\|\nabla u(t)\|_{L^p(\mathbb{R}_+^3)}$ and $\|\nabla^2 d(t)\|_{L^p(\mathbb{R}_+^3)}$ for $1\leq p\leq6$. Theorem \ref{th1.2} further generalizes these results to the extended range $6<p\leq \infty$.
\end{Remark}
\begin{Remark}
 The algebraic decay rates are similar to that of incompressible Navier-Stokes equations in the half-space obtained in \cite{Han-2010,Han-2011,Han-2012}.
\end{Remark}

\begin{Remark}
In order to establish the $L^1$-decay of the strong solution, we need to avoid the unboundedness of the projection operator $\mathbb{P}$ in $L^1(\mathbb{R}_+^3)$. 
To overcome this problem, we use a decomposition for $\mathbb{P}(u\cdot\nabla u+\nabla\cdot(\nabla d\odot\nabla d))$, see $(\ref{demcomposition})$, and translate the unboundedness of the operator $\mathbb{P}$ into the treatment of a carefully designed elliptic Neumann problem $(\ref{neumann})$.
\end{Remark}
\begin{Remark}
 Compared with Theorem 1.2 in \cite{huang-wang-wen-2019} (i.e., Lemma \ref{le2.5**} in this paper), the additional conditions \textquotedblleft $u_0\in D^1(\mathbb{R}_+^3)$ and $\nabla d_0\in D^1(\mathbb{R}_+^3)$\textquotedblright are not required to establish the decay estimate $\|\nabla u(t)\|_{L^1(\mathbb{R}_+^3)}\leq Ct^{-\frac12}$. This is attributed to our ability to derive the decay rate of $\nabla^3d$ in $L^r(\mathbb{R}_+^3)$, as detailed in Lemma \ref{le3.9*}.
\end{Remark}

Our second result is to establish the decay estimates for the higher-order spatial derivatives of the strong solution.
\begin{Theorem}\label{th1.3-main}
Assume $u_0, d_0-e_3\in L^1(\mathbb{R}_+^{3})$. Then for every integer $k\geq1$, there exists $t(k)>0$ such that for $1<r\leq\infty$ and $t\geq t(k)$, the strong solution $(u,d,p)$ obtained in Proposition $\ref{solu1}$ satisfies
$$\|\nabla^{2+k}u(t)\|_{L^r(\mathbb{R}_+^3)}+\|\nabla^{1+k}p(t)\|_{L^r(\mathbb{R}_+^3)}+\|\nabla^{2+k}d(t)\|_{L^r(\mathbb{R}_+^3)}\leq Ct^{-\frac{2+k}2-\frac 32(1-\frac1r)}.$$
\end{Theorem}
\begin{Remark}
 The algebraic decay rate is similar to that of the incompressible Navier-Stokes equations in the half-space obtained in \cite{Han-2016}, which can be considered the optimal compared to the heat kernel.
\end{Remark}
Our third main result states that: the more rapid decay rates of the strong solution are obtained if the given initial data lie in a suitable weighted space.
\begin{Theorem}\label{th1.3}
Under the same assumptions of Theorem $\ref{th1.2}$, if, in addition,
\begin{equation}\label{1.5}
\int_{\mathbb{R}_+^3}x_3|u_0(x)|\,dx<\infty,
\end{equation}
then the decay estimates on $u,p$ can be improved as
\begin{align}
\left\{
\begin{array}{lll}
	\displaystyle	\|\nabla u(t)\|_{L^p(\mathbb{R}_+^3)}\leq Ct^{-1-\frac32(1-\frac1p)},\\
\displaystyle	\|\nabla^2  u(t)\|_{L^\infty(\mathbb{R}_+^3)}\leq Ct^{-3},\nonumber\\
    \displaystyle \|\nabla^{2}u(t)\|_{{L^{r}(\mathbb{R}_{+}^{3})}}+\|\partial_{t}u(t)\|_{{L^{r}(\mathbb{R}_{+}^{3})}}+\|\nabla p(t)\|_{{L^{r}(\mathbb{R}_{+}^{3})}}\leq Ct^{-\frac32-\frac32(1-\frac1r)},\nonumber\\
     \displaystyle \|\nabla^{2+k}u(t)\|_{L^p(\mathbb{R}_+^3)}+\|\nabla^{1+k}p(t)\|_{L^p(\mathbb{R}_+^3)}\leq Ct^{-\frac12-\frac{2+k}2-\frac 32(1-\frac1p)},
\end{array}
\right.
\end{align}
 for $t>t_k$ with some suitable number $t_k>0$, $p\in(1,\infty]$, and $r\in(1,\infty)$.
\end{Theorem}
\begin{Remark}
 It should be pointed out that the faster decay rates of $\| u(t)\|_{L^r(\mathbb{R}_+^3)}$ and $\|\nabla u(t)\|_{L^p(\mathbb{R}_+^3)}$ are obtained by Huang-Wang-Wen \cite{huang-wang-wen-2019} for $1<r\leq\infty$ and $1<p\leq6$.
\end{Remark}
\begin{Remark}
When the condition $(\ref{1.5})$ is satisfied, the decay rates of $u$ will be accelerated. At this point, the estimate of $\|u(t)\|_{L^r(\mathbb{R}_+^3)}$ is equivalent to the estimate of $\|\nabla d(t)\|_{L^r(\mathbb{R}_+^3)}$ for all $1<r\leq\infty$, so that the $L^r$-decay rates of $u\cdot\nabla u$ aligns with that of $\nabla \cdot(\nabla d\odot\nabla d)$. In terms of $(\ref{uB2})_1$, we are able to obtain the faster decay rates of the higher-order derivatives of $u$.
\end{Remark}
\begin{Remark}
 As stated in \cite{huang-wang-wen-2019}, it remains an interesting question whether the director field $d$ satisfies improved estimates on $\nabla d$, analogously to those presented in Theorem $\ref{th1.3}$, under the condition that $\|x_3\nabla d_0\|_{L^1(\mathbb{R}_+^3)}<\infty$.
\end{Remark}

The remaining part of this paper is organized as follows.
In Section $\ref{sec2}$, we introduce some basic known results about the Stokes operator and some important inequalities which will be used throughout the paper. In Section $\ref{sec3}$, we firstly give the decay estimates of $\|\nabla^2d\|_{L^\infty}$ and $\|\nabla u\|_{L^\infty}$ for the strong solution $(u,d)$ through a continuity argument. Then, we will make full use of the estimates of the fractional powers of the Stokes operator to prove the decay estimates of $\|\nabla^2u\|_{L^p}$ in the case $1<p<\infty$. In the end, we give the decay estimates of $\|\nabla^iu\|_{L^j}$ for $(i,j)=(1,1),(2,1)$ and $(2,\infty)$. The unboundedness of the projector $\mathbb{P}$ is overcome by employing a decomposition for the nonlinear terms.
In Section $\ref{sec4'}$, the asymptotic behavior for the higher-order spacial derivatives of the strong solution
 is given in $\mathbb{R}_{+}^{3}$.
In Section $\ref{sec4}$, we establish further decay estimates of the strong solution $u$ if the initial data lie in a suitable weighted space.

\section{Preliminaries }\label{sec2}

In section 2, we list some useful lemmas which are frequently used in remaining sections. First, we recall some facts about Stokes operator $\mathbb{A}$ that can be founded in \cite{Sohr-2001}. The Stokes operator $\mathbb{A}$ is defined by $\mathbb{A}=-\mathbb{P}\Delta\colon D(\mathbb{A})\to L_{\sigma}^{2}(\mathbb{R}_{+}^{3})$ with
$$D(\mathbb{A}):=H^2(\mathbb{R}_+^3)\cap H_{0,\sigma}^1(\mathbb{R}_+^3),$$
where $$\mathbb{P}:L^r(\mathbb{R}_+^3)\mapsto L_\sigma^r(\mathbb{R}_+^3)$$
is the associated projection operator, which is bounded for any $1<r<\infty$ (see \cite{McCracken-1981}). Since $\mathbb{A}$ is a positive self-adjoint operator, there exists a uniquely determined resolution $\{ E_\lambda \mid \lambda \geq0  \} $ of identity in $L_\sigma^2(\mathbb{R}_+^3)$ such that $\mathbb{A}$ has the spectral representation
$$\mathbb{A}=\int_0^\infty\lambda\,dE_\lambda$$
with the domain
$$D(\mathbb{A})=\Big\{v\in L_\sigma^2(\mathbb{R}_+^3)\mid  \|\mathbb{A}v\|_{L^2(\mathbb{R}_+^3)}^2=\int_0^\infty\lambda^2\,d\|E_\lambda v\|_{L^2(\mathbb{R}_+^3)}^2<\infty\Big\}.$$
More generally, for any $\alpha\in(0,1)$, we define the fractional order of Stokes operator $\mathbb{A}^\alpha$ as:
$$\mathbb{A}^\alpha=\int_0^\infty\lambda^\alpha\,{d}E_\lambda$$
with the domain
$$D(\mathbb{A}^\alpha)=\Big\{v\in L_\sigma^2(\mathbb{R}_+^3)\mid  \|\mathbb{A}^\alpha v\|_{L^2(\mathbb{R}_+^3)}^2=\int_0^\infty\lambda^{2\alpha}\,d\|E_\lambda v\|_{L^2(\mathbb{R}_+^3)}^2<\infty\Big\}.$$
The operator
$$ -\Delta:D(-\Delta)\to L^2(\mathbb{R}_+^3)$$
is defined by
$$-\Delta=\int_0^\infty\lambda\,{d} \tilde{E}_{\lambda}$$
with the domain
$$D(-\Delta)=D(\Delta)=\{u\in W_0^{1,2}(\mathbb{R}_+^3)\mid  \Delta u\in L^2(\mathbb{R}_+^3)\},$$
where $ \{ \tilde{E}_\lambda \mid \lambda \geq0  \} $ denotes the resolution of identity for $-\Delta$ (see \cite[p.101]{Sohr-2001}).
The fractional order of Laplace operator
$$ (-\Delta)^\alpha=\int_0^\infty\lambda^\alpha\,d \tilde{E}_{\lambda} $$
with the domain
$$D((-\Delta)^{{\alpha}})=\Big\{v\in L^2(\mathbb{R}_+^3)\mid  \|(-\Delta)^\alpha v\|_{L^2(\mathbb{R}_+^3)}^2=\int_0^\infty\lambda^{2\alpha} \,d\|\tilde{E}_\lambda v\|_{L^2(\mathbb{R}_+^3)}^2<\infty\Big\}$$
is well defined for $\alpha\in(0,1).$
We conclude that $E_\lambda=\tilde{E}_{\lambda}\mathbb{P}$ for all $\lambda\geq0$, since the resolution of identity is uniquely determined by $\mathbb{A}$. In this case we see that
$$\mathbb{A}^\alpha u=\int_0^\infty\lambda^\alpha \,dE_\lambda u=\int_0^\infty\lambda^\alpha \,d\tilde{E}_\lambda \mathbb{P}u=\int_0^\infty\lambda^\alpha \,d\tilde{E}_\lambda u=(-\Delta)^\alpha u$$
holds for all $u\in D(\mathbb{A}^\alpha)$ (see \cite[p.140]{Sohr-2001} for details).

It should be noted that the operator $\mathbb{A}$ on $\mathbb{R}_+^3$ generates a bounded analytic $C_0-$semigroup $\{e^{-t\mathbb{A}}\}_{t\geq0}$ on $L_\sigma^r(\mathbb{R}_+^3)$, so the solution $u$ of Stokes problem
\begin{align*}
	\left\{
	\begin{array}{lll}
		u_t-\Delta u+\nabla p=0 &\text{in}\quad\mathbb{R}_+^3\times\mathbb{R}^+,\\
		\nabla\cdot u=0 &\text{in}\quad\mathbb{R}_+^3\times\mathbb{R}^+,\\
		u(x,t)=0&\text{on}\quad\partial\mathbb{R}_+^3\times\mathbb{R}^+,\\
		u(x, 0)=u_0(x)&\text{in}\quad\mathbb{R}_+^3
	\end{array}
	\right.
\end{align*}
can be expressed by $u(t)=e^{-t\mathbb{A}}u_0$. Moreover, the Stokes flow $e^{-t\mathbb{A}}u_0$ satisfies the following $L^p-L^q$ estimates.

\begin{Lemma}\label{le2.1}(\cite{Fujigaki-Miyakawa-2001})
Let $a\in L_\sigma^q(\mathbb{R}_+^3)$,  then for any $t > 0$, the following estimate hold
\begin{equation}\label{2.2}
\|\nabla^ke^{-t\mathbb{A}}a\|_{L^p(\mathbb{R}_+^3)}\leq Ct^{-\frac k2-\frac 32(\frac1q-\frac1p)}\|a\|_{L^q(\mathbb{R}_+^3)}
\end{equation}
with $k=0,1,\ldots$ , provided that $1\leq q<p\leq\infty$ or $1<q\leq p<\infty.$

Furthermore,
\begin{equation}\label{2.3}
\|\nabla e^{-t\mathbb{A}}a\|_{L^1(\mathbb{R}_+^3)}\leq Ct^{-\frac12}\|a\|_{L^1(\mathbb{R}_+^3)},\ \forall\ t>0,
\end{equation}
and $(\ref{2.2})$ and $(\ref{2.3})$ still hold, if we replace the operator $e^{-t\mathbb{A}}$ by $e^{t\Delta}$ with $a\in L^q(\mathbb{R}_+^3).$
\end{Lemma}

\begin{Lemma}\label{le2.4}(\cite{Han-2012})
 Let $1<q<\infty$. Assume that $a\in W^{1,q}(\mathbb{R}_+^3)$ satisfies $\nabla\cdot a=0$ in $\mathbb{R}_+^3$ and $a_3|_{\partial\mathbb{R}_+^3}=0$. Then for any $0<\theta<1$ and $t>0$,
$$\|\nabla^2e^{-t\mathbb{A}}a\|_{L^\infty(\mathbb{R}_+^3)}\leq C(t^{-\frac12-\frac 3{2q}}\|\nabla a\|_{L^q(\mathbb{R}_+^3)}+
t^{-1+\frac{\theta}{2}-\frac 3{2q}}\|y_3^{-\theta} a\|_{L^q(\mathbb{R}_+^3)}).$$
\end{Lemma}

\begin{Lemma}\label{lea21}(\cite{Han-2014*})
Let $a\in W^{1,1}(\mathbb{R}_+^3)$ satisfies $\nabla\cdot a=0$ in $\mathbb{R}_+^3$ and $a_3|_{\partial\mathbb{R}_+^3}=0$. Then for any $0<\eta <1$ and $t>0$,
$$\|\nabla^2e^{-t\mathbb{A}}a\|_{L^1(\mathbb{R}_+^3)}\leq C\big(t^{-\frac{1}{2}}\|\nabla a\|_{L^1(\mathbb{R}_+^3)}+t^{-1+\frac{\eta}{2}}\int_{\mathbb{R}_+^3}y_3^{-\eta}|a(y)|\,dy\big),$$
and
$$\|\nabla^2e^{-t\mathbb{A}}a\|_{L^1(\mathbb{R}_+^3)}\leq Ct^{-1}\|a\|_{L^1(\mathbb{R}_+^3)}.$$
\end{Lemma}

We will also use the following elementary inequalities.
\begin{Lemma}\label{le2.2}(\cite{huang-wang-wen-2019})
For any $\alpha_1\in (0,1)$ and $\alpha_2, \alpha_3>0$, it holds that
\begin{align*}
&\left\{\begin{array}{ll}
\displaystyle \int_0^t(t-s)^{-\alpha_1}s^{-\alpha_2}\,{d}s\leq Ct^{1-\alpha_1-\alpha_2},\mathrm{~for~}0<\alpha_2<1,  \smallskip\\
	\displaystyle		\int_{\frac t2}^t(t-s)^{-\alpha_1}s^{-\alpha_2}\,{d}s\leq Ct^{1-\alpha_1-\alpha_2},\mathrm{~for~}\alpha_2>0;
\end{array}\right.\\
&\int_0^t(t-s)^{-\alpha_1}(1+s)^{-\alpha_3}\,{d}s\leq
\left\{\begin{array}{ll}
\displaystyle Ct^{-\alpha_1},\mathrm{~for~}\alpha_3>1,  \smallskip\\
	\displaystyle		Ct^{1-\alpha_1-\alpha_3},\mathrm{~for~}0<\alpha_3<1,\\
\displaystyle		Ct^{-\alpha_1}\ln(1+t),\mathrm{~for~}\alpha_3=1;
\end{array}\right.
\end{align*}
and
$$\int_{\frac t2}^t(t-s)^{-\alpha_1}(1+s)^{-\alpha_3}\,{d}s\leq Ct^{1-\alpha_1-\alpha_3}\mathrm{~for~}\alpha_3>0.$$
\end{Lemma}
To establish further decay estimates, we will use the following results obtained in \cite{huang-wang-wen-2019}.
\begin{Lemma}\label{le2.3}
Assume $u_0, d_0-e_3\in L^1(\mathbb{R}_+^{3})$. Then the strong solution $(u,d)$ of $(\ref{1.1-main})$-$(\ref{1.2})$ obtained in Proposition $\ref{solu1}$ satisfies the following decay estimates:
\begin{align}
	&\|u(t)\|_{L^r(\mathbb{R}_+^3)}+\|(d-e_3)(t)\|_{L^r(\mathbb{R}_+^3)}\leq Ct^{-\frac{3}{2}(1-\frac{1}{r})},\nonumber\\
		&\|\nabla d(t)\|_{L^s(\mathbb{R}_+^3)}\leq Ct^{-\frac12-\frac32(1-\frac1s)},\label{2.7}\\
   & \|\nabla u(t)\|_{L^p(\mathbb{R}_+^3)}\leq Ct^{-\frac12-\frac32(1-\frac1p)},\label{2.8}\\
    & \|\nabla^2d(t)\|_{L^q(\mathbb{R}_+^3)}\leq Ct^{-1-\frac32(1-\frac1q)}\label{2.9},
\end{align}
for any $t>0$, $r\in (1,\infty]$, $s\in [1,\infty]$, $p\in (1,6]$, and $q\in [1,6].$

Further, if $u_0$ satisfies $\|x_3u_0\|_{L^1(\mathbb{R}_+^3)}<\infty$, then the estimates on $u$ can be improved as
\begin{align*}
\left\{
\begin{array}{lll}
	\displaystyle	\|u(t)\|_{L^r(\mathbb{R}_+^3)}\leq Ct^{-\frac12-\frac{3}{2}(1-\frac{1}{r})},\\
    \displaystyle \|\nabla u(t)\|_{L^p(\mathbb{R}_+^3)}\leq Ct^{-1-\frac32(1-\frac1p)},
\end{array}
\right.
\end{align*}
for any $t>0$, $r\in (1,\infty]$, and $p\in (1,6]$.
\end{Lemma}
The following result on the global solution $u$ is from Theorem 1.2 in \cite{huang-wang-wen-2019}.
\begin{Lemma}\label{le2.5**}
Under the same assumptions of Theorem \ref{th1.2}, if, in addition, $u_0\in D^1(\mathbb{R}_+^3)$ and $\nabla d_0\in D^1(\mathbb{R}_+^3)$, then
$$\|\nabla u(t)\|_{L^1(\mathbb{R}_+^3)}\leq Ct^{-\frac12}$$
holds for any $t>0$.
\end{Lemma}
Next we recall the classical Gagliardo-Nirenberg-Sobolev inequality in the half-space.
\begin{Lemma}\label{gn}(\cite{Nirenberg-1959})
Let $m$ be a positive integer, $p,q,r\in[1,\infty]$. If $u\in L^q(\mathbb{R}_+^3)\cap D^{m,p}(\mathbb{R}_+^3)$, then for any integer $k\in[0,m)$,
$$\|\nabla^{k}u\|_{L^{r}(\mathbb{R}_+^3)}\leq C\|\nabla^{m}u\|_{L^{p}(\mathbb{R}_+^3)}^{\alpha}\|u\|_{L^{q}(\mathbb{R}_+^3)}^{1-\alpha},$$
where
$$\frac{1}{r}-\frac{k}{3}=\alpha(\frac{1}{p}-\frac{m}{3})+(1-\alpha)\frac{1}{q},$$
with
\begin{align*}
\left\{
\begin{array}{lll}
	\displaystyle	\alpha\in[{\frac{k}{m}},1),\quad if\ p\in(1,+\infty)\ and\ m-k-{\frac{3}{p}}\in\mathbb{N},\\
    \displaystyle \alpha\in[{\frac{k}{m}},1],\quad otherwise.
\end{array}
\right.
\end{align*}
However, when $q=+\infty$, $k=0$ and $mp<3$, the additional condition is needed: $u(x)\to 0$, as $x\to\infty$ or $u\in L^\gamma(\mathbb{R}_+^3)$ for some finite $\gamma\geq1$. The above positive constant C depends only on $m,k,p,q$ and $\alpha$.

As a special case often used in this paper, it holds that
\begin{equation}\label{gn-}
\|f\|_{L^{\infty}(\mathbb{R}_+^3)}\leq C \|f\|_{L^{s}(\mathbb{R}_+^3)}^{1-\frac3s}\|\nabla f\|_{L^{s}(\mathbb{R}_+^3)}^{\frac3s},\ \forall\ f\in W^{1,s}(\mathbb{R}_+^3),\ 3<s<\infty.
\end{equation}
\end{Lemma}

Now we recall the regularity estimates of the steady Stokes equations in the half-space (see \cite[p.254]{Galdi-2011}).
\begin{Lemma}
Let $f,g\in C_0^\infty(\overline{\mathbb{R}_+^3})$, $m\geq 0$, $q\in(1,\infty)$. Then any solution $(\phi   ,\pi )$ to the nonhomogeneous Stokes system
\begin{align*}
	\left\{
	\begin{array}{lll}
		-\Delta \phi  +\nabla\pi=f &\text{in}\quad\mathbb{R}_+^3,\\
		\nabla\cdot \phi  =g &\text{in}\quad\mathbb{R}_+^3,\\
		\phi  (x)=0&\text{on}\quad\partial\mathbb{R}_+^3,\\
		\phi  (x)\to0&\text{as}\quad  |x|\to\infty
	\end{array}
	\right.
\end{align*}
is $C^\infty$, and belongs to $W^{m+2,q}(\mathbb{R}_+^3)\times W^{m+1,q}(\mathbb{R}_+^3)$. Moreover, for each integer $\ell\geq0$, the following estimate holds
\begin{equation}\label{3.43}
\|\nabla^{\ell+2}\phi \|_{L^q(\mathbb{R}_+^3)}+\|\nabla^{\ell+1}\pi\|_{L^q(\mathbb{R}_+^3)}\leq C(\ell,q)(\|\nabla^\ell f\|_{L^q(\mathbb{R}_+^3)}+\|\nabla^{\ell+1}g\|_{L^q(\mathbb{R}_+^3)}).
\end{equation}
\end{Lemma}
Let $g=\mathcal{N}f$ denote the solution of the Neumann problem
\begin{align}\label{neumann}
	\left\{
	\begin{array}{lll}
		-\Delta g=f &\text{in}\quad\mathbb{R}_+^3,\\
		g(x)\to 0 &as \  |x|\to \infty,\\
		\partial_\nu g \mid_{\partial\mathbb{R}_+^3}=0.
	\end{array}
	\right.
\end{align}
Then (see \cite{Han-2010})
\begin{equation}\label{N}
\mathcal{N}=\int_{0}^{\infty}F(\tau )\,{d}\tau,
\end{equation}
where the operator $F(t)$ is defined as follows:
 $$(F(t)f)(x)=\int_{\mathbb{R}_+^3}[G_t(x'-y',x_3-y_3)+G_t(x'-y',x_3+y_3)]f(y)\,{d}y,$$
 and $G_t(x)=(4\pi t)^{-\frac32}e^{-\frac{|x|^2}{4t}}$ is the Gaussian kernel in $\mathbb{R}^3$. In addition, it holds that (see \cite{huang-wang-wen-2019})
 \begin{align}\label{demcomposition}
  &\mathbb{P}\big(u\cdot\nabla u+\nabla\cdot(\nabla d\odot\nabla d)\big)\nonumber\\
  =&(u\cdot\nabla u+\nabla\cdot(\nabla d\odot\nabla d))+\sum_{i,j=1}^3\nabla\mathcal{N}\partial_i\partial_j\big(u_iu_j+\langle\partial_id,\partial_jd\rangle\big).
\end{align}

Moreover, we have the following estimates.

\begin{Lemma}\label{pro3.2}(\cite{Liu-2019})
Let $0<\theta<1$. Then for any $u,v\in C_{0,\sigma}^\infty(\mathbb{R}_+^3)$
\begin{align*}
\|x_3^{-\theta}\mathbb{P}(u\cdot\nabla v)\|_{L^q(\mathbb{R}_+^3)}=&\|x_3^{-\theta}(u\cdot\nabla v+\sum_{i,j=1}^3\nabla\mathcal{N}\partial_i\partial_j(u_iv_j))\|_{L^q(\mathbb{R}_+^3)}\\
\leq &C\big(\|u\|_{L^{q_1}(\mathbb{R}_+^3)}\|v\|_{L^{q_2}(\mathbb{R}_+^3)}+\|\nabla u\|_{L^{q_1}(\mathbb{R}_+^3)}\|\nabla v\|_{L^{q_2}(\mathbb{R}_+^3)}\\
&+\|u\|_{L^{q_1}(\mathbb{R}_+^3)}\|\nabla^2v\|_{L^{q_2}(\mathbb{R}_+^3)}+\|u\|_{L^{q_1}(\mathbb{R}_+^3)}\|\nabla v\|_{L^{q_2}(\mathbb{R}_+^3)}\big),
\end{align*}
where $1\leq q<\infty$, $q\leq q_1,q_2\leq\infty$, $\frac{1}{q}=\frac{1}{q_1}+\frac{1}{q_2}$.

\end{Lemma}
We refer to \cite[p.1774]{huang-wang-wen-2019} and \cite[p.3945]{Han-2011} for the following lemma.
\begin{Lemma}\label{le3.3}
Let $1\leq k,m\leq 3$. Then for any $u,v\in C_{0,\sigma}^\infty(\mathbb{R}_+^3)$ and $d,g\in C^\infty(\mathbb{R}_+^3)$, we have the following estimates: for any $t>0$,
\begin{align}\label{3.49*}
&\|\sum_{i,j=1}^3\partial_k\mathcal{N}\partial_i\partial_j(u_iv_j+\langle\partial_id,\partial_jg\rangle)(t)\|_{L^q(\mathbb{R}_+^3)}\nonumber\\
\leq &C\big(\|u(t)\|_{L^{q_1}(\mathbb{R}_+^3)}\|v(t)\|_{L^{q_2}(\mathbb{R}_+^3)}+\|\nabla u(t)\|_{L^{q_1}(\mathbb{R}_+^3)}\|\nabla v(t)\|_{L^{q_2}(\mathbb{R}_+^3)}\nonumber\\
&+\|\nabla d(t)\|_{L^{q_1}(\mathbb{R}_+^3)}\|\nabla g(t)\|_{L^{q_2}(\mathbb{R}_+^3)}+\|\nabla^3 d(t)\|_{L^{q_1}(\mathbb{R}_+^3)}\|\nabla g(t)\|_{L^{q_2}(\mathbb{R}_+^3)}\nonumber\\
&+\|\nabla^2 d(t)\|_{L^{q_1}(\mathbb{R}_+^3)}\|\nabla^2 g(t)\|_{L^{q_2}(\mathbb{R}_+^3)}+\|\nabla d(t)\|_{L^{q_1}(\mathbb{R}_+^3)}\|\nabla ^3g(t)\|_{L^{q_2}(\mathbb{R}_+^3)}\big);
\end{align}
and
\begin{align}\label{3.50}
&\|\sum_{i,j=1}^3\partial_k\partial_m\mathcal{N}\partial_i\partial_j(u_iv_j+\langle\partial_id,\partial_jg\rangle)(t)\|_{L^q(\mathbb{R}_+^3)}\nonumber\\
\leq& C\big(\|\nabla ^2u(t)\|_{L^{q_1}(\mathbb{R}_+^3)}\|\nabla v(t)\|_{L^{q_2}(\mathbb{R}_+^3)}+\|\nabla u(t)\|_{L^{q_1}(\mathbb{R}_+^3)}\|\nabla ^2v(t)\|_{L^{q_2}(\mathbb{R}_+^3)}\nonumber\\
&+\|\nabla u(t)\|_{L^{q_1}(\mathbb{R}_+^3)}\|v(t)\|_{L^{q_2}(\mathbb{R}_+^3)}+\|u(t)\|_{L^{q_1}(\mathbb{R}_+^3)}\|\nabla v(t)\|_{L^{q_2}(\mathbb{R}_+^3)}\nonumber\\
&+\|\nabla d(t)\|_{L^{q_1}(\mathbb{R}_+^3)}\|\nabla^4g(t)\|_{L^{q_2}(\mathbb{R}_+^3)}+\|\nabla ^2d(t)\|_{L^{q_1}(\mathbb{R}_+^3)}\|\nabla^3g(t)\|_{L^{q_2}(\mathbb{R}_+^3)}\nonumber\\
&+\|\nabla^3 d(t)\|_{L^{q_1}(\mathbb{R}_+^3)}\|\nabla^2g(t)\|_{L^{q_2}(\mathbb{R}_+^3)}+\|\nabla ^4d(t)\|_{L^{q_1}(\mathbb{R}_+^3)}\|\nabla g(t)\|_{L^{q_2}(\mathbb{R}_+^3)}\nonumber\\
&+\|\nabla ^2d(t)\|_{L^{q_1}(\mathbb{R}_+^3)}\|\nabla g(t)\|_{L^{q_2}(\mathbb{R}_+^3)}+\|\nabla d(t)\|_{L^{q_1}(\mathbb{R}_+^3)}\|\nabla ^2 g(t)\|_{L^{q_2}(\mathbb{R}_+^3)}\big),
\end{align}
where $1\leq q\leq\infty$, $\frac{1}{q}=\frac{1}{q_1}+\frac{1}{q_2}$, $1\leq q_1,q_2\leq\infty$.

Especially, we have
\begin{align}\label{le2.5}
&\|\sum_{i,j=1}^3\nabla\mathcal{N}\partial_i\partial_j(u_iu_j+\langle\partial_id,\partial_jd\rangle)(t)\|_{L^1(\mathbb{R}_+^3)}\nonumber\\
\leq &C\big(\|u(t)\|_{H^1(\mathbb{R}_+^3)}^2+\|\nabla d(t)\|_{H^1(\mathbb{R}_+^3)}^2+\||\nabla d(t)||\nabla^3d(t)|\|_{L^1(\mathbb{R}_+^3)}\big).
\end{align}
\end{Lemma}
\begin{proof}
We only prove $(\ref{3.50})$, and $(\ref{3.49*})$ can be proved in the same way. Let $1\leq k,m\leq3$. Using $(\ref{N})$, we deduce for $1\leq q,q_1,q_2\leq\infty$,
\begin{align*}
&\big\|\sum_{i,j=1}^3\partial_k\partial_m\mathcal{N}\partial_i\partial_j\big(u_iv_j+\langle\partial_id,\partial_jg\rangle\big)(t)\big\|_{L^q(\mathbb{R}_+^3)}\\
=&\big\|\sum\limits_{i,j=1}^3\partial_k\partial_m\int_0^\infty{F}(\tau)\partial_i\partial_j\big(u_iv_j+\langle\partial_id,\partial_jg\rangle\big)(t)\,{d}\tau\big\|_{L^q(\mathbb{R}_+^3)}\\
\leq&\big\|\sum_{i,j=1}^3\partial_k\partial_m\big(\int_0^1+\int_1^\infty\big)G_\tau*\Big(\partial_i\partial_j\big(u_iv_j+\langle\partial_id,\partial_jg\rangle\big)\Big)_*(t)\,{d}\tau\big\|_{L^q(\mathbb{R}_+^3)}\\
\leq&\sum_{i,j=1}^3(\int_0^1{\tau}^{-\frac12}\,{d}\tau)\|\nabla G_1\|_{L^1(\mathbb{R}^3)} \big\|\partial_m\Big(\partial_i\partial_j\big(u_iv_j+\langle\partial_id,\partial_jg\rangle\big)\Big)(t)\big\|_{L^q(\mathbb{R}_+^3)}\\
    &+(\int_1^\infty{\tau}^{-\frac32}\,{d}\tau)\|\nabla^3G_1\|_{L^1(\mathbb{R}^3)}
\big(\|\nabla u(t)\|_{L^{q_1}(\mathbb{R}_+^3)}\|v(t)\|_{L^{q_2}(\mathbb{R}_+^3)}+\|u(t)\|_{L^{q_1}(\mathbb{R}_+^3)}\|\nabla v(t)\|_{L^{q_2}(\mathbb{R}_+^3)}\\
&+\|\nabla ^2d(t)\|_{L^{q_1}(\mathbb{R}_+^3)}\|\nabla g(t)\|_{L^{q_2}(\mathbb{R}_+^3)}+\|\nabla d(t)\|_{L^{q_1}(\mathbb{R}_+^3)}\|\nabla ^2 g(t)\|_{L^{q_2}(\mathbb{R}_+^3)}\big)\\
\leq& C\big(\|\nabla ^2u(t)\|_{L^{q_1}(\mathbb{R}_+^3)}\|\nabla v(t)\|_{L^{q_2}(\mathbb{R}_+^3)}+\|\nabla u(t)\|_{L^{q_1}(\mathbb{R}_+^3)}\|\nabla ^2v(t)\|_{L^{q_2}(\mathbb{R}_+^3)}\\
&+\|\nabla d(t)\|_{L^{q_1}(\mathbb{R}_+^3)}\|\nabla^4g(t)\|_{L^{q_2}(\mathbb{R}_+^3)}+\|\nabla ^2d(t)\|_{L^{q_1}(\mathbb{R}_+^3)}\|\nabla^3g(t)\|_{L^{q_2}(\mathbb{R}_+^3)}\nonumber\\
&+\|\nabla^3 d(t)\|_{L^{q_1}(\mathbb{R}_+^3)}\|\nabla^2g(t)\|_{L^{q_2}(\mathbb{R}_+^3)}+\|\nabla ^4d(t)\|_{L^{q_1}(\mathbb{R}_+^3)}\|\nabla g(t)\|_{L^{q_2}(\mathbb{R}_+^3)}\nonumber\\
&+\|\nabla u(t)\|_{L^{q_1}(\mathbb{R}_+^3)}\|v(t)\|_{L^{q_2}(\mathbb{R}_+^3)}+\|u(t)\|_{L^{q_1}(\mathbb{R}_+^3)}\|\nabla v(t)\|_{L^{q_2}(\mathbb{R}_+^3)}\\
&+\|\nabla ^2d(t)\|_{L^{q_1}(\mathbb{R}_+^3)}\|\nabla g(t)\|_{L^{q_2}(\mathbb{R}_+^3)}+\|\nabla d(t)\|_{L^{q_1}(\mathbb{R}_+^3)}\|\nabla ^2 g(t)\|_{L^{q_2}(\mathbb{R}_+^3)}\big),
\end{align*}
where $f*g$ represents the convolution of $f$ and $g$, and $f_*$ represents the even extension of function $f$ with respect to $x_3$ from $\mathbb{R}_+^3 $ to $\mathbb{R}^3$:
\begin{align*}
f_*(x^{\prime},x_3)=
\left\{\begin{array}{ll}
\displaystyle f(x^{\prime},x_3),\ &\mathrm{~if~}\ x_3\geq0,  \smallskip\\
	\displaystyle	f(x^{\prime},-x_3),\ &\mathrm{~if~}\ x_3<0.
\end{array}\right.
\end{align*}
\end{proof}

Next we recall the classical Young inequality in the half-space.
\begin{Lemma}\label{pro3.1}(\cite{Han-2016})
Let $K(x,y)$, $f(y)$ be defined in $\mathbb{R}_+^3\times\mathbb{R}_+^3$, $\mathbb{R}_+^3$, respectively, and let $1\leq p,q\leq r\leq\infty$ satisfy $1+\frac1r=\frac1p+\frac1q$. Then
$$\|\int_{\mathbb{R}_+^3}K(\cdot,y)f(y)\,dy\|_{L^r(\mathbb{R}_+^3)}\leq\sup_{x\in\mathbb{R}_+^3}\|K(x,\cdot)\|_{L^p(\mathbb{R}_+^3)}^{1-\frac pr}\sup_{y\in\mathbb{R}_+^3}\|K(\cdot,y)\|_{L^p(\mathbb{R}_+^3)}^{\frac pr}\|f\|_{L^q(\mathbb{R}_+^3)}.$$
Especially there holds
$$\|\int_{\mathbb{R}_+^3}K(\cdot,y)f(y)\,dy\|_{L^1(\mathbb{R}_+^3)}\leq\sup_{y\in\mathbb{R}_+^3}\|K(\cdot,y)\|_{L^1(\mathbb{R}_+^3)}\|f\|_{L^1(\mathbb{R}_+^3)};$$
and
$$\|\int_{\mathbb{R}_+^3}K(\cdot,y)f(y)\,dy\|_{L^r(\mathbb{R}_+^3)}\leq(\sup_{x\in\mathbb{R}_+^3}\|K(x,\cdot)\|_{L^p(\mathbb{R}_+^3)}+\sup_{y\in\mathbb{R}_+^3}\|K(\cdot,y)\|_{L^p(\mathbb{R}_+^3)})\|f\|_{L^q(\mathbb{R}_+^3)}.$$
\end{Lemma}

\begin{Lemma}\label{le3.10}
Let $0<\theta<1$. Then for any $d,g\in C^\infty(\mathbb{R}_+^3)$, the following estimate
\begin{align*}
&\|x_3^{-\theta}\mathbb{P} \nabla \cdot(\nabla d\odot\nabla g)\|_{L^q(\mathbb{R}_+^3)}\\
 \leq&C\big(\|\nabla ^3d\|_{L^{q_1}(\mathbb{R}_+^3)}\|\nabla g\|_{L^{q_2}(\mathbb{R}_+^3)}
+\|\nabla ^2d\|_{L^{q_1}(\mathbb{R}_+^3)}\|\nabla^2 g\|_{L^{q_2}(\mathbb{R}_+^3)}\\
&+\|\nabla d\|_{L^{q_1}(\mathbb{R}_+^3)}\|\nabla ^3g\|_{L^{q_2}(\mathbb{R}_+^3)}+
  \|\nabla d\|_{L^{q_1}(\mathbb{R}_+^3)}\|\nabla g\|_{L^{q_2}(\mathbb{R}_+^3)}\\
  &+ \|\nabla^2 d\|_{L^{q_1}(\mathbb{R}_+^3)}\|\nabla g\|_{L^{q_2}(\mathbb{R}_+^3)}+\|\nabla d\|_{L^{q_1}(\mathbb{R}_+^3)}\|\nabla^2 g\|_{L^{q_2}(\mathbb{R}_+^3)}\big)
\end{align*}
holds for any $t>0$, where $1\leq q<\infty$, $q\leq q_1,q_2\leq\infty$, $\frac1q=\frac1{q_1}+\frac1{q_2}$.
\end{Lemma}
\begin{proof}
We closely follow the proof of Lemma 2.2 in \cite{Liu-2019}. From $(\ref{N})$, $(\ref{3.49*})$ and Lemma $\ref{pro3.1}$, one has for any $1\leq k\leq3$ and $1\leq q\leq\infty$ that
\begin{align}\label{3.47}
&\big\|\sum_{i,j=1}^3x_3^{-\theta}\partial_k\mathcal{N}\partial_i\partial_j\big(\langle\partial_id,\partial_jg\rangle\big)(t)\big\|_{L^q(\mathbb{R}_+^3)}\nonumber\\
\leq&\big\|\sum_{i,j=1}^3x_3^{-\theta}\partial_k\int_0^1G_\tau*\Big(\partial_i\partial_j\big(\langle\partial_id,\partial_jg\rangle\big)\Big)_*(t)\,{d}\tau\big\|_{L^q(\mathbb{R}^2\times(0,1))}\nonumber\\
&+\big\|\sum_{i,j=1}^3x_3^{-\theta}\partial_k\int_1^\infty G_\tau*\Big(\partial_i\partial_j\big(\langle\partial_id,\partial_jg\rangle\big)\Big)_*(t)\,{d}\tau\big\|_{L^q(\mathbb{R}^2\times(0,1))}\nonumber\\
&+\big\|\sum_{i,j=1}^3\partial_k\mathcal{N}\partial_i\partial_j\big(\langle\partial_id,\partial_jg\rangle\big)(t)\big\|_{L^q(\mathbb{R}_+^3)}\nonumber\\
\leq&C\sup_{y\in\mathbb{R}^3}\big\|\int_0^1x_3^{-\theta}\partial_kG_\tau(x-y)\,d\tau\big\|_{L^1(\mathbb{R}^{2}\times(0,1))}\big\|\sum_{i,j=1}^3\partial_i\partial_j(\langle\partial_id,\partial_jg\rangle)\big\|_{L^q(\mathbb{R}_+^3)}\nonumber\\
&+C\sup_{y\in\mathbb{R}^3}\sum_{i,j=1}^3\big\|\int_1^\infty x_3^{-\theta}\partial_k\partial_i\partial_jG_\tau(x-y)\,d\tau\big\|_{L^1(\mathbb{R}^{2}\times(0,1))}\big\||\nabla d||\nabla g|\big\|_{L^q(\mathbb{R}_+^3)}\nonumber\\
&+\big\|\sum_{i,j=1}^3\partial_k\mathcal{N}\partial_i\partial_j(\langle\partial_id,\partial_jg\rangle)(t)\big\|_{L^q(\mathbb{R}_+^3)}\nonumber\\
\leq&C\big(\|\nabla ^3d\|_{L^{q_1}(\mathbb{R}_+^3)}\|\nabla g\|_{L^{q_2}(\mathbb{R}_+^3)}
+\|\nabla ^2d\|_{L^{q_1}(\mathbb{R}_+^3)}\|\nabla^2 g\|_{L^{q_2}(\mathbb{R}_+^3)}\nonumber\\
&+\|\nabla d\|_{L^{q_1}(\mathbb{R}_+^3)}\|\nabla ^3g\|_{L^{q_2}(\mathbb{R}_+^3)}+
  \|\nabla d\|_{L^{q_1}(\mathbb{R}_+^3)}\|\nabla g\|_{L^{q_2}(\mathbb{R}_+^3)}\big),
\end{align}
where we use the fact that (see \cite[p.1758]{Han-2012}) $$\|\int_0^1x_3^{-\theta}\partial_kG_\tau(x-y)\,d\tau\|_{L^1(\mathbb{R}^{2}\times(0,1))}\leq C_1,$$and
$$\|\int_1^\infty x_3^{-\theta}\partial_k\partial_i\partial_jG_\tau(x-y)\,d\tau\|_{L^1(\mathbb{R}^{2}\times(0,1))}\leq C_2.$$
By one-dimensional Hardy inequality (see \cite{Hardy-1952}), we get for $0<\theta <1$ that
\begin{align}\label{3.48}
&\|x_3^{-\theta} \nabla \cdot(\nabla d\odot\nabla g)\|_{L^q(\mathbb{R}_+^3)}^q\nonumber\\
\leq& \|x_3^{-\theta} \nabla \cdot(\nabla d\odot\nabla g)\|_{L^q(\mathbb{R}^2\times(0,1))}^q+\|x_3^{-\theta} \nabla \cdot(\nabla d\odot\nabla g)\|_{L^q(\mathbb{R}^2\times[1,\infty))}^q\nonumber\\
\leq&\int_{\mathbb{R}^{2}}\int_0^1x_3^{q-q\theta}\frac{|(\nabla \cdot(\nabla d\odot\nabla g))(x^{\prime},x_3)|^q}{x_3^q}\,dx_3dx^{\prime}+\|\nabla \cdot(\nabla d\odot\nabla g)\|_{L^q(\mathbb{R}^{2}\times[1,+\infty))}^q\nonumber\\
\leq&\int_{\mathbb{R}^{2}}\int_0^\infty \frac{|(\nabla \cdot(\nabla d\odot\nabla g))(x^{\prime},x_3)|^q}{x_3^q}\,dx_3dx^{\prime}+\|\nabla \cdot(\nabla d\odot\nabla g)\|_{L^q(\mathbb{R}^{3}_+)}^q\nonumber\\
\leq&C\int_{\mathbb{R}^{2}}\int_0^\infty {|\partial_3(\nabla\cdot(\nabla d\odot\nabla g))(x^{\prime},x_3)|^q}\,dx_3dx^{\prime}+\|\nabla\cdot(\nabla d\odot\nabla g)\|_{L^q(\mathbb{R}^{3}_+)}^q\nonumber\\
\leq&C\big(\|\nabla ^3d\|_{L^{q_1}(\mathbb{R}_+^3)}\|\nabla g\|_{L^{q_2}(\mathbb{R}_+^3)}
+\|\nabla ^2d\|_{L^{q_1}(\mathbb{R}_+^3)}\|\nabla^2 g\|_{L^{q_2}(\mathbb{R}_+^3)}\nonumber\\
&+\|\nabla d\|_{L^{q_1}(\mathbb{R}_+^3)}\|\nabla ^3g\|_{L^{q_2}(\mathbb{R}_+^3)}+
  \|\nabla^2 d\|_{L^{q_1}(\mathbb{R}_+^3)}\|\nabla g\|_{L^{q_2}(\mathbb{R}_+^3)}+\|\nabla d\|_{L^{q_1}(\mathbb{R}_+^3)}\|\nabla^2 g\|_{L^{q_2}(\mathbb{R}_+^3)}\big),
\end{align}
where $1\leq q<\infty$, $q\leq q_1,q_2\leq\infty$, $\frac1q=\frac1{q_1}+\frac1{q_2}$.
Hence, Lemma $\ref{le3.10}$ directly follows from $(\ref{3.47})$, $(\ref{3.48})$ and the decomposition
$$\mathbb{P} \nabla \cdot(\nabla d\odot\nabla g)=\nabla \cdot(\nabla d\odot\nabla g)+\sum_{i,j=1}^3\nabla\mathcal{N}\partial_i\partial_j(\langle\partial_id,\partial_jg\rangle).$$
\end{proof}
\section{Proof of Theorem $\ref{th1.2}$ }\label{sec3}
In this section, we will provide the proof for Theorem $\ref{th1.2}$. We first recall a lemma on the estimates for the $L^3$-norm of $(u, \nabla d)$, which will play an important role in the sequel.
\begin{Lemma}\label{le2.5*}(\cite{huang-wang-wen-2019})
There exists $\varepsilon_0>0$ such that if
$$\left\|u_0\right\|_{L^3(\mathbb{R}_+^3)}^3+\left\|\nabla d_0\right\|_{L^3(\mathbb{R}_+^3)}^3\leq\varepsilon_0^3,$$
then for any $0<t<\infty$,
$$\left\|u(t)\right\|_{L^3(\mathbb{R}_+^3)}^3+\left\|\nabla d(t)\right\|_{L^3(\mathbb{R}_+^3)}^3\leq\left\|u_0\right\|_{L^3(\mathbb{R}_+^3)}^3+\left\|\nabla d_0\right\|_{L^3(\mathbb{R}_+^3)}^3.$$
\end{Lemma}
Then we have the following lemma.
\begin{Lemma}\label{le3.1}
Assume there exist $C_1,C_2>0$ such that if the global solution $(u,d)$ given by Proposition $\ref{solu1}$ satisfies
\begin{align}
&\|\nabla^2 d(t)\|_{L^\infty(\mathbb{R}_+^3)}\leq 2C_1t^{-\frac52},\label{3.1}\\
&\|\nabla u(t)\|_{L^\infty(\mathbb{R}_+^3)}\leq 2C_2t^{-2}\label{3.2**}
\end{align}
for any $t\in(0,\infty)$. Then the following estimates
\begin{align}
	&\|\nabla u(t)\|_{L^p(\mathbb{R}_+^3)}\leq Ct^{-\frac12-\frac32(1-\frac1p)},\label{3.3*}\\
		&\|\nabla^2 d(t)\|_{L^r(\mathbb{R}_+^3)}\leq Ct^{-1-\frac32(1-\frac1r)},\label{3.4*}\\
   & \|\nabla^2 d(t)\|_{L^\infty(\mathbb{R}_+^3)}\leq \frac32C_1t^{-\frac52},\label{3.2}\\
    & \|\nabla u(t)\|_{L^\infty(\mathbb{R}_+^3)}\leq \frac32C_2t^{-2}\label{3.2*}
\end{align}
hold for any $t\in(0,\infty)$, where $p\in(1,\infty)$ and $r\in[1,\infty)$.

\end{Lemma}
\begin{proof}

\textbf{Step 1:}\ Proof of ($\ref{3.2}$).

From $(\ref{2.9})$, $(\ref{3.1})$ and the interpolation inequality, we obtain that for any $2< r<\infty$,
$$\|\nabla^2 d(t)\|_{L^r(\mathbb{R}_+^3)}\leq\|\nabla^2 d(t)\|_{L^2(\mathbb{R}_+^3)}^{\frac2r}\|\nabla^2 d(t)\|_{L^\infty(\mathbb{R}_+^3)}^{\frac{r-2}{r}}\leq Ct^{-1-\frac32(1-\frac1r)}.$$
Similarly, by $(\ref{2.8})$ and $(\ref{3.2**})$, we obtain that for any $2<p<\infty$,
$$\|\nabla u(t)\|_{L^p(\mathbb{R}_+^3)}\leq Ct^{-\frac12-\frac32(1-\frac1p)}.$$
Applying $\nabla$ to $(\ref{1.1-main})_2$, it gives
$$(\nabla d)_t-\Delta(\nabla d)=-\nabla(u\cdot\nabla d)+\nabla(|\nabla d|^2d),$$
in terms of $\text{Duhamel}$'s principle, we have
\begin{equation}\label{3.7*}
\nabla d(t)=e^{\frac{t}{2}\Delta}\nabla d(\frac{t}{2})+\int_{\frac{t}{2}}^{t}e^{(t-s)\Delta}\Big(-\nabla(u\cdot\nabla d)+\nabla(|\nabla d|^2d)\Big)(s)\,{d}s,
\end{equation}
and
\begin{equation}\label{3.7'}
\nabla^2d(t)=\nabla e^{\frac{t}{2}\Delta}\nabla d(\frac{t}{2})+\int_{\frac{t}{2}}^{t}\nabla e^{(t-s)\Delta}\Big(-\nabla(u\cdot\nabla d)+\nabla(|\nabla d|^2d)\Big)(s)\,{d}s.
\end{equation}
This, together with $(\ref{2.2})$ and $(\ref{2.7})$, gives rise to
\begin{align*}
&\|\nabla^2d(t)\|_{L^\infty(\mathbb{R}_+^3)}\\
\leq&Ct^{-\frac54}\|\nabla d(t)\|_{L^2(\mathbb{R}_+^3)}+\int_{\frac{t}{2}}^t(t-s)^{-\frac{1}{2}-\frac{3}{2q}}\big\|\nabla(u\cdot\nabla d)-\nabla(|\nabla d|^2d)\big\|_{L^q(\mathbb{R}_+^3)}\,{d}s\\
\leq &C_1t^{-\frac52}+\int_{\frac t2}^t(t-s)^{-\frac{1}{2}-\frac{3}{2q}}\big\||\nabla u||\nabla d|+|u||\nabla^2d|+|\nabla d||\nabla^2d|+|\nabla d|^3\big\|_{L^q(\mathbb{R}_+^3)}\,{d}s\\
\leq& C_1t^{-\frac52}+K_1+K_2+K_3+K_4,
\end{align*}
where we choose $q\in(3,\infty)$ so that $-\frac{1}{2}-\frac{3}{2q}>-1$. We first estimate $K_1$ by using $(\ref{2.7})$, $(\ref{3.3*})$, the interpolation inequality, Lemma $\ref{le2.2}$ and Lemma $\ref{le2.5*}$,
\begin{align*}
K_1&=\int_{\frac t2}^t(t-s)^{-\frac{1}{2}-\frac{3}{2q}}\big\||\nabla u||\nabla d|\big\|_{L^q(\mathbb{R}_+^3)}\,{d}s\\
&\leq C\int_{\frac t2}^t(t-s)^{-\frac{1}{2}-\frac{3}{2q}}\|\nabla u(s)\|_{L^{2q}(\mathbb{R}_+^3)}\|\nabla d(s)\|_{L^3(\mathbb{R}_+^3)}^{\frac{3}{2q}}\|\nabla d(s)\|_{L^{\infty}(\mathbb{R}_+^3)}^{\frac{2q-3}{2q}}\,{d}s\\
&\leq C\varepsilon_0^{\frac{3}{2q}}\int_{\frac t2}^t(t-s)^{-\frac{1}{2}-\frac{3}{2q}}s^{-4+\frac{15}{4q}}\,{d}s\\
&\leq C \varepsilon_0^{\frac{3}{2q}} t^{-\frac52}.
\end{align*}
Similarly, using $(\ref{3.3*})$, $(\ref{3.4*})$, the interpolation inequality, Lemma $\ref{le2.2}$, Lemma $\ref{le2.3}$ and Lemma $\ref{le2.5*}$, we can establish the estimates as follows:
\begin{align*}
K_2&=\int_{\frac t2}^t(t-s)^{-\frac{1}{2}-\frac{3}{2q}}\big\|| u||\nabla^2 d|\big\|_{L^q(\mathbb{R}_+^3)}\,{d}s\\
&\leq C\int_{\frac t2}^t(t-s)^{-\frac{1}{2}-\frac{3}{2q}}  \| u(s)\|_{L^{3}(\mathbb{R}_+^3)}^{\frac{3}{2q}}  \| u(s)\|_{L^\infty(\mathbb{R}_+^3)}^{\frac{2q-3}{2q}}\|\nabla^2 d(s)\|_{L^{2q}(\mathbb{R}_+^3)}\,{d}s\\
&\leq C\varepsilon_0^{\frac{3}{2q}}\int_{\frac t2}^t(t-s)^{-\frac{1}{2}-\frac{3}{2q}}s^{-4+\frac{3}{q}}\,{d}s\\
&\leq C \varepsilon_0^{\frac{3}{2q}} t^{-\frac52},\\
K_3&=\int_{\frac t2}^t(t-s)^{-\frac{1}{2}-\frac{3}{2q}}\big\||\nabla d||\nabla^2 d|\big\|_{L^q(\mathbb{R}_+^3)}\,{d}s\\
&\leq C\int_{\frac t2}^t(t-s)^{-\frac{1}{2}-\frac{3}{2q}}\|\nabla^2 d(s)\|_{L^{2q}(\mathbb{R}_+^3)}\|\nabla d(s)\|_{L^3(\mathbb{R}_+^3)}^{\frac{3}{2q}}\|\nabla d(s)\|_{L^{\infty}(\mathbb{R}_+^3)}^{\frac{2q-3}{2q}}\,{d}s\\
&\leq C\varepsilon_0^{\frac{3}{2q}}\int_{\frac t2}^t(t-s)^{-\frac{1}{2}-\frac{3}{2q}}s^{-\frac92+\frac{15}{4q}}\,{d}s\\
&\leq C \varepsilon_0^{\frac{3}{2q}} t^{-\frac52},\\
K_4&=\int_{\frac t2}^t(t-s)^{-\frac{1}{2}-\frac{3}{2q}}\big\|\nabla d(s)\big\|_{L^{3q}(\mathbb{R}_+^3)}^3\,{d}s\\
&\leq \int_{\frac t2}^t(t-s)^{-\frac{1}{2}-\frac{3}{2q}}\|\nabla d(s)\|_{L^{3}(\mathbb{R}_+^3)}^{\frac 3q}
\|\nabla d(s)\|_{L^{\infty}(\mathbb{R}_+^3)}^{3-\frac 3q}
\,{d}s\\
&\leq C\varepsilon_0^{\frac{3}{q}}\int_{\frac t2}^t(t-s)^{-\frac{1}{2}-\frac{3}{2q}}s^{-6+\frac{6}{q}}\,{d}s\\
&\leq C \varepsilon_0^{\frac{3}{q}} t^{-\frac52}.
\end{align*}
Thus $(\ref{3.2})$ follows by choosing $\varepsilon_0$ sufficiently small so that
$$C \varepsilon_0^{\frac{3}{q}} +C \varepsilon_0^{\frac{3}{2q}} \leq \frac{C_1}{8}.$$
\textbf{Step 2:}\ Proof of ($\ref{3.2*}$).

By $(\ref{uB2})_1$, $(\ref{3.3*})$, $(\ref{3.4*})$, the interpolation inequality, Lemma \ref{le2.1} and Lemma \ref{le2.3}, one has for any $t>0$,
\begin{align*}
&\|\nabla u(t)\|_{L^\infty(\mathbb{R}_+^3)}\nonumber\\
\leq&Ct^{-\frac54}\|u(\frac{t}{2})\|_{L^2(\mathbb{R}_+^3)}+C\int_{\frac{t}{2}}^{t}(t-s)^{-\frac12-\frac{3}{2r}}\big\|\big(u\cdot\nabla u+\nabla\cdot(\nabla d\odot\nabla d)\big)(s)\big\|_{L^r(\mathbb{R}_+^3)}\,{d}s\nonumber\\
\leq& C_2 t^{-2}+C\int_{\frac t2}^t(t-s)^{-\frac12-\frac{3}{2r}}\Big(\|u(s)\|_{L^{2r}(\mathbb{R}_+^3)}\|\nabla u(s)\|_{L^{2r}(\mathbb{R}_+^3)}+\|\nabla d(s)\|_{L^{2r}(\mathbb{R}_+^3)}\|\nabla^2d(s)\|_{L^{2r}(\mathbb{R}_+^3)}\Big)\,{d}s\nonumber\\
\leq& C_2 t^{-2}+C\int_{\frac t2}^t(t-s)^{-\frac{1}{2}-\frac{3}{2r}}\|\nabla u(s)\|_{L^{2r}(\mathbb{R}_+^3)}\| u(s)\|_{L^3(\mathbb{R}_+^3)}^{\frac{3}{2r}}\| u(s)\|_{L^{\infty}(\mathbb{R}_+^3)}^{\frac{2r-3}{2r}}\,{d}s\nonumber\\
&+C\int_{\frac t2}^t(t-s)^{-\frac{1}{2}-\frac{3}{2r}}\|\nabla^2 d(s)\|_{L^{2r}(\mathbb{R}_+^3)}\|\nabla d(s)\|_{L^3(\mathbb{R}_+^3)}^{\frac{3}{2r}}\|\nabla d(s)\|_{L^{\infty}(\mathbb{R}_+^3)}^{\frac{2r-3}{2r}}\,{d}s\nonumber\\
\leq &C_2 t^{-2}+C\varepsilon_0^{\frac{3}{2r}}\int_{\frac t2}^t(t-s)^{-\frac{1}{2}-\frac{3}{2r}}s^{-\frac72+\frac{3}{r}}\,{d}s\nonumber\\
\leq&C_2 t^{-2}+C \varepsilon_0^{\frac{3}{2r}} t^{-2},
\end{align*}
here, we choose $r\in(3,\infty)$ so that $-\frac{1}{2}-\frac{3}{2r}>-1$.
Thus $(\ref{3.2*})$ follows by choosing $\varepsilon_0$ sufficiently small so that $C \varepsilon_0^{\frac{3}{2r}} \leq \frac{C_2}{2}.$

\end{proof}
Now we can utilize the continuity argument to obtain the subsequent corollary.
\begin{Corollary}\label{cor3.1}
Assume $u_0, d_0-e_3\in L^1(\mathbb{R}_+^{3})$, if $(u,d)$ is the global strong solution obtained by Proposition $\ref{solu1}$, then
\begin{align*}
&\|\nabla^2 d(t)\|_{L^p(\mathbb{R}_+^3)}\leq Ct^{-1-\frac32(1-\frac1p)},\\
&\|\nabla u(t)\|_{L^r(\mathbb{R}_+^3)}\leq Ct^{-\frac12-\frac32(1-\frac1r)}
\end{align*}
hold for any $t>0$, where $p\in[1,\infty]$, $r\in(1,\infty]$.
\end{Corollary}

In order to obtain the $L^r$-decay rates of the second-order derivatives of $u$, we need the following lemma to overcome the difficulty of integration singularity.
\begin{Lemma}\label{le3.3*}
Assume $u_0, d_0-e_3\in L^1(\mathbb{R}_+^{3})$, if $(u,d)$ is the global strong solution obtained by Proposition $\ref{solu1}$, then it holds for any $t>0$ that
\begin{align}
&\|\mathbb{A}^\alpha u(t+h)-\mathbb{A}^\alpha u(t)\|_{L^r(\mathbb{R}_+^3)}\leq C\Big(h^\delta t^{-\alpha-\delta-\frac 32+\frac 3{2r}}+h^{1-\alpha}t^{-\frac72+\frac 3{2r}}+h^\delta t^{-\alpha-\delta-\frac52+\frac 3{2r}}\Big);\label{3.9}\\
&\|(-\Delta)^\alpha d(t+h)-(-\Delta)^\alpha  d(t)\|_{L^r(\mathbb{R}_+^3)}\leq C\Big(h^\delta t^{-\alpha-\delta-\frac 32+\frac 3{2r}}+h^{1-\alpha}t^{-\frac72+\frac 3{2r}}+h^\delta t^{-\alpha-\delta-\frac52+\frac 3{2r}}\Big);\label{3.12}\\
&\|(-\Delta)^\alpha \nabla d(t+h)-(-\Delta)^\alpha  \nabla d(t)\|_{L^r(\mathbb{R}_+^3)}\leq C\Big(h^\delta t^{-\alpha-\delta-2+\frac 3{2r}}+h^{1-\alpha}t^{-4+\frac 3{2r}}+h^\delta t^{-\alpha-\delta-3+\frac 3{2r}}\Big),\label{3.14}
\end{align}
where $1<r<\infty$, $h>0$, $0<\alpha<1$ and $0<\delta<1-\alpha$.
\end{Lemma}
\begin{proof}

\textbf{Step 1:}\ Proof of ($\ref{3.9}$).

It is easy to verify that the $u$ equation of system $(\ref{1.1-main})$ can be transformed into the integral form as for any $0\leq\tau<t$,
$$u(t)=e^{-(t-\tau)\mathbb{A}}u(\tau)-\int_{\tau}^{t}e^{-(t-s)\mathbb{A}}\mathbb{P}\Big(u(s)\cdot\nabla u(s)+\nabla\cdot\big(\nabla d(s)\odot\nabla d(s)\big)\Big)\,{d}s.$$
Then, for any $t>0$, $h>0$,
 \begin{align}\label{3.4}
 u(t+h)&=e^{-(t+h-\frac t2)\mathbb{A}}u(\frac t2)-\int_{\frac t2}^{t+h}e^{-(t+h-s)\mathbb{A}}\mathbb{P}\Big(u(s)\cdot\nabla u(s)+\nabla\cdot\big(\nabla d(s)\odot\nabla d(s)\big)\Big)\,{d}s\nonumber\\
 &=e^{-h\mathbb{A}}e^{-\frac t2\mathbb{A}}u(\frac t2)-\Big(\int_t^{t+h}+\int_{\frac t2}^t\Big)e^{-(t-s)\mathbb{A}}e^{-h\mathbb{A}}\mathbb{P}\Big(u(s)\cdot\nabla u(s)+\nabla\cdot\big(\nabla d(s)\odot\nabla d(s)\big)\Big)\,{d}s.
 \end{align}
 For any $\varphi\in{D}(A^\delta)$ and $1<q<\infty$, we have
  \begin{equation}\label{3.5}
\|(e^{-h\mathbb{A}}-I)\varphi\|_{L^q(\mathbb{R}_+^3)}
=h\big\|\int_0^1\mathbb{A}^{1-\delta}e^{-sh\mathbb{A}}\mathbb{A}^\delta\varphi \,ds\big\|_{L^q(\mathbb{R}_+^3)}
\leq\frac C\delta h^\delta\|\mathbb{A}^\delta\varphi\|_{L^q(\mathbb{R}_+^3)},
\end{equation}
where we use the fact that $\mathbb{A}^\alpha e^{-t\mathbb{A}}=\int_0^\infty\lambda^\alpha e^{-t\lambda}dE_\lambda$ is a bounded operator with operator norm $\|\mathbb{A}^\alpha e^{-t\mathbb{A}}\|\leq t^{-\alpha}$ (see \cite[p.206]{Sohr-2001}). It is worth mentioning that $(\ref{3.5})$ still holds if we replace the operator $\mathbb{A}$ by $-\Delta$.

Whence, from $(\ref{3.4})$, $(\ref{3.5})$, Lemma $\ref{le2.2}$, Lemma $\ref{le2.3}$ and Corollary $\ref{cor3.1}$, we deduce for $1<r<\infty$ and any $t>0$, $h>0$,
\begin{align*}
&\big\|\mathbb{A}^\alpha u(t+h)-\mathbb{A}^\alpha u(t)\big\|_{L^r(\mathbb{R}_+^3)}\\
\leq&\big\|\mathbb{A}^\alpha(e^{-h\mathbb{A}}-I)e^{-\frac t2\mathbb{A}}u(\frac t2)\big\|_{L^r(\mathbb{R}_+^3)}\\
  &+\int_t^{t+h}\big\|\mathbb{A}^\alpha e^{-(t+h-s)\mathbb{A}}\mathbb{P}\big(u\cdot\nabla u+\nabla\cdot(\nabla d\odot\nabla d)\big)(s)\big\|_{L^r(\mathbb{R}_+^3)}\,ds\\
  &+\int_{\frac t2}^t\big\|(e^{-h\mathbb{A}}-I)\mathbb{A}^\alpha e^{-(t-s)\mathbb{A}}\mathbb{P}\big(u\cdot\nabla u+\nabla\cdot(\nabla d\odot\nabla d)\big)(s)\big\|_{L^r(\mathbb{R}_+^3)}\,ds\\
\leq&Ch^\delta\big\|\mathbb{A}^{\alpha+\delta}e^{-\frac t2\mathbb{A}}u\big(\frac t2)\big\|_{L^r(\mathbb{R}_+^3)}\\
   &+C\int_t^{t+h}(t+h-s)^{-\alpha}\Big(\|u(s)\|_{L^{r_1}(\mathbb{R}_+^3)}\|\nabla u(s)\|_{L^{r_2}(\mathbb{R}_+^3)}+\|\nabla d(s)\|_{L^{r_1}(\mathbb{R}_+^3)}\|\nabla^2 d(s)\|_{L^{r_2}(\mathbb{R}_+^3)}\Big)\,ds\\
      &+Ch^\delta\int_{\frac t2}^t(t-s)^{-\alpha-\delta}\Big(\|u(s)\|_{L^{r_1}(\mathbb{R}_+^3)}\|\nabla u(s)\|_{L^{r_2}(\mathbb{R}_+^3)}+\|\nabla d(s)\|_{L^{r_1}(\mathbb{R}_+^3)}\|\nabla ^2d(s)\|_{L^{r_2}(\mathbb{R}_+^3)}\Big)\,ds\\
\leq &Ch^\delta t^{-\alpha-\delta}\|u(\frac t2)\|_{L^r(\mathbb{R}_+^3)}\\
   &+C\int_t^{t+h}(t+h-s)^{-\alpha}\Big(s^{-\frac 32(1-\frac1{r_1})-\frac12-\frac 32(1-\frac1{r_2})}+s^{-\frac12-\frac 32(1-\frac1{r_1})-1-\frac 32(1-\frac1{r_2})}\Big)\,ds\\
&+Ch^\delta\int_{\frac t2}^t(t-s)^{-\alpha-\delta}\Big(s^{-\frac 32(1-\frac1{r_1})-\frac12-\frac 32(1-\frac1{r_2})}+s^{-\frac12-\frac 32(1-\frac1{r_1})-1-\frac 32(1-\frac1{r_2})}\Big)\,ds\\
\leq&C\Big(h^\delta t^{-\alpha-\delta-\frac 32+\frac 3{2r}}+h^{1-\alpha}t^{-\frac72+\frac 3{2r}}+h^\delta t^{-\alpha-\delta-\frac52+\frac 3{2r}}\Big),
\end{align*}
where $\frac1r=\frac1{r_1}+\frac1{r_2}$, ${r_1}, {r_2}\in(1,\infty)$.

\textbf{Step 2:}\ Proof of ($\ref{3.12}$).

Applying a similar argument as in the proof of $(\ref{3.9})$, we are able to obtain, after a straightforward computation, for any $t>0$, $h>0$ that
\begin{equation}\label{3.7}
(d-e_3)(t+h)=e^{{h}\Delta}e^{\frac{t}{2}\Delta}(d-e_3)(\frac{t}{2})-\Big(\int_t^{t+h}+\int_{\frac t2}^t\Big)e^{(t-s)\Delta}e^{{h}\Delta}\Big(u(s)\cdot\nabla d(s)-\left|\nabla d(s)\right|^2d(s)\Big)\,{d}s.
\end{equation}
Therefore, from $(\ref{3.5})$, $(\ref{3.7})$, Lemma $\ref{le2.2}$, Lemma $\ref{le2.3}$ and Corollary $\ref{cor3.1}$, we conclude that for all $1<r<\infty$ and any $t>0$, $h>0$,
\begin{align*}
&\big\|(-\Delta)^\alpha  d(t+h)-(-\Delta)^\alpha  d(t)\big\|_{L^r(\mathbb{R}_+^3)}\\
\leq&\big\|(-\Delta)^\alpha(e^{h\Delta}-I)e^{\frac t2\Delta} (d-e_3)(\frac t2)\big\|_{L^r(\mathbb{R}_+^3)}\\
&+\int_t^{t+h}\big\|(-\Delta)^\alpha e^{(t+h-s)\Delta}\big(u\cdot\nabla d+|\nabla d|^2d\big)(s)\big\|_{L^r(\mathbb{R}_+^3)}\,ds\\
&+\int_{\frac t2}^t\big\|(e^{h\Delta}-I)(-\Delta)^\alpha e^{(t-s)\Delta}\big(u\cdot\nabla d+|\nabla d|^2d\big)(s)\big\|_{L^r(\mathbb{R}_+^3)}\,ds\\
\leq&Ch^\delta\big\|(-\Delta)^{\alpha+\delta}e^{\frac t2\Delta} (d-e_3)\big(\frac t2\big)\big\|_{L^r(\mathbb{R}_+^3)}\\
  &+C\int_t^{t+h}(t+h-s)^{-\alpha}\Big(\|u(s)\|_{L^{r_1}(\mathbb{R}_+^3)}\|\nabla d(s)\|_{L^{r_2}(\mathbb{R}_+^3)}+\| \nabla d(s)\|_{L^{2r}(\mathbb{R}_+^3)}^2\Big)\,ds\\
  &+Ch^\delta\int_{\frac t2}^t(t-s)^{-\alpha-\delta}\Big(\|u(s)\|_{L^{r_1}(\mathbb{R}_+^3)}\|\nabla d(s)\|_{L^{r_2}(\mathbb{R}_+^3)}+\| \nabla d(s)\|_{L^{2r}(\mathbb{R}_+^3)}^2\Big)\,ds\\
\leq &Ch^\delta t^{-\alpha-\delta}\| (d-e_3)(\frac t2)\|_{L^r(\mathbb{R}_+^3)}\\
    &+C\int_t^{t+h}(t+h-s)^{-\alpha}s^{-\frac12-3+\frac3{2r}}\,ds+Ch^\delta\int_{\frac t2}^t(t-s)^{-\alpha-\delta}s^{-\frac12-3+\frac3{2r}}\,ds\\
\leq&C\Big(h^\delta t^{-\alpha-\delta-\frac 32+\frac 3{2r}}+h^{1-\alpha}t^{-\frac72+\frac 3{2r}}+h^\delta t^{-\alpha-\delta-\frac52+\frac 3{2r}}\Big),
\end{align*}
where $\frac1r=\frac1{r_1}+\frac1{r_2}$, ${r_1}, {r_2}\in(1,\infty)$.

\textbf{Step 3:}\ Proof of ($\ref{3.14}$).

In the same way, it follows from $(\ref{3.7*})$ that
\begin{equation}\label{3.15}
\nabla d(t+h)=e^{{h}\Delta}e^{\frac{t}{2}\Delta}\nabla d(\frac{t}{2})+\Big(\int_t^{t+h}+\int_{\frac t2}^t\Big)e^{(t-s)\Delta}e^{{h}\Delta}\Big(-\nabla(u\cdot\nabla d)+\nabla(|\nabla d|^2d)\Big)(s)\,{d}s.
\end{equation}
Therefore, from $(\ref{3.5})$, $(\ref{3.15})$, Lemma $\ref{le2.2}$, Lemma $\ref{le2.3}$ and Corollary $\ref{cor3.1}$, we conclude that for all $1<r<\infty$ and any $t>0$, $h>0$,
\begin{align*}
&\big\|(-\Delta)^\alpha \nabla d(t+h)-(-\Delta)^\alpha \nabla d(t)\big\|_{L^r(\mathbb{R}_+^3)}\\
\leq&\big\|(-\Delta)^\alpha(e^{h\Delta}-I)e^{\frac t2\Delta} \nabla d(\frac t2)\big\|_{L^r(\mathbb{R}_+^3)}\\
  &+\int_t^{t+h}\big\|(-\Delta)^\alpha e^{(t+h-s)\Delta}\big(\nabla(u\cdot\nabla d)+\nabla(|\nabla d|^2d)\big)(s)\big\|_{L^r(\mathbb{R}_+^3)}\,ds\\
     &+\int_{\frac t2}^t\big\|(e^{h\Delta}-I)(-\Delta)^\alpha e^{(t-s)\Delta}\big(\nabla(u\cdot\nabla d)+\nabla(|\nabla d|^2d)\big)(s)\big\|_{L^r(\mathbb{R}_+^3)}\,ds\\
\leq&Ch^\delta\big\|(-\Delta)^{\alpha+\delta}e^{\frac t2\Delta} \nabla d(\frac t2)\big\|_{L^r(\mathbb{R}_+^3)}\\
   &+C\int_t^{t+h}(t+h-s)^{-\alpha}\Big(\|\nabla u(s)\|_{L^{r_1}(\mathbb{R}_+^3)}\|\nabla d(s)\|_{L^{r_2}(\mathbb{R}_+^3)}
+\|u(s)\|_{L^{r_1}(\mathbb{R}_+^3)}\|\nabla^2 d(s)\|_{L^{r_2}(\mathbb{R}_+^3)}\\
   &+\|\nabla d(s)\|_{L^{r_1}(\mathbb{R}_+^3)}\|\nabla^2 d(s)\|_{L^{r_2}(\mathbb{R}_+^3)}
+\| \nabla d(s)\|_{L^{3r}(\mathbb{R}_+^3)}^3\Big)\,ds\\
&+Ch^\delta\int_{\frac t2}^t(t-s)^{-\alpha-\delta}\Big(\|\nabla u(s)\|_{L^{r_1}(\mathbb{R}_+^3)}\|\nabla d(s)\|_{L^{r_2}(\mathbb{R}_+^3)}
+\|u(s)\|_{L^{r_1}(\mathbb{R}_+^3)}\|\nabla^2 d(s)\|_{L^{r_2}(\mathbb{R}_+^3)}\\
&+\|\nabla d(s)\|_{L^{r_1}(\mathbb{R}_+^3)}\|\nabla^2 d(s)\|_{L^{r_2}(\mathbb{R}_+^3)}
+\| \nabla d(s)\|_{L^{3r}(\mathbb{R}_+^3)}^3\Big)\,ds\\
\leq &Ch^\delta t^{-\alpha-\delta}\| \nabla d(\frac t2)\|_{L^r(\mathbb{R}_+^3)}\\
  &+C\int_t^{t+h}(t+h-s)^{-\alpha}s^{-\frac12-\frac12-3+\frac3{2r}}\,ds+Ch^\delta\int_{\frac t2}^t(t-s)^{-\alpha-\delta}s^{-\frac12-\frac12-3+\frac3{2r}}\,ds\\
\leq&C\Big(h^\delta t^{-2-\alpha-\delta+\frac 3{2r}}+h^{1-\alpha}t^{-4+\frac 3{2r}}+h^\delta t^{-\alpha-\delta-3+\frac 3{2r}}\Big),
\end{align*}
where $\frac1r=\frac1{r_1}+\frac1{r_2}$, ${r_1}, {r_2}\in(1,\infty)$.
\end{proof}
Base on Lemma \ref{le3.3*}, we can deduce the following decay rates.
\begin{Lemma}\label{le3.4}
Assume $u_0, d_0-e_3\in L^1(\mathbb{R}_+^{3})$, if $(u,d,p)$ is the global strong solution obtained by Proposition $\ref{solu1}$, then it holds that for any $t>0$ and $1<r<\infty$,
$$\|\nabla^{2}u(t)\|_{{L^{r}(\mathbb{R}_{+}^{3})}}+\|\partial_{t}u(t)\|_{{L^{r}(\mathbb{R}_{+}^{3})}}+\|\nabla p(t)\|_{{L^{r}(\mathbb{R}_{+}^{3})}}+\|\partial_{t}d(t)\|_{{L^{r}(\mathbb{R}_{+}^{3})}}\leq Ct^{-1-\frac{3}{2}(1-\frac1r)}.$$
\end{Lemma}
\begin{proof}
By a direct calculation, we obtain for $1<r<\infty$ and $t>0$ that
\begin{align}\label{3.16}
\mathbb{A}u(t)=&\mathbb{A}e^{-\frac{3t}4\mathbb{A}}u{(\frac t4)}-(I-e^{-\frac t2\mathbb{A}})\mathbb{P}\Big(u\cdot\nabla u+\nabla\cdot(\nabla d\odot\nabla d)\Big)(t)\nonumber\\
&-\int_{\frac t4}^{\frac t2}\mathbb{A}e^{-(t-s)\mathbb{A}}\mathbb{P}\Big(u\cdot\nabla u+\nabla\cdot(\nabla d\odot\nabla d)\Big)(s)\,ds\nonumber\\
&-\int_{\frac t2}^t\mathbb{A}e^{-(t-s)\mathbb{A}}\Big(\mathbb{P}(u\cdot\nabla u)(s)-\mathbb{P}(u\cdot\nabla u)(t)\Big)\,ds\nonumber\\
&-\int_{\frac t2}^t\mathbb{A}e^{-(t-s)\mathbb{A}}\Big(\mathbb{P}\nabla\cdot(\nabla d\odot\nabla d)(s)-\mathbb{P}\nabla\cdot(\nabla d\odot\nabla d)(t)\Big)\,ds\nonumber\\
=&J_1(t)+J_2(t)+J_3(t)+J_4(t)+J_5(t).
\end{align}
According to Lemma $\ref{le2.3}$ and Corollary $\ref{cor3.1}$, we have for any $1<r<\infty$ and $t>0$,
\begin{align}
\|J_1(t)\|_{L^r(\mathbb{R}_+^3)}&\leq Ct^{-1}\|u{(\frac t4)}\|_{L^r(\mathbb{R}_+^3)}\leq Ct^{-1-\frac 32(1-\frac1r)};\label{3.17}\\
\|J_2(t)\|_{L^r(\mathbb{R}_+^3)}&\leq2\|\mathbb{P}(u\cdot\nabla u)(t)+\mathbb{P}\nabla\cdot(\nabla d\odot\nabla d)(t)\|_{L^r(\mathbb{R}_+^3)}\nonumber\\
&\leq C\|u(t)\|_{L^{2r}(\mathbb{R}_+^3)}\|\nabla u(t)\|_{L^{2r}(\mathbb{R}_+^3)}+C\|\nabla d(t)\|_{L^{2r}(\mathbb{R}_+^3)}\|\nabla ^2d(t)\|_{L^{2r}(\mathbb{R}_+^3)}\nonumber\\
&\leq C t^{-2-\frac 32(1-\frac1r)};\label{3.18}\\
\|J_3(t)\|_{L^r(\mathbb{R}_+^3)}&\leq C\int_{\frac t4}^{\frac t2}(t-s)^{-1}\Big(\|\mathbb{P}(u\cdot\nabla u)(s)\|_{L^r(\mathbb{R}_+^3)}+\|\mathbb{P}(\nabla\cdot(\nabla d\odot\nabla d) )(s)\|_{L^r(\mathbb{R}_+^3)}\Big)\,ds\nonumber\\
&\leq C\int_{\frac t4}^{\frac t2}(t-s)^{-1}\Big(\|u(s)\|_{L^{2r}(\mathbb{R}_+^3)}\|\nabla u(s)\|_{L^{2r}(\mathbb{R}_+^3)}+\|\nabla d(s)\|_{L^{2r}(\mathbb{R}_+^3)}\|\nabla ^2d(s)\|_{L^{2r}(\mathbb{R}_+^3)}\Big)\,ds\nonumber\\
&\leq C\int_{\frac t4}^{\frac t2}(t-s)^{-1}\Big(s^{-\frac12-3+\frac 3{2r}}+s^{-\frac32-3+\frac 3{2r}}\Big)\,ds\nonumber\\
&\leq   Ct^{-2-\frac 32(1-\frac1r)}.\label{3.19}
\end{align}

By Lemma $\ref{le3.3*}$ and calculations similar to \cite[p.3950]{Han-2011}, we have
\begin{equation}\label{3.20}
\|J_4(t)\|_{L^r(\mathbb{R}_+^3)}\leq Ct^{-2-\frac 32(1-\frac1r)}.
\end{equation}
Note that $\|\nabla d(t)\|_{L^r(\mathbb{R}_+^3)}\approx \|(-\Delta)^{\frac12}d(t)\|_{L^r(\mathbb{R}_+^3)}$, using Lemma $\ref{le2.2}$, Lemma $\ref{le2.3}$, Lemma $\ref{le3.3*}$ and Corollary $\ref{cor3.1}$, we obtain the upper bound of the time decay rate for $J_5(t)$ as
\begin{align}\label{3.21}
\|J_5(t)\|_{L^r(\mathbb{R}_+^3)}\leq & C\int_{\frac t2}^t(t-s)^{-1}\big\|\mathbb{P}\nabla\cdot(\nabla d\odot\nabla d)(s)-\mathbb{P}\nabla\cdot(\nabla d\odot\nabla d)(t)\big\|_{L^r(\mathbb{R}_+^3)}\,ds\nonumber\\
 \leq &C\int_{\frac t2}^t(t-s)^{-1}\big\|\mathbb{P}\nabla\cdot\Big(\nabla d(s)\odot\big(\nabla d(s)-\nabla d(t)\big)\Big)\big\|_{L^r(\mathbb{R}_+^3)}\,ds\nonumber\\
 &+C\int_{\frac t2}^t(t-s)^{-1}\big\|\mathbb{P}\nabla\cdot\Big(\big(\nabla d(s)-\nabla d(t)\big)\odot \nabla d(t)\Big)\big\|_{L^r(\mathbb{R}_+^3)}\,ds\nonumber\\
 \leq & C\int_{\frac t2}^t(t-s)^{-1}\|\nabla^2 d(s)\|_{L^{2r}(\mathbb{R}_+^3)}\|\nabla d(s)-\nabla d(t)\|_{L^{2r}(\mathbb{R}_+^3)}\,ds\nonumber\\
 &+C\int_{\frac t2}^t(t-s)^{-1}\|\nabla d(s)\|_{L^{2r}(\mathbb{R}_+^3)}\|\nabla (\nabla d(s)-\nabla d(t))\|_{L^{2r}(\mathbb{R}_+^3)}\,ds\nonumber\\
 \leq & C\int_{\frac t2}^t\Big((t-s)^{\delta-1}s^{-\frac92-\delta+\frac 3{2r}}+(t-s)^{-\frac12}s^{-6+\frac 3{2r}}+(t-s)^{\delta-1}s^{-\delta-\frac{11}2+\frac 3{2r}}\Big)\,ds\nonumber\\
 \leq & Ct^{-\frac92+\frac 3{2r}}.
\end{align}

A combination of $(\ref{3.16})-$$(\ref{3.21})$ shows that for any $1<r<\infty$ and $t>0$,
$$\|\mathbb{A}u(t)\|_{L^r(\mathbb{R}_+^3)}\leq{C}t^{-1-\frac 32(1-\frac1r)},$$
which implies that (see \cite{Borchers-Miyakawa1988})
$$\|\nabla^2u(t)\|_{L^r(\mathbb{R}_+^3)}\leq{C}\|\mathbb{A}u(t)\|_{L^r(\mathbb{R}_+^3)}\leq Ct^{-1-\frac 32(1-\frac1r)}.$$
Note that $\partial_tu(t)=-\mathbb{A}u(t)-\mathbb{P}(u\cdot\nabla u)(t)-\mathbb{P}\nabla\cdot(\nabla d\odot\nabla d)(t)$, we have
\begin{align*}
\|\partial_tu(t)\|_{L^r(\mathbb{R}_+^3)}&\leq\|\mathbb{A}u(t)\|_{L^r(\mathbb{R}_+^3)}+\|u(t)\|_{L^{r_1}(\mathbb{R}_+^3)}\|\nabla u(t)\|_{L^{r_2}(\mathbb{R}_+^3)}+\|\nabla d(t)\|_{L^{r_1}(\mathbb{R}_+^3)}\|\nabla^2 d(t)\|_{L^{r_2}(\mathbb{R}_+^3)}\\
&\leq C\Big(t^{-1-\frac 32(1-\frac1r)}+t^{-\frac72+\frac3{2r}}+t^{-\frac92+\frac3{2r}}\Big)  \leq Ct^{-1-\frac 32(1-\frac1r)},
\end{align*}
and then
\begin{align*}
\|\nabla p(t)\|_{L^r(\mathbb{R}_+^3)}\leq&\|\partial_tu(t)\|_{L^r(\mathbb{R}_+^3)}+\|\nabla^2u(t)\|_{L^r(\mathbb{R}_+^3)}+\|u(t)\|_{L^{r_1}(\mathbb{R}_+^3)}\|\nabla u(t)\|_{L^{r_2}(\mathbb{R}_+^3)}\\
&+\|\nabla d(t)\|_{L^{r_1}(\mathbb{R}_+^3)}\|\nabla^2 d(t)\|_{L^{r_2}(\mathbb{R}_+^3)}\\
\leq& Ct^{-1-\frac 32(1-\frac1r)},
\end{align*}
where $1<r<\infty$, $\frac1r=\frac1{r_1}+\frac1{r_2}$, ${r_1}, {r_2}\in(1,\infty)$.

On the other hand, by $(\ref{1.1-main})_2$,
it follows from Lemma $\ref{le2.3}$ and Corollary $\ref{cor3.1}$ that
\begin{align*}
\|\partial_t{d}(t)\|_{L^{r}(\mathbb{R}_+^3)}
\leq&C \Big(\|u(t)\|_{L^{2r}(\mathbb{R}_+^3)}\|\nabla d(t)\|_{L^{2r}(\mathbb{R}_+^3)}+\|\nabla ^2d(t)\|_{L^{r}(\mathbb{R}_+^3)}+\|\nabla d(t)\|_{L^{2r}(\mathbb{R}_+^3)}^2\Big)\\
\leq &Ct^{-1-\frac{3}{2} (1-\frac1{r})}.
\end{align*}
Thus, we complete the proof of Lemma \ref{le3.4}.
\end{proof}
Next we will prove the Theorem $\ref{th1.3-main}$ in the case $k=1,2$ by using the similar arguments for the incompressible Navier-Stokes equations in the half-space (see \cite{Han-2016}).

\begin{Lemma}\label{le3.7**}
Assume $u_0, d_0-e_3\in L^1(\mathbb{R}_+^{3})$. Then there exists $t_1>0$ such that for $1<r<\infty$ and $t\geq t_1$, the strong solution $(u,d)$ of $(\ref{1.1-main})$-$(\ref{1.2})$ obtained in Proposition $\ref{solu1}$ satisfies
$$\|\nabla^{3}d(t)\|_{L^r(\mathbb{R}_+^3)}\leq Ct^{-\frac{3}2-\frac 32(1-\frac1r)}.$$
\end{Lemma}
\begin{proof}
Set $d_1(t)=\partial_t d(t)$. It is not difficult to verify that $d_1(t)$ satisfies
\begin{align}\label{d1}
	\left\{
	\begin{array}{lll}
		\partial_t d_1+u_t\cdot\nabla d+u\cdot\nabla d_1=\Delta d_1+|\nabla d|^2d_1+2(\nabla d:\nabla d_1) d&\text{in}\quad\mathbb{R}_+^3\times(1,\infty),\\
       \frac{\partial d_1}{\partial x_3}=0&\text{on}\quad\partial\mathbb{R}_+^3\times(1,\infty),\\
		d_1(x,1)=\partial_td(1)&\text{in}\quad\mathbb{R}_+^3.\\
	\end{array}
	\right.
\end{align}
Using the $\text{Duhamel}$'s principle and the solution operators to the heat equation in $\mathbb{R}_+^3$ (see \cite{Ukai-1987}), we can know that
\begin{align}\label{3.23}
d_1(x,t)
=&\int_{\mathbb{R}_+^3}\mathcal{W}(x,y,\frac{t}{2})d_1(y,\frac{t}{2})\,dy\nonumber\\
&-\int_{\frac{t}{2}}^t\int_{\mathbb{R}_+^3}\mathcal{W}(x,y,t-s)\big(u_t\cdot\nabla d+u\cdot\nabla d_1-|\nabla d|^2d_1-2(\nabla d:\nabla d_1) d\big)(y,s)\,dyds,
\end{align}
where $\mathcal{W}(x,y,t)=G_t(x-y)-G_t(x-y^*)$, $y^*=(y_1,y_2,-y_3)$. Moreover, $G_t(x-y^*)$ satisfies the following estimate (see \cite[p.346]{Solonnikov-2003}):
\begin{equation}\label{3.24}
|\partial_t^{\widetilde{s}}\partial_x^{\widetilde{\ell}}\partial_y^{\widetilde{m}}G_t(x-y^*)|\leq Ct^{-s-\frac{\widetilde{m}_3}{2}}(t+x_3^2)^{-\frac{\widetilde{\ell}_3}{2}}(|x-y^*|^2+t)^{-\frac{3+|\widetilde{\ell}'|+|\widetilde{m}'|}{2}}e^{-\frac{cy_3^2}{t}},
\end{equation}
where $\widetilde{m}=(\widetilde{m}_1,\widetilde{m}_2,\widetilde{m}_3)=(\widetilde{m}',\widetilde{m}_3)$, $
\widetilde{\ell}=(\widetilde{\ell}_1,\widetilde{\ell}_2,\widetilde{\ell}_3)=(\widetilde{\ell'},\widetilde{\ell}_3).$

Let $1<r<\infty$. Applying Corollary $\ref{cor3.1}$, Lemma  $\ref{le2.2}$, Lemma $\ref{le2.3}$, Lemma $\ref{pro3.1}$ and Lemma $\ref{le3.4}$, together with $(\ref{3.24})$, we obtain that for any integer $m\geq1$ and $t>2$,
\begin{align}\label{3.25}
&\|\nabla_x^m\int_{\mathbb{R}_+^3}\mathcal{W}(\cdot,y,\frac{t}{2})d_1(y,\frac{t}{2})\,dy\|_{L^r(\mathbb{R}_+^3)}\nonumber\\
\leq &C\|\nabla^mG_t\|_{L^1(\mathbb{R}_+^3)}\|d_1(\frac t2)\|_{L^r(\mathbb{R}_+^3)}
+C\|d_1(\frac t2)\|_{L^r(\mathbb{R}_+^3)}\times\nonumber\\
&\sum_{m=m_3+|m'|}\int_{\mathbb{R}_+^3}(x_3+\sqrt{t} )^{-m_3}(|x'|+x_3+\sqrt{t})^{-3-|m'|}\,dx'dx_3\nonumber\\
\leq & Ct^{-\frac{m+2}{2}-\frac{3}{2}(1-\frac{1}{r})};
\end{align}
and
\begin{align}\label{3.26}
&\big\|\int_{\frac{t}{2}}^{t}\int_{\mathbb{R}_{+}^{3}}\nabla_{x}\mathcal{W}(x,y,t-s)\big(u_t\cdot\nabla d+u\cdot\nabla d_1-|\nabla d|^2d_1-2(\nabla d:\nabla d_1) d\big)(y,s)\,dyds\big\|_{L^{r}(\mathbb{R}_{+}^{3})}\nonumber\\
\leq & C\int_{\frac{t}{2}}^t\|\nabla G_{t-s}(\cdot)\|_{L^1(\mathbb{R}_+^3)}\big\|\big(u_t\cdot\nabla d+u\cdot\nabla d_1+|\nabla d|^2d_1+(\nabla d:\nabla d_1) d\big)(s)\big\|_{L^r(\mathbb{R}_+^3)}\,ds\nonumber\\
\leq &C\int_{\frac t2}^t(t-s)^{-\frac12}\big(\|\partial_tu(s)\|_{L^{2r}(\mathbb{R}_+^3)}\|\nabla d(s)\|_{L^{2r}(\mathbb{R}_+^3)}
+\|u(s)\|_{L^{\infty}(\mathbb{R}_+^3)}\|\nabla d_1(s)\|_{L^{r}(\mathbb{R}_+^3)}\nonumber\\
&+\|\nabla d(s)\|_{L^{4r}(\mathbb{R}_+^3)}^2\|d_1(s)\|_{L^{2r}(\mathbb{R}_+^3)}+\|\nabla d(s)\|_{L^{\infty}(\mathbb{R}_+^3)}\|\nabla d_1(s)\|_{L^{r}(\mathbb{R}_+^3)}\big)\,ds\nonumber\\
\leq &C\int_{\frac{t}{2}}^{t}(t-s)^{-\frac{1}{2}}\big(s^{-1-\frac{3}{2}(1-\frac{1}{2r})-\frac{1}{2}-\frac{3}{2}(1-\frac{1}{2r})}
+s^{-1-3(1-\frac{1}{4r})-1-\frac{3}{2}(1-\frac{1}{2r})}\big)\,ds\nonumber\\
&+Cf_1(t)\int_{\frac{t}{2}}^t(t-s)^{-\frac{1}{2}}\big(s^{-\frac{3}{2}-\frac{3}{2}-\frac{3}{2}(1-\frac{1}{r})}+s^{-2-\frac{3}{2}-\frac{3}{2}(1-\frac{1}{r})}\big)\,ds\nonumber\\
\leq &Ct^{-\frac{5}{2}-\frac{3}{2}(1-\frac1r)}+C_1t^{-\frac{5}{2}-\frac{3}{2}(1-\frac1r)}f_1(t),
\end{align}
where $f_1(t)=\sup\limits_{0<s\leq t}[s^{\frac{3}{2}+\frac{3}{2}(1-\frac{1}{r})}\|\nabla d_1(s)\|_{L^r(\mathbb{R}_+^3)}].$

By $(\ref{3.23})$, $(\ref{3.25})$ and $(\ref{3.26})$, we have
$$t^{\frac{3}{2}+\frac{3}{2}(1-\frac{1}{r})}\|\nabla d_1(t)\|_{L^r(\mathbb{R}_+^3)}\leq C+C_1t^{-1}f_1(t),$$
from which,
$$f_1(t)\leq C+C_1t^{-1}f_1(t).$$
This implies that $f_1(t)$ is bounded for some $t_1\leq t$, namely,
\begin{equation}\label{3.28}
\|\nabla\partial_td(t)\|_{L^r(\mathbb{R}_+^3)}=\|\nabla d_1(t)\|_{L^r(\mathbb{R}_+^3)}\leq Ct^{-\frac{3}{2}-\frac{3}{2}(1-\frac{1}{r})}.
\end{equation}
In view of Lemma $\ref{le2.3}$ and Corollary $\ref{cor3.1}$, we find that
 \begin{align}\label{3.30}
 \|\nabla(u\cdot\nabla d)(t)\|_{L^r(\mathbb{R}_+^3)}&\leq C(\|\nabla u(t)\|_{L^{2r}(\mathbb{R}_+^3)}\|\nabla d(t)\|_{L^{2r}(\mathbb{R}_+^3)}
+\|u(t)\|_{L^{2r}(\mathbb{R}_+^3)}\|\nabla^2d(t)\|_{L^{2r}(\mathbb{R}_+^3)})\nonumber\\
&\leq Ct^{{-\frac{5}{2}-\frac{3}{2}(1-\frac{1}{r})}},
 \end{align}
 and
  \begin{align}\label{3.31}
 \|\nabla( |\nabla d|^2d)(t)\|_{L^r(\mathbb{R}_+^3)}&\leq C(
\|\nabla d(t)\|_{L^{2r}(\mathbb{R}_+^3)}\|\nabla^2 d(t)\|_{L^{2r}(\mathbb{R}_+^3)}+\|\nabla d(t)\|_{L^{3r}(\mathbb{R}_+^3)}^3)\nonumber\\
&\leq Ct^{{-3-\frac{3}{2}(1-\frac{1}{r})}}
  \end{align}
 hold for any $1<r<\infty$ and $t\geq t_1$.

Applying the standard elliptic estimates (see \cite{Saal-2006}) to system $(\ref{1.1-main})$-$(\ref{1.2})$, we can know that for each integer $\ell\geq0$,
\begin{equation}\label{3.29}
\|\nabla^{\ell+2}d(t)\|_{L^r(\mathbb{R}_+^3)}\leq C\big(\|\nabla^\ell(u\cdot\nabla d)(t)\|_{L^r(\mathbb{R}_+^3)}+\|\nabla^\ell\partial_td(t)\|_{L^r(\mathbb{R}_+^3)}+\|\nabla^\ell( |\nabla d|^2d)(t)\|_{L^r(\mathbb{R}_+^3)}\big).
\end{equation}
Putting $(\ref{3.28})$-$(\ref{3.29})$ together, we obtain for $1<r<\infty$ and $t\geq t_1$,
  \begin{align*}
  \|\nabla^{3}d(t)\|_{L^r(\mathbb{R}_+^3)}&\leq C\big(\|\nabla(u\cdot\nabla d)(t)\|_{L^r(\mathbb{R}_+^3)}+\|\nabla\partial_td(t)\|_{L^r(\mathbb{R}_+^3)}+\|\nabla( |\nabla d|^2d)(t)\|_{L^r(\mathbb{R}_+^3)}\big)\nonumber\\
  &\leq Ct^{-\frac{3}{2}-\frac{3}{2}(1-\frac{1}{r})}.
\end{align*}
\end{proof}
\begin{Lemma}\label{le3.8}
Assume $u_0, d_0-e_3\in L^1(\mathbb{R}_+^{3})$. Then there exists $t_2>0$ such that for $1<r<\infty$ and $t\geq t_2$, the strong solution $(u,d,p)$ of $(\ref{1.1-main})$-$(\ref{1.2})$ obtained in Proposition $\ref{solu1}$ satisfies
\begin{align*}
&\|\nabla^{3}u(t)\|_{L^r(\mathbb{R}_+^3)}+\|\nabla^2p(t)\|_{L^r(\mathbb{R}_+^3)}\leq Ct^{-\frac{3}2-\frac 32(1-\frac1r)};\\
&\|\nabla^{4}d(t)\|_{L^r(\mathbb{R}_+^3)}\leq Ct^{-2-\frac 32(1-\frac1r)}.
\end{align*}
\end{Lemma}
\begin{proof}We will divide the proof into three steps.

\textbf{Step 1:}\ The estimate of $\nabla\partial_t u(t)$.

Set $u_1(t)=\partial_t u(t)$. Taking $\partial_t$ over problem $(\ref{1.1-main})_1$, we find that $u_1(t)$ satisfies
\begin{align}\label{u1}
	\left\{
	\begin{array}{lll}
		\partial_tu_1-\Delta u_1+(u_1\cdot\nabla u+u\cdot\nabla u_1)+\nabla\partial_tp+ \nabla \cdot(\nabla d_1\odot\nabla d)+\nabla \cdot(\nabla d\odot\nabla d_1)=0&\text{in}\quad\mathbb{R}_+^3\times(1,\infty),\\
         \nabla \cdot u_1=0&\text{in}\quad\mathbb{R}_+^3\times(1,\infty),\\
       u_1(x,t)=0&\text{on}\quad\partial\mathbb{R}_+^3\times(1,\infty),\\
		u_1(x,1)=\partial_tu(1)&\text{in}\quad\mathbb{R}_+^3.\\
	\end{array}
	\right.
\end{align}
Applying the $\text{Duhamel}$'s principle and the nonstationary Stokes solution formula (see \cite{Solonnikov2003,Solonnikov-2003}), we can know that
\begin{align}\label{u1-}
u_1(x,t)=&\int_{\mathbb{R}_+^3}\mathcal{M}(x,y,\frac{t}{2})u_1(y,\frac{t}{2})\,dy\nonumber\\
&-\int_{\frac{t}{2}}^t\int_{\mathbb{R}_+^3}\mathcal{M}(x,y,t-s)\mathbb{P}\big(u_1\cdot\nabla u+u\cdot\nabla u_1+\nabla \cdot(\nabla d_1\odot\nabla d)+\nabla \cdot(\nabla d\odot\nabla d_1)\big)(y,s)\,dyds,
\end{align}
where $\mathcal{M}=(M_{ij})_{i,j=1,2,3}$ is defined as follows for $1\leq i,j\leq3$
$$M_{ij}(x,y,t)=\delta_{ij}(G_t(x-y)-G_t(x-y^*))+M_{ij}^*(x,y,t)=\delta_{ij}G_t(x-y)+N_{ij}^*(x,y,t).$$
Here
$$M_{ij}^*(x,y,t)=4(1-\delta_{j3})\frac{\partial}{\partial x_j}\int_{0}^{x_3}\int_{\mathbb{R}^{2}}\frac{\partial E(x-z)}{\partial x_i}G_t(z-y^*)\,dz,$$
and
$$N_{ij}^*(x,y,t)=-\delta_{ij}G_t(x-y^*)+M_{ij}^*(x,y,t),$$
$E(z)=-\frac{1}{4\pi|z|}$ is the fundamental solution of the Laplace equation. In addition, $M^*=(M_{ij}^*)_{3\times 3}$ and $ \mathcal{N}^*=(N_{ij}^*)_{3\times 3}$ still satisfy $(\ref{3.24})$ (see \cite{Solonnikov-2003}), namely,
\begin{align}\label{3.36*}
&|\partial_t^{\widetilde{s}}\partial_x^{\widetilde{\ell}}\partial_y^{\widetilde{m}}M^*(x,y,t)|+|\partial_t^{\widetilde{s}}\partial_x^{\widetilde{\ell}}\partial_y^{\widetilde{m}}\mathcal{N}^*(x,y,t)|\nonumber\\
\leq & Ct^{-s-\frac{\widetilde{m}_3}{2}}(t+x_3^2)^{-\frac{\widetilde{\ell}_3}{2}}(|x-y^*|^2+t)^{-\frac{3+|\widetilde{\ell}'|+|\widetilde{m}'|}{2}}e^{-\frac{cy_3^2}{t}}.
\end{align}

Let $1<r<\infty$. Similarly, using Corollary $\ref{cor3.1}$, Lemmas $\ref{le2.2}$-$\ref{le2.3}$, Lemmas $\ref{le3.4}$-$\ref{le3.7**}$ and Lemma $\ref{pro3.1}$, with the help of $(\ref{3.36*})$, we can establish the estimates for any integer $m\geq1$ and $t\geq 2t_1$,
\begin{align}\label{3.35}
&\|\nabla_x^m\int_{\mathbb{R}_+^3}\mathcal{M}(\cdot,y,\frac{t}{2})u_1(y,\frac{t}{2})\,dy\|_{L^r(\mathbb{R}_+^3)}\nonumber\\ 
\leq &C\|\nabla^mG_t\|_{L^1(\mathbb{R}_+^3)}\|u_1(\frac t2)\|_{L^r(\mathbb{R}_+^3)}+C\|u_1(\frac t2)\|_{L^r(\mathbb{R}_+^3)}\times\nonumber\\
&\sum_{m=m_3+|m'|}\int_{\mathbb{R}_+^3}(x_3+\sqrt{t} )^{-m_3}(|x'|+x_3+\sqrt{t} )^{-3-|m'|}\,dx'dx_3\nonumber\\
\leq & Ct^{-\frac{m+2}{2}-\frac{3}{2}(1-\frac{1}{r})};
\end{align}
and
\begin{align}\label{3.36}
&\big\|\int_{\frac{t}{2}}^{t}\int_{\mathbb{R}_{+}^{3}}\nabla_{x}\mathcal{M}(x,y,t-s)\mathbb{P}\big(u_1\cdot\nabla u+u\cdot\nabla u_1+\nabla \cdot(\nabla d_1\odot\nabla d)+\nabla \cdot(\nabla d\odot\nabla d_1)\big)(y,s)\,dyds\big\|_{L^{r}(\mathbb{R}_{+}^{3})}\nonumber\\
\leq & C\int_{\frac{t}{2}}^t\|\nabla G_{t-s}\|_{L^1(\mathbb{R}_+^3)}\big\|\mathbb{P}\big(u_1\cdot\nabla u+u\cdot\nabla u_1+\nabla \cdot(\nabla d_1\odot\nabla d)+\nabla \cdot(\nabla d\odot\nabla d_1)\big)(s)\big\|_{L^r(\mathbb{R}_+^3)}\,ds\nonumber\\
&+C\int_{\frac{t}{2}}^t\big(\sup\limits_{\substack{x\in\mathbb{R}_+^3\\}}\|\nabla_x\mathcal{N}^{*}(x,\cdot,t-s)\|_{L^1(\mathbb{R}_+^3)}+\sup\limits_{\substack{y\in\mathbb{R}_+^3\\}}\|\nabla_x\mathcal{N}^{*}(\cdot,y,t-s)\|_{L^1(\mathbb{R}_+^3)}\big)\nonumber\\
&\times\|\mathbb{P}\big(u_1\cdot\nabla u+u\cdot\nabla u_1+\nabla \cdot(\nabla d_1\odot\nabla d)+\nabla \cdot(\nabla d\odot\nabla d_1)\big)(s)\|_{L^r(\mathbb{R}_+^3)}\,ds\nonumber\\
 \leq &C\int_{\frac t2}^t(t-s)^{-\frac12}\big(\|u_1\|_{L^{2r}(\mathbb{R}_+^3)}\|\nabla u(s)\|_{L^{2r}(\mathbb{R}_+^3)}
+\|u(s)\|_{L^{\infty}(\mathbb{R}_+^3)}\|\nabla u_1(s)\|_{L^{r}(\mathbb{R}_+^3)}\nonumber\\
&+\|\nabla d(s)\|_{L^{\infty}(\mathbb{R}_+^3)}\|\nabla^2 d_1(s)\|_{L^{r}(\mathbb{R}_+^3)}+\|\nabla d_1(s)\|_{L^{2r}(\mathbb{R}_+^3)}\|\nabla ^2d(s)\|_{L^{2r}(\mathbb{R}_+^3)}\big)\,ds\nonumber\\
\leq &C\int_{\frac{t}{2}}^{t}(t-s)^{-\frac{1}{2}}s^{-\frac{3}{2}-\frac{3}{2}-\frac{3}{2}(1-\frac{1}{r})}\,ds  \nonumber\\
&+C\int_{\frac{t}{2}}^{t}(t-s)^{-\frac{1}{2}}s^{-\frac{3}{2}}\big(\|\nabla u_1(s)\|_{L^{r}(\mathbb{R}_{+}^{3})}+\|\nabla^2 d_1(s)\|_{L^{r}(\mathbb{R}_{+}^{3})}  \big )\,ds\nonumber\\
\leq &C t^{-\frac 52-\frac 32(1-\frac1r)}+Ct^{-\frac{5}{2}-\frac{3}{2}(1-\frac{1}{r})}f_2(t),
\end{align}
where $f_2(t)=\sup\limits_{0<s\leq t}[s^{\frac{3}{2}+\frac{3}{2}(1-\frac{1}{r})}(\|\nabla u_1(s)\|_{L^{r}(\mathbb{R}_{+}^{3})}+\|\nabla^2 d_1(s)\|_{L^{r}(\mathbb{R}_{+}^{3})})] $.

\textbf{Step 2:}\ The estimate of $\nabla^2\partial_t d(t)$.

Recall that the following results hold (see \cite[p.16]{Han-2016}): For $1\leq k \leq 3$ and $0<s<t$, we have
\begin{align}\label{3.37}
&\int_{\mathbb{R}_+^3}\partial_{x_k}\partial_{x_3}\mathcal{W}(x,y,t-s)g(y,s)\,dy\nonumber\\
=&\int_{\mathbb{R}_+^3}\partial_{x_k}\mathcal{W}(x,y,t-s)\partial_{y_3}g(y,s)\,dy-2\int_{\mathbb{R}_+^3}\partial_{x_k}\partial_{x_3}G_{t-s}(x-y^*)g(y,s)\,dy
\end{align}
and
\begin{align}\label{3.33}
&\int_{\mathbb{R}_+^3}\partial_{x_k}\partial_{x_3}G_{t-s}(x-y)g(y,s)\,dy\nonumber\\
=&\int_{\mathbb{R}_+^3}\partial_{x_k}(G_{t-s}(x-y)-G_{t-s}(x-y^*))\partial_{y_3}g(y,s)\,dy-\int_{\mathbb{R}_+^3}\partial_{x_k}\partial_{x_3}G_{t-s}(x-y^*)g(y,s)\,dy.
\end{align}

Let $1<r<\infty$ and set $D(x,t):=u_1\cdot\nabla d+u\cdot\nabla d_1-|\nabla d|^2d_1-2(\nabla d:\nabla d_1) d$. Using one-dimensional Hardy inequality yields for $0<\theta <1$,
\begin{align}\label{3.38*}
&\|x_3^{-\theta}D(x,t)\|_{L^r(\mathbb{R}_+^3)}^r\nonumber\\
\leq&\int_{\mathbb{R}^{2}}\int_0^1x_3^{r-r\theta}\frac{|D(x^{\prime},x_3)|^r}{x_3^r}\,dx_3dx^{\prime}+\|D(x,t)\|_{L^r(\mathbb{R}^{2}\times[1,+\infty))}^r\nonumber\\
\leq&C(\|D(x,t)\|_{L^r(\mathbb{R}^{3}_+)}+\|\nabla_x D(x,t)\|_{L^r(\mathbb{R}^{3}_+)})^r.
\end{align}
Using $(\ref{3.24})$, $(\ref{3.28})$, $(\ref{3.37})$-$(\ref{3.38*})$, Lemmas $\ref{le2.2}$-$\ref{le2.3}$, Lemmas $\ref{le3.4}$-$\ref{le3.7**}$, Lemma $\ref{pro3.1}$ and Corollary $\ref{cor3.1}$, we find for $1\leq k\leq3$, $1\leq j\leq 2$ and $t\geq 2t_1$ that
\begin{align}\label{3.38}
&\big\|\int_{\frac{t}{2}}^{t}\int_{\mathbb{R}_{+}^{3}}\partial_{x_k}\partial_{x_j}\mathcal{W}(x,y,t-s)\big(u_1\cdot\nabla d+u\cdot\nabla d_1-|\nabla d|^2d_1-2(\nabla d:\nabla d_1) d\big)(y,s)\,dyds\big\|_{L^{r}(\mathbb{R}_{+}^{3})}\nonumber\\
&+\big\|\int_{\frac{t}{2}}^{t}\int_{\mathbb{R}_{+}^{3}}\partial_{x_k}\partial_{x_3}\mathcal{W}(x,y,t-s)\big(u_1\cdot\nabla d+u\cdot\nabla d_1-|\nabla d|^2d_1-2(\nabla d:\nabla d_1) d\big)(y,s)\,dyds\big\|_{L^{r}(\mathbb{R}_{+}^{3})}\nonumber\\
\leq&C\int_{\frac{t}{2}}^{t}\|\partial_{x_k}G_{t-s}(\cdot)\|_{L^1(\mathbb{R}_+^3)}(\|\partial_3D(s)\|_{L^r(\mathbb{R}_+^3)}+\|\partial_jD(s)\|_{L^r(\mathbb{R}_+^3)})\,ds\nonumber\\
&+C\int_{\frac{t}{2}}^{t}\|x_{3}^{2\varepsilon}\partial_{x_{k}}\partial_{x_{3}}G_{t-s}(\cdot)\|_{L^{1}(\mathbb{R}_{+}^{3})}\|y_{3}^{-2\varepsilon}D(s)\|_{L^{r}(\mathbb{R}_{+}^{3})}\,ds\nonumber\\
\leq& C\int_{\frac{t}{2}}^{t}(t-s)^{-\frac{1}{2}}
\big(\| u_1(s)\|_{L^{2r}(\mathbb{R}_{+}^{3})}\|\nabla^2 d(s)\|_{L^{2r}(\mathbb{R}_{+}^{3})}+\|\nabla u(s)\|_{L^{2r}(\mathbb{R}_{+}^{3})}\|\nabla d_1(s)\|_{L^{2r}(\mathbb{R}_{+}^{3})}\nonumber\\
&+\|\nabla d(s)\|_{L^{4r}(\mathbb{R}_{+}^{3})}^2\|\nabla d_1(s)\|_{L^{2r}(\mathbb{R}_{+}^{3})}
+\| \nabla d(s)\|_{L^{3r}(\mathbb{R}_{+}^{3})}\|\nabla^2 d(s)\|_{L^{3r}(\mathbb{R}_{+}^{3})}\|  d_1(s)\|_{L^{3r}(\mathbb{R}_{+}^{3})}\nonumber\\
&+\|\nabla^2 d(s)\|_{L^{2r}(\mathbb{R}_{+}^{3})}\|\nabla d_1(s)\|_{L^{2r}(\mathbb{R}_{+}^{3})}\big)\,ds\nonumber\\
&+C\int_{\frac{t}{2}}^{t}(t-s)^{-\frac{1}{2}}\big(\|\nabla d(s)\|_{L^{\infty}(\mathbb{R}_{+}^{3})}\|\nabla u_1(s)\|_{L^{r}(\mathbb{R}_{+}^{3})}
          +(\| u(s)\|_{L^{\infty}(\mathbb{R}_{+}^{3})}
          +\| \nabla d(s)\|_{L^{\infty}(\mathbb{R}_{+}^{3})})\|\nabla^2 d_1(s)\|_{L^{r}(\mathbb{R}_{+}^{3})}   \big )\,ds\nonumber\\
&+C\int_{\frac t2}^t(t-s)^{-1+\varepsilon}
    \big( \|u_1(s)\|_{L^{2r}(\mathbb{R}_+^3)}\|\nabla d(s)\|_{L^{2r}(\mathbb{R}_+^3)}
      +\|u(s)\|_{L^{2r}(\mathbb{R}_+^3)}\|\nabla d_1(s)\|_{L^{2r}(\mathbb{R}_+^3)} \nonumber\\
&+\|\nabla d(s)\|_{L^{4r}(\mathbb{R}_+^3)}^2\| d_1(s)\|_{L^{2r}(\mathbb{R}_+^3)}
+\|\nabla d(s)\|_{L^{2r}(\mathbb{R}_+^3)}\|\nabla d_1(s)\|_{L^{2r}(\mathbb{R}_+^3)}   \big)\,ds\nonumber\\
\leq&C\int_{\frac{t}{2}}^{t}(t-s)^{-\frac{1}{2}}s^{-2-\frac{3}{2}-\frac{3}{2}(1-\frac{1}{r})}ds+C\int_{\frac{t}{2}}^t(t-s)^{-1+\varepsilon}s^{-\frac{3}{2}-\frac{3}{2}-\frac{3}{2}(1-\frac{1}{r})}ds\nonumber\\
     &+C\int_{\frac{t}{2}}^{t}(t-s)^{-\frac{1}{2}}s^{-\frac{3}{2}}\big(\|\nabla u_1(s)\|_{L^{r}(\mathbb{R}_{+}^{3})}+\|\nabla^2 d_1(s)\|_{L^{r}(\mathbb{R}_{+}^{3})}  \big )\,ds\nonumber\\
\leq &C_\varepsilon t^{-\frac 52-\frac 32(1-\frac1r)}
+Ct^{-\frac{5}{2}-\frac{3}{2}(1-\frac{1}{r})}f_2(t),
\end{align}
where $\varepsilon\in(0,\frac12)$, $f_2(t)=\sup\limits_{0<s\leq t}[s^{\frac{3}{2}+\frac{3}{2}(1-\frac{1}{r})}(\|\nabla u_1(s)\|_{L^{r}(\mathbb{R}_{+}^{3})}+\|\nabla^2 d_1(s)\|_{L^{r}(\mathbb{R}_{+}^{3})})] $.

\textbf{Step 3:}\ The estimates of $\nabla^3 u(t)$ and $\nabla^4d(t)$.

Using $(\ref{3.23})$, $(\ref{3.25})$, $(\ref{u1-})$, $(\ref{3.35})$, $(\ref{3.36})$ and $(\ref{3.38})$, we conclude that for $t\geq 2t_1$,
\begin{equation}\label{3.39}
t^{\frac{3}{2}+\frac{3}{2}(1-\frac{1}{r})}\big(\|\nabla u_1(t)\|_{L^{r}(\mathbb{R}_{+}^{3})}+\|\nabla^2 d_1(t)\|_{L^{r}(\mathbb{R}_{+}^{3})}  \big )\leq C+C_2t^{-1}f_2(t).
\end{equation}
From $(\ref{3.39})$, we can know that for some $t_2\leq t$,
\begin{equation}\label{3.40}
\|\nabla \partial_tu(t)\|_{L^{r}(\mathbb{R}_{+}^{3})}+\|\nabla^2\partial_t d(t)\|_{L^{r}(\mathbb{R}_{+}^{3})}\leq Ct^{-\frac{3}{2}-\frac{3}{2}(1-\frac{1}{r})}.
\end{equation}
Substituting $(\ref{3.40})$ into $(\ref{3.38})$ and running the same argument as in $(\ref{3.38})$, combined with $(\ref{3.25})$, we derive for $t\geq t_2$,
\begin{equation}\label{3.41}
\|\nabla^2 \partial_td(t)\|_{L^{r}(\mathbb{R}_{+}^{3})}\leq Ct^{-2-\frac{3}{2}(1-\frac{1}{r})}.
\end{equation}
By applying $(\ref{3.29})$ with $\ell=2$, together with $(\ref{3.41})$, Lemma $\ref{le2.3}$, Lemma $\ref{le3.4}$, Lemma $\ref{le3.7**}$ and Corollary $\ref{cor3.1}$, we conclude for $1<r<\infty$ and $t\geq t_2$,
\begin{align*}
\|\nabla^{4}d(t)\|_{L^r(\mathbb{R}_+^3)}&\leq C\big(\|\nabla^2(u\cdot\nabla d)(t)\|_{L^r(\mathbb{R}_+^3)}+\|\nabla^2\partial_td(t)\|_{L^r(\mathbb{R}_+^3)}+\|\nabla^2( |\nabla d|^2d)(t)\|_{L^r(\mathbb{R}_+^3)}\big)\nonumber\\
&\leq Ct^{-2-\frac{3}{2}(1-\frac{1}{r})}.
\end{align*}

On the other hand, applying $(\ref{3.43})$ to system $(\ref{1.1-main})$-$(\ref{1.2})$, we get for $1<r<\infty$, $\ell\geq0$ and $t\geq t_2$,
\begin{align}\label{3.44}
&\|\nabla^{\ell+2}u(t)\|_{L^r(\mathbb{R}_+^3)}+\|\nabla^{\ell+1}p(t)\|_{L^r(\mathbb{R}_+^3)}\nonumber\\
\leq &C\big(\|\nabla^\ell(u\cdot\nabla u)(t)\|_{L^r(\mathbb{R}_+^3)}+\|\nabla^\ell\partial_tu(t)\|_{L^r(\mathbb{R}_+^3)}+\|\nabla^\ell(\nabla \cdot(\nabla d\odot\nabla d))(t)\|_{L^r(\mathbb{R}_+^3)}\big).
\end{align}
Using $(\ref{3.40})$, $(\ref{3.44})$, Lemma $\ref{le2.3}$, Lemma $\ref{le3.4}$, Lemma $\ref{le3.7**}$ and Corollary $\ref{cor3.1}$, we get for $1<r<\infty$ and $t\geq t_2$,
\begin{align*}
&\|\nabla^3u(t)\|_{L^r(\mathbb{R}_+^3)}+\|\nabla^2p(t)\|_{L^r(\mathbb{R}_+^3)}\nonumber\\
\leq&C\big(\|\nabla(u\cdot\nabla u)(t)\|_{L^r(\mathbb{R}_+^3)}+\|\nabla\partial_tu(t)\|_{L^r(\mathbb{R}_+^3)}+\|\nabla(\nabla \cdot(\nabla d\odot\nabla d))(t)\|_{L^r(\mathbb{R}_+^3)}\big)\nonumber\\
\leq& Ct^{-\frac{3}{2}-\frac{3}{2}(1-\frac{1}{r})}.
\end{align*}
Thus, we complete the proof of Lemma \ref{le3.8}.
\end{proof}

Next, we will establish the $L^1$-decay rates for the first-order spatial derivatives of velocity $u$.

\begin{Lemma}\label{le3.9*}
Assume $u_0, d_0-e_3\in L^1(\mathbb{R}_+^{3})$, if $(u,d)$ is the global strong solution obtained by Proposition $\ref{solu1}$, then
$$\|\nabla u(t)\|_{L^1(\mathbb{R}_+^3)}\leq Ct^{-\frac12}$$
holds for any $t>0$.
\end{Lemma}
\begin{proof}
Using Lemma $\ref{le2.2}$, Lemma $\ref{le2.3}$, Lemma $\ref{le3.7**}$ and  Corollary \ref{cor3.1}, together with $(\ref{uB1})_1$, $(\ref{2.3})$ and $(\ref{le2.5})$, we obtain for any $t>0$ that
\begin{align*}
&\|\nabla u(t)\|_{L^1(\mathbb{R}_+^3)}\\
\leq& Ct^{-\frac{1}{2}}\|u_0\|_{L^1(\mathbb{R}_+^3)}+\int_0^t\|\nabla e^{-(t-s)\mathbb{A}}\mathbb{P}\big(u\cdot\nabla u+\nabla\cdot(\nabla d\odot\nabla d)\big)(s)\|_{L^1(\mathbb{R}_+^3)}\,{d}s\\
\leq& Ct^{-\frac{1}{2}}+C\int_0^t(t-s)^{-\frac{1}{2}}\big(\|u\cdot\nabla u+\nabla\cdot(\nabla d\odot\nabla d)\|_{L^1(\mathbb{R}_+^3)} \\ &+\|\sum_{i,j=1}^3\nabla\mathcal{N}\partial_i\partial_j(u_iu_j+\langle\partial_id,\partial_jd\rangle)\|_{L^1(\mathbb{R}_+^3)}  \big)\,ds\\
\leq&Ct^{-\frac{1}{2}}+C\int_0^t(t-s)^{-\frac{1}{2}}\big(\|u\|_{H^1(\mathbb{R}_+^3)}^2+\|\nabla d\|_{H^1(\mathbb{R}_+^3)}^2\\
&+\||\nabla d||\nabla^2d|\|_{L^1(\mathbb{R}_+^3)}+\||\nabla d||\nabla^3d|\|_{L^1(\mathbb{R}_+^3)}\big)(s)\,{d}s\\
\leq&Ct^{-\frac{1}{2}}+C\int_0^t(t-s)^{-\frac{1}{2}}\big(\|u(s)\|_{H^1(\mathbb{R}_+^3)}^2+\|\nabla d(s)\|_{H^1(\mathbb{R}_+^3)}^2\\
&+\|\nabla d(s)\|_{L^2(\mathbb{R}_+^3)}\|\nabla ^2 d(s)\|_{L^2(\mathbb{R}_+^3)}+\|\nabla d(s)\|_{L^2(\mathbb{R}_+^3)}\|\nabla ^3 d(s)\|_{L^2(\mathbb{R}_+^3)}\big)\,ds\\
\leq&Ct^{-\frac{1}{2}}+C\int_0^t(t-s)^{-\frac{1}{2}}\big((1+s)^{-\frac32}+(1+s)^{-\frac52}+(1+s)^{-3}+(1+s)^{-\frac72}\big)\,ds\\
\leq&Ct^{-\frac{1}{2}}.
\end{align*}
\end{proof}

Finally, we derive the decay rates of the second-order derivatives of $u$ in $L^r(\mathbb{R}_+^3)$ with $r=1,\infty$.

\begin{Lemma}\label{le3.7}
Assume $u_0, d_0-e_3\in L^1(\mathbb{R}_+^{3})$, if $(u,d)$ is the global strong solution obtained by Proposition $\ref{solu1}$, then
$$\|\nabla^2 u(t)\|_{L^1(\mathbb{R}_+^3)}\leq Ct^{-1}$$
holds for any $t>0$.
\end{Lemma}
\begin{proof}
It follows from Lemmas $\ref{lea21}$-$\ref{le2.3}$, Lemma $\ref{le3.7**}$, Corollary $\ref{cor3.1}$ and $(\ref{le2.5})$ that for any $t>0$,

\begin{align*}
&\big\|\int_0^{\frac t2}\nabla^2e^{-(t-s)\mathbb{A}}\mathbb{P}\Big(u(s)\cdot\nabla u(s)+\nabla\cdot(\nabla d\odot\nabla d)(s)\Big)\,ds\big\|_{L^1(\mathbb{R}_+^3)}\\
\leq &C\int_{0}^{\frac{t}{2}}(t-s)^{-1}\big(\|(u\cdot\nabla u)(s)+\nabla\cdot(\nabla d\odot\nabla d)(s) \|_{L^1(\mathbb{R}_+^3)}\\
    &+\big\|\sum\limits_{i,j=1}^3\nabla\mathcal{N}\partial_i\partial_j(u_iu_j+\langle\partial_id,\partial_jd\rangle)(s)\big\|_{L^1(\mathbb{R}_+^3)}\big)\,ds\\
\leq &Ct^{-1}\int_0^{\frac t2}\big(\|u(s)\|_{H^1(\mathbb{R}_+^3)}^2+\|\nabla d(s)\|_{H^1(\mathbb{R}_+^3)}^2\\
&+\||\nabla d(t)||\nabla^3d(s)|\|_{L^1(\mathbb{R}_+^3)}+ \||\nabla d(t)||\nabla^2d(s)|\|_{L^1(\mathbb{R}_+^3)} \big)\,ds\\
\leq & Ct^{-1}\int_0^{\frac{t}{2}}(1+s)^{-\frac{3}{2}}\,ds\\
\leq & Ct^{-1}.
\end{align*}
In the same way, using Lemmas $\ref{lea21}$-$\ref{le2.3}$, Lemma $\ref{pro3.2}$, Lemma $\ref{le3.10}$, Lemmas $\ref{le3.7**}$-$\ref{le3.8}$ and $(\ref{3.50})$, we derive for any $0<\eta<1$ and $t>0$,
\begin{align*}
&\big\|\int_{\frac t2}^{t}\nabla^2e^{-(t-s)\mathbb{A}}\mathbb{P}\Big(u(s)\cdot\nabla u(s)+\nabla\cdot(\nabla d\odot\nabla d)(s)\Big)\,{d}s\big\|_{L^1(\mathbb{R}_+^3)}\\
\leq &C\int_{\frac t2}^{t}(t-s)^{-\frac12}\big(\big\|\nabla\Big(u\cdot\nabla u+\nabla\cdot(\nabla d\odot\nabla d)\Big)(s)\big\|_{L^1(\mathbb{R}_+^3)}\\
  &+\big\|\nabla\Big(\sum\limits_{i,j=1}^3
  \nabla\mathcal{N}\partial_i\partial_j(u_iu_j+\langle\partial_id,\partial_jd\rangle)\Big)(s)\big\|_{L^1(\mathbb{R}_+^3)}\big)\,{d}s\\
  &+C\int_{\frac t2}^{t}(t-s)^{-1+\frac{\eta}2}\Big(\|x_{3}^{-\eta}\mathbb{P}u(s)\cdot\nabla u(s)\|_{L^1(\mathbb{R}_+^3)}
  +\|x_{3}^{-\eta}\mathbb{P}\nabla\cdot(\nabla d\odot\nabla d)(s)\|_{L^1(\mathbb{R}_+^3)} \Big)\,{d}s\\
\leq &C\int_{\frac t2}^{t}(t-s)^{-\frac12}\big(\|u\|_{H^{1}(\mathbb{R}_+^3)}^2+\|\nabla^2u\|_{L^{2}(\mathbb{R}_+^3)}^2+\|\nabla ^2d\|_{L^{2}(\mathbb{R}_+^3)}^2\\
&+\|\nabla d(t)\|_{L^2(\mathbb{R}_+^3)}\|\nabla^4d(t)\|_{L^2(\mathbb{R}_+^3)}+\|\nabla ^2d(t)\|_{L^2(\mathbb{R}_+^3)}\|\nabla^3d(t)\|_{L^2(\mathbb{R}_+^3)}\\
&+\|\nabla d\|_{L^{2}(\mathbb{R}_+^3)}\|\nabla^2d\|_{L^{2}(\mathbb{R}_+^3)}+\|\nabla d(t)\|_{L^2(\mathbb{R}_+^3)}\|\nabla^3d(t)\|_{L^2(\mathbb{R}_+^3)}\big)\,{d}s\\
&+C\int_{\frac t2}^{t}(t-s)^{-1+\frac{\eta}2}\big(\|u\|_{H^{1}(\mathbb{R}_+^3)}^2+\|\nabla^2u\|_{L^{2}(\mathbb{R}_+^3)}^2+\|\nabla d\|_{L^{2}(\mathbb{R}_+^3)}^2+\|\nabla^2 d\|_{L^{2}(\mathbb{R}_+^3)}^2\\
&+\|\nabla d\|_{L^{2}(\mathbb{R}_+^3)}\|\nabla^2d\|_{L^{2}(\mathbb{R}_+^3)}+\|\nabla d(t)\|_{L^2(\mathbb{R}_+^3)}\|\nabla^3d(t)\|_{L^2(\mathbb{R}_+^3)}\big)\,{d}s\\
\leq& C\int_{\frac t2}^{t}\big((t-s)^{-\frac12}+(t-s)^{-1+\frac{\eta}2}\big)s^{-\frac32}\,{d}s\\
\leq&Ct^{-1}.
\end{align*}
 Using these two estimates, together with $(\ref{uB1})_1$ and Lemma $\ref{lea21}$, we complete the proof of Lemma \ref{le3.7}.
\end{proof}

\begin{Lemma}\label{3.11}
Assume $u_0, d_0-e_3\in L^1(\mathbb{R}_+^{3})$, if $(u,d)$ is the global strong solution obtained by Proposition $\ref{solu1}$, then
$$\|\nabla^2 u(t)\|_{L^\infty(\mathbb{R}_+^3)}\leq Ct^{-\frac52}$$
holds for any $t>0$.
\end{Lemma}
\begin{proof}
Let $0<\theta<1$ and $\frac{3}{\theta}<r<\infty$. It follows from Lemma $\ref{le2.4}$, Lemmas $\ref{le2.2}$-$\ref{le2.3}$, Lemmas $\ref{pro3.2}$-$\ref{le3.3}$, Lemma $\ref{le3.10}$, Lemmas $\ref{le3.7**}$-$\ref{le3.8}$ and Corollary $\ref{cor3.1}$ that for any $t>0$,
\begin{align*}
&\big\|\int_{\frac{t}{2}}^{t}\nabla^2e^{-(t-s)\mathbb{A}}\mathbb{P}\Big(u(s)\cdot\nabla u(s)+\nabla\cdot(\nabla d\odot\nabla d)(s)\Big)\,ds\big\|_{L^\infty(\mathbb{R}_+^3)}\\
\leq &C\int_{\frac{t}{2}}^{t}(t-s)^{-\frac{1}{2}-\frac{3}{2r}}\big\|\nabla\mathbb{P}\Big((u\cdot\nabla u)(s)+\nabla\cdot(\nabla d\odot\nabla d)(s)\Big)\big\|_{L^r(\mathbb{R}_+^3)}\,ds\\
&+C\int_{\frac{t}{2}}^t(t-s)^{-1+\frac{\theta}{2}-\frac{3}{2r}}\big\|x_3^{-\theta}\mathbb{P}\Big((u\cdot\nabla u)(s)+\nabla\cdot(\nabla d\odot\nabla d)(s)\Big)\big\|_{L^r(\mathbb{R}_+^3)}\,ds\\
\leq &C\int_{\frac t2}^{t}(t-s)^{-\frac12-\frac{3}{2r}}\Big(\big\|\nabla\Big((u\cdot\nabla u)(s)+\nabla\cdot(\nabla d\odot\nabla d)(s)\Big)\big\|_{L^r(\mathbb{R}_+^3)}\\
&+\big\|\nabla\Big(\sum\limits_{i,j=1}^3\nabla\mathcal{N}\partial_i\partial_j\big(u_iu_j+\langle\partial_id,\partial_jd\rangle\big)\Big)(s)\big\|_{L^r(\mathbb{R}_+^3)}\Big)\,{d}s\\
&+C\int_{\frac{t}{2}}^t(t-s)^{-1+\frac{\theta}{2}-\frac{3}{2r}}\Big(\|x_3^{-\theta}\mathbb{P}(u\cdot\nabla u)(s)\|_{L^r(\mathbb{R}_+^3)}
+\|x_3^{-\theta}\mathbb{P}\nabla\cdot(\nabla d\odot\nabla d)(s)\|_{L^r(\mathbb{R}_+^3)}\Big) \,ds\\
\leq &C\int_{\frac t2}^{t}(t-s)^{-\frac12-\frac{3}{2r}}\big(\|u\|_{L^{2r}(\mathbb{R}_+^3)}^2+\|\nabla ^2d\|_{L^{2r}(\mathbb{R}_+^3)}^2+\|\nabla u\|_{L^{2r}(\mathbb{R}_+^3)}^2+\|\nabla^2u\|_{L^{2r}(\mathbb{R}_+^3)}^2\\
&+\|\nabla d(t)\|_{L^{2r}(\mathbb{R}_+^3)}\|\nabla^4d(t)\|_{L^{2r}(\mathbb{R}_+^3)}+\|\nabla ^2d(t)\|_{L^{2r}(\mathbb{R}_+^3)}\|\nabla^3d(t)\|_{L^{2r}(\mathbb{R}_+^3)}\\
&+\|\nabla d(t)\|_{L^{2r}(\mathbb{R}_+^3)}\|\nabla^3d(t)\|_{L^{2r}(\mathbb{R}_+^3)}+\|\nabla d(t)\|_{L^{2r}(\mathbb{R}_+^3)}\|\nabla^2d(t)\|_{L^{2r}(\mathbb{R}_+^3)}\big)\,{d}s\\
&+C\int_{\frac{t}{2}}^t(t-s)^{-1+\frac{\theta}{2}-\frac{3}{2r}}\big(\|u(s)\|_{L^{2r}(\mathbb{R}_+^3)}^2+\|\nabla u(s)\|_{L^{2r}(\mathbb{R}_+^3)}^2+\|\nabla d(s)\|_{L^{2r}(\mathbb{R}_+^3)}^2+\|\nabla ^2u\|_{L^{2r}(\mathbb{R}_+^3)}^2\\
&+\|\nabla ^2d(s)\|_{L^{2r}(\mathbb{R}_+^3)}^2+\|\nabla ^3d(s)\|_{L^{2r}(\mathbb{R}_+^3)}^2\big)\,ds\\
\leq &C\int_{\frac t2}^{t}(t-s)^{-\frac12-\frac{3}{2r}}s^{-3(1-\frac{1}{2r})}\,ds+C\int_{\frac{t}{2}}^t(t-s)^{-1+\frac{\theta}{2}-\frac{3}{2r}}s^{-3(1-\frac{1}{2r})}\,ds\\
\leq & Ct^{-\frac52}.
\end{align*}
Using the above estimate, combined with $(\ref{uB2})_1$ and $(\ref{2.2})$, the proof of Lemma \ref{3.11} is complete.
\end{proof}
\begin{proof}[Proof of Theorem $\ref{th1.2}$]
Theorem $\ref{th1.2}$ directly follows from Corollary $\ref{cor3.1}$, Lemma $\ref{le3.4}$ and Lemmas $\ref{le3.9*}$-$\ref{3.11}$.
\end{proof}

\section{Proof of Theorem $\ref{th1.3-main}$ }\label{sec4'}
This section is devoted to establishing the decay rates of $\| (\nabla^{2+k}u,\nabla^{2+k}d)(t)\|_{L^{r}(\mathbb{R}_{+}^{3})}$, where $k\geq1$, $1<r\leq\infty$, and $(u,d)$ is the strong solution obtained by Proposition \ref{solu1}. Firstly, we provide a review of the decay estimates of the strong solution that have been derived in Lemma \ref{le2.3}, Lemmas \ref{le3.7**}-\ref{le3.8} and Theorem \ref{th1.2}, as they are frequently utilized in the proofs of subsequent theorems.
\begin{Lemma}\label{le4.1-}
Assume $u_0, d_0-e_3\in L^1(\mathbb{R}_+^{3})$, if $(u,d,p)$ is the global strong solution obtained by Proposition $\ref{solu1}$, then the following estimates
\begin{align*}
\left\{
\begin{array}{lll}
	\displaystyle	\|(\nabla^m u,\nabla^m d)(t)\|_{L^p(\mathbb{R}_+^3)}\leq Ct^{-\frac m2-\frac32(1-\frac1p)},\ m=0,1,2,\\
	\displaystyle	\|\nabla^l d(t)\|_{L^q(\mathbb{R}_+^3)}\leq Ct^{-\frac l2-\frac32(1-\frac1q)},\ l=3,4,\\
    \displaystyle \|\nabla^3 u(t)\|_{L^q(\mathbb{R}_+^3)}\leq Ct^{-\frac 32-\frac32(1-\frac1q)},\\
     \displaystyle \|\partial_tu(t)\|_{{L^{q}(\mathbb{R}_{+}^{3})}}+\|\nabla p(t)\|_{{L^{q}(\mathbb{R}_{+}^{3})}}+\| \partial_td(t)\|_{L^q(\mathbb{R}_+^3)}\leq Ct^{-1-\frac32(1-\frac1q)}
\end{array}
\right.
\end{align*}
hold for any $t>0$, where $p\in (1,\infty]$ and $q\in (1,\infty)$.
\end{Lemma}

\begin{Lemma}\label{le3.9}
Assume $u_0, d_0-e_3\in L^1(\mathbb{R}_+^{3})$. Then there exists $t_3>0$ such that for $1<r<\infty$ and $t\geq t_3$, the strong solution $(u,d,p)$ of $(\ref{1.1-main})$-$(\ref{1.2})$ obtained in Proposition $\ref{solu1}$ satisfies
\begin{align*}
&\|\nabla^4 u(t)\|_{L^{r}(\mathbb{R}_{+}^{3})}+\|\nabla^3p(t)\|_{L^r(\mathbb{R}_+^3)}\leq Ct^{-2-\frac{3}{2}(1-\frac{1}{r})},\\
&\|\nabla^5 u(t)\|_{L^{r}(\mathbb{R}_{+}^{3})}+\|\nabla^4p(t)\|_{L^r(\mathbb{R}_+^3)}\leq Ct^{-\frac52-\frac{3}{2}(1-\frac{1}{r})},\\
&\|\nabla^5 d(t)\|_{L^{r}(\mathbb{R}_{+}^{3})}\leq Ct^{-\frac52-\frac{3}{2}(1-\frac{1}{r})},\\
&\|\nabla^6 d(t)\|_{L^{r}(\mathbb{R}_{+}^{3})}\leq Ct^{-3-\frac{3}{2}(1-\frac{1}{r})}.
\end{align*}
\end{Lemma}
\begin{proof}
We will divide the proof into six steps.

\textbf{Step1:}.
Using the regular estimate $(\ref{3.29})$ with $\ell=3$, we can obtain that for $t>1$,
$$\|\nabla^{5}d(t)\|_{L^r(\mathbb{R}_+^3)}\leq C\big(\|\nabla^3(u\cdot\nabla d)(t)\|_{L^r(\mathbb{R}_+^3)}+\|\nabla^3\partial_td(t)\|_{L^r(\mathbb{R}_+^3)}+\|\nabla^3( |\nabla d|^2d)(t)\|_{L^r(\mathbb{R}_+^3)}\big).$$
So in order to get the decay of $\|\nabla^{5}d(t)\|_{L^r(\mathbb{R}_+^3)}$, we should gain the decay of $\|\nabla^3\partial_td(t)\|_{L^r(\mathbb{R}_+^3)}$ at the first time.
Applying the elliptic estimates to $d_1$ system $(\ref{d1})$, we get for $t>1$,
\begin{align}\label{3.49}
\|\nabla^{3}d_1(t)\|_{L^r(\mathbb{R}_+^3)}\leq &C\big(\|\nabla(u_1\cdot\nabla d+u\cdot\nabla d_1)(t)\|_{L^r(\mathbb{R}_+^3)}\nonumber\\
&+\|\nabla(|\nabla d|^2d_1)(t)\|_{L^r(\mathbb{R}_+^3)}+\|\nabla((\nabla d:\nabla d_1) d)(t)\|_{L^r(\mathbb{R}_+^3)}+\|\nabla \partial_td_1(t)\|_{L^r(\mathbb{R}_+^3)}\big).
\end{align}
So it is sufficient to establish the decay of $\|\nabla \partial_td_1(t)\|_{L^r(\mathbb{R}_+^3)}$ in $(\ref{3.49})$.

Next we try to obtain the decay of $\|\nabla^6 d(t)\|_{L^{r}(\mathbb{R}_+^3)}$. Similarly, using the elliptic estimates, we find for $t>1$,
\begin{equation}\label{3.58}
\|\nabla^{6}d(t)\|_{L^r(\mathbb{R}_+^3)}\leq C\big(\|\nabla^4(u\cdot\nabla d)(t)\|_{L^r(\mathbb{R}_+^3)}+\|\nabla^4\partial_td(t)\|_{L^r(\mathbb{R}_+^3)}+\|\nabla^4( |\nabla d|^2d)(t)\|_{L^r(\mathbb{R}_+^3)}\big).
\end{equation}
Applying the elliptic estimates to $d_1$ system again, we get for $t>1$,
\begin{align}\label{3.59}
\|\nabla^{4}d_1(t)\|_{L^r(\mathbb{R}_+^3)}\leq &C\big(\|\nabla^2(u_1\cdot\nabla d+u\cdot\nabla d_1)(t)\|_{L^r(\mathbb{R}_+^3)}\nonumber\\
&+\|\nabla^2(|\nabla d|^2d_1)(t)\|_{L^r(\mathbb{R}_+^3)}+\|\nabla^2((\nabla d:\nabla d_1) d)(t)\|_{L^r(\mathbb{R}_+^3)}+\|\nabla^2 \partial_td_1(t)\|_{L^r(\mathbb{R}_+^3)}\big).
\end{align}
In order to get the estimate of $\|\nabla^6 d(t)\|_{L^{r}(\mathbb{R}_+^3)}$, it is sufficient to establish the decay of $\|\nabla^2 \partial_td_1(t)\|_{L^r(\mathbb{R}_+^3)}$ in terms of $(\ref{3.58})$ and $(\ref{3.59})$.

\textbf{Step2:}\ The estimate of $\|\nabla^2u_1\|_{L^r(\mathbb{R}_+^3)}$.

Using the regularity estimate $(\ref{3.44})$ with $\ell=2$, we get for $t>1$,
\begin{align*}
&\|\nabla^{4}u(t)\|_{L^r(\mathbb{R}_+^3)}+\|\nabla^{3}p(t)\|_{L^r(\mathbb{R}_+^3)}\nonumber\\
\leq &C(\|\nabla^2(u\cdot\nabla u)(t)\|_{L^r(\mathbb{R}_+^3)}+\|\nabla^2\partial_tu(t)\|_{L^r(\mathbb{R}_+^3)}+\|\nabla^2(\nabla \cdot(\nabla d\odot\nabla d))(t)\|_{L^r(\mathbb{R}_+^3)}).
\end{align*}
So in order to get the estimate of $\|\nabla^4 u(t)\|_{L^{r}(\mathbb{R}_+^3)}$, it is key to establish the decay of $\|\nabla^2\partial_tu(t)\|_{L^r(\mathbb{R}_+^3)}$.

Let $1<r<\infty$ and set $U(x,t):=u_1\cdot\nabla u+u\cdot\nabla u_1+ \nabla \cdot(\nabla d_1\odot\nabla d)+\nabla \cdot(\nabla d\odot\nabla d_1)$, $d_2(t)=\partial_{tt}d(t)$, $u_2(t)=\partial_{tt}u(t)$. Using $(\ref{3.28})$, $(\ref{3.36*})$, $(\ref{3.33})$, $(\ref{3.40})$, $(\ref{3.41})$, Lemmas $\ref{pro3.2}$-$\ref{le3.10}$ and Lemma $\ref{le4.1-}$, we find for $1\leq k\leq3$, $1\leq j\leq 2$, $0<\varepsilon<\frac12$ and $t\geq 2t_2$,
\begin{align}\label{4.2}
&\big\|\int_{\frac{t}{2}}^{t}\int_{\mathbb{R}_{+}^{3}}\partial_{x_k}\partial_{x_j}\mathcal{M}(x,y,t-s)\mathbb{P}U(y,s)\,dyds\big\|_{L^{r}(\mathbb{R}_{+}^{3})}\nonumber\\
&+\big\|\int_{\frac{t}{2}}^{t}\int_{\mathbb{R}_{+}^{3}}\partial_{x_k}\partial_{x_3}\mathcal{M}(x,y,t-s)\mathbb{P}U(y,s)\,dyds\big\|_{L^{r}(\mathbb{R}_{+}^{3})}\nonumber\\
\leq&C\int_{\frac{t}{2}}^{t}\|\partial_{x_k}G_{t-s}(\cdot)\|_{L^1(\mathbb{R}_+^3)}(\|\partial_3\mathbb{P}U(s)\|_{L^r(\mathbb{R}_+^3)}+\|\partial_j\mathbb{P}U(s)\|_{L^r(\mathbb{R}_+^3)})\,ds\nonumber\\
    &+C\int_{\frac{t}{2}}^{t}\|x_{3}^{2\varepsilon}\partial_{x_{k}}\partial_{x_{3}}G_{t-s}(\cdot)\|_{L^{1}(\mathbb{R}_{+}^{3})}\|y_{3}^{-2\varepsilon}\mathbb{P}U(s)\|_{L^{r}(\mathbb{R}_{+}^{3})}\,ds\nonumber\\
    &+C\int_{\frac{t}{2}}^t\big(\sup\limits_{x=(x',x_3)\in\mathbb{R}_+^3}\|(x_3+y_3)^{2\varepsilon}\partial_{x_k}\partial_{x_3}\mathcal{N}^*(x,y,t-s)\|_{L^1(\mathbb{R}_+^3)}\nonumber\\
    &+\sup\limits_{y=(y',y_3)\in\mathbb{R}_+^3}\|(x_3+y_3)^{2\varepsilon}\partial_{x_k}\partial_{x_3}\mathcal{N}^*(x,y,t-s)\|_{L^1(\mathbb{R}_+^3)}\big)\times\|z_3^{-2\varepsilon}\mathbb{P}U(\cdot,s)\|_{L^r(\mathbb{R}_+^3)}\,ds\nonumber\\
    &+C\int_{\frac{t}{2}}^t\big(\sup\limits_{\substack{x\in\mathbb{R}_+^3\\}}\|\partial_{x_k}\mathcal{N}^*(x,y,t-s)\|_{L^1(\mathbb{R}_+^3)}+\sup\limits_{\substack{y\in\mathbb{R}_+^3\\}}\|\partial_{x_k}\mathcal{N}^*(x,y,t-s)\|_{L^1(\mathbb{R}_+^3)}\big)
   \times\|\partial_{j}\mathbb{P}U(s)\|_{L^r(\mathbb{R}_+^3)}\,ds\nonumber\\
\leq &C\int_{\frac t2}^t(t-s)^{-1+\varepsilon}
    \big( \|u_1(s)\|_{L^{2r}(\mathbb{R}_+^3)}\|\nabla u(s)\|_{L^{2r}(\mathbb{R}_+^3)}
      +\|u(s)\|_{L^{2r}(\mathbb{R}_+^3)}\|\nabla u_1(s)\|_{L^{2r}(\mathbb{R}_+^3)} \nonumber\\
    &+\|u_1(s)\|_{L^{2r}(\mathbb{R}_+^3)}\| u(s)\|_{L^{2r}(\mathbb{R}_+^3)}
      +\|\nabla u(s)\|_{L^{2r}(\mathbb{R}_+^3)}\|\nabla u_1(s)\|_{L^{2r}(\mathbb{R}_+^3)} \nonumber\\
    &+\|\nabla ^2d_1(s)\|_{L^{2r}(\mathbb{R}_+^3)}\|\nabla d(s)\|_{L^{2r}(\mathbb{R}_+^3)}
      +\|\nabla d_1(s)\|_{L^{2r}(\mathbb{R}_+^3)}\|\nabla^2 d(s)\|_{L^{2r}(\mathbb{R}_+^3)} \nonumber\\
    &+\|\nabla d_1(s)\|_{L^{2r}(\mathbb{R}_+^3)}\|\nabla d(s)\|_{L^{2r}(\mathbb{R}_+^3)}
      +\|\nabla^2 d_1(s)\|_{L^{2r}(\mathbb{R}_+^3)}\|\nabla^2 d(s)\|_{L^{2r}(\mathbb{R}_+^3)} \nonumber\\
    &+\|\nabla d_1(s)\|_{L^{2r}(\mathbb{R}_+^3)}\|\nabla^3 d(s)\|_{L^{2r}(\mathbb{R}_+^3)}+\| u_1(s)\|_{L^{2r}(\mathbb{R}_+^3)}\|\nabla^2 u(s)\|_{L^{2r}(\mathbb{R}_+^3)}\big)\,ds\nonumber\\
     &+\int_{\frac t2}^t(t-s)^{-1+\varepsilon}(\|\nabla d(s)\|_{L^{\infty}(\mathbb{R}_+^3)}\|\nabla^3 d_1(s)\|_{L^{r}(\mathbb{R}_+^3)}+\|u(s)\|_{L^{\infty}(\mathbb{R}_+^3)}\|\nabla ^2u_1(s)\|_{L^{r}(\mathbb{R}_+^3)})\,ds\nonumber\\
&+C\int_{\frac{t}{2}}^{t}(t-s)^{-\frac{1}{2}}\big(\|u(s)\|_{L^{2r}(\mathbb{R}_{+}^{3})}\|\nabla u_1(s)\|_{L^{2r}(\mathbb{R}_{+}^{3})}   \nonumber\\
&+\|u_1(s)\|_{L^{2r}(\mathbb{R}_{+}^{3})}\|\nabla u(s)\|_{L^{2r}(\mathbb{R}_{+}^{3})}
 +\|\nabla u_1(s)\|_{L^{2r}(\mathbb{R}_{+}^{3})}\|\nabla u(s)\|_{L^{2r}(\mathbb{R}_{+}^{3})} \nonumber\\
&+\|\nabla u_1(s)\|_{L^{2r}(\mathbb{R}_{+}^{3})}\|\nabla^2 u(s)\|_{L^{2r}(\mathbb{R}_{+}^{3})}
 +\|u_1(s)\|_{L^{2r}(\mathbb{R}_{+}^{3})}\|\nabla^2 u(s)\|_{L^{2r}(\mathbb{R}_{+}^{3})}\nonumber\\
 &+\|\nabla^2 d_1(s)\|_{L^{2r}(\mathbb{R}_{+}^{3})}\|\nabla d(s)\|_{L^{2r}(\mathbb{R}_{+}^{3})}+\|\nabla d_1(s)\|_{L^{2r}(\mathbb{R}_{+}^{3})}\|\nabla^2 d(s)\|_{L^{2r}(\mathbb{R}_{+}^{3})} \nonumber\\
  &+\|\nabla^2 d_1(s)\|_{L^{2r}(\mathbb{R}_{+}^{3})}\|\nabla^3 d(s)\|_{L^{2r}(\mathbb{R}_{+}^{3})}
 +\|\nabla d_1(s)\|_{L^{2r}(\mathbb{R}_{+}^{3})}\|\nabla^4 d(s)\|_{L^{2r}(\mathbb{R}_{+}^{3})} \nonumber\\
 &+\|\nabla d_1(s)\|_{L^{2r}(\mathbb{R}_{+}^{3})}\|\nabla^3 d(s)\|_{L^{2r}(\mathbb{R}_{+}^{3})}+\|\nabla ^2d_1(s)\|_{L^{2r}(\mathbb{R}_{+}^{3})}\|\nabla^2 d(s)\|_{L^{2r}(\mathbb{R}_{+}^{3})}\big )\,ds \nonumber\\
 &+C\int_{\frac{t}{2}}^{t}(t-s)^{-\frac{1}{2}}\big((\|u(s)\|_{L^{\infty}(\mathbb{R}_{+}^{3})}+\|\nabla u(s)\|_{L^{\infty}(\mathbb{R}_{+}^{3})}) \|\nabla^2u_1(s)\|_{L^{r}(\mathbb{R}_{+}^{3})}\nonumber\\
&+(\|\nabla d(s)\|_{L^{\infty}(\mathbb{R}_{+}^{3})}+\|\nabla^2 d(s)\|_{L^{\infty}(\mathbb{R}_{+}^{3})}) \|\nabla^3d_1(s)\|_{L^{r}(\mathbb{R}_{+}^{3})}+\|\nabla d(s)\|_{L^{\infty}(\mathbb{R}_{+}^{3})}\|\nabla^4 d_1(s)\|_{L^{r}(\mathbb{R}_{+}^{3})}\big)\,ds\nonumber\\
\leq&C\int_{\frac{t}{2}}^{t}(t-s)^{-\frac{1}{2}}s^{-3-\frac{3}{2}(1-\frac{1}{r})}ds+C\int_{\frac{t}{2}}^t(t-s)^{-1+\varepsilon}s^{-\frac52-\frac{3}{2}(1-\frac{1}{r})}\,ds\nonumber\\
     &+C\int_{\frac{t}{2}}^{t}(t-s)^{-\frac{1}{2}}s^{-\frac{3}{2}}\big(\|\nabla d_2(s)\|_{L^{r}(\mathbb{R}_{+}^{3})}+\|\nabla^2 u_1(s)\|_{L^{r}(\mathbb{R}_{+}^{3})}+
     \|\nabla^2 d_2(s)\|_{L^{r}(\mathbb{R}_{+}^{3})}  \big )\,ds,
\end{align}
where we have used the estimates $(\ref{3.49})$ and $(\ref{3.59})$ in bounding the final inequality.

\textbf{Step3:}\ The estimate of $\|\nabla \partial_td_1\|_{L^r(\mathbb{R}_+^3)}$.

From $(\ref{d1})$, we find for $t>1$,
\begin{align*}
	\left\{
	\begin{array}{lll}
		\partial_t d_2+u_2\cdot\nabla d+2u_1\cdot\nabla d_1+u\cdot\nabla d_2-\Delta d_2-|\nabla d|^2d_2 \\
       \quad\quad=2|\nabla d_1|^2d+2(\nabla d:\nabla d_2) d+4(\nabla d:\nabla d_1) d_1 &\text{in}\quad\mathbb{R}_+^3\times(1,\infty),\\
       \frac{\partial d_2}{\partial x_3}=0&\text{on}\quad\partial\mathbb{R}_+^3\times(1,\infty),\\
		d_2(x,1)=\partial_{tt}d(1)&\text{in}\quad\mathbb{R}_+^3.\\
	\end{array}
	\right.
\end{align*}
Furthermore, it holds for $t>2$ that
\begin{align}\label{d2-}
d_2(x,t)=&\int_{\mathbb{R}_+^3}\mathcal{W}(x,y,\frac{t}{2})d_2(y,\frac{t}{2})\,dy-\int_{\frac{t}{2}}^t\int_{\mathbb{R}_+^3}\mathcal{W}(x,y,t-s)\big(u_2\cdot\nabla d+2u_1\cdot\nabla d_1\nonumber\\
&+u\cdot\nabla d_2-|\nabla d|^2d_2-2|\nabla d_1|^2d-2(\nabla d:\nabla d_2) d-4(\nabla d:\nabla d_1) d_1\big)(y,s)\,dyds.
\end{align}
From $(\ref{d1})$, $(\ref{3.28})$, $(\ref{3.41})$ and Lemma \ref{le4.1-}, we find for $t\geq 2t_2$,
\begin{align}\label{3.53}
&\|d_2(t)\|_{L^{r}(\mathbb{R}_+^3)}=\|\partial_t d_1(t)\|_{L^{r}(\mathbb{R}_+^3)}\nonumber\\
\leq&\|(u_1\cdot \nabla d+u\cdot \nabla d_1)(t)\|_{L^{r}(\mathbb{R}_+^3)}
+\|\nabla^2  d_1(t)\|_{L^{r}(\mathbb{R}_+^3)}+\|  (|\nabla d|^2d_1)(t)\|_{L^{r}(\mathbb{R}_+^3)}
+2\|(\nabla d:\nabla d_1)(t)\|_{L^{r}(\mathbb{R}_+^3)}\nonumber\\
\leq&Ct^{-2-\frac{3}{2}(1-\frac{1}{r})}.
\end{align}
Let $1<r<\infty$. Using $(\ref{3.24})$ and $(\ref{3.53})$ yields for each integer $m\geq1$ and $t\geq 2t_2$,
\begin{align}\label{4.9}
&\|\nabla_x^m\int_{\mathbb{R}_+^3}\mathcal{W}(\cdot,y,\frac{t}{2})d_2(y,\frac{t}{2})\,dy\|_{L^r(\mathbb{R}_+^3)}\nonumber\\
\leq &C\|\nabla^mG_t\|_{L^1(\mathbb{R}_+^3)}\|d_2(\frac t2)\|_{L^r(\mathbb{R}_+^3)}+C\|d_2(\frac t2)\|_{L^r(\mathbb{R}_+^3)}\nonumber\\
&\times\sum_{m=m_3+|m'|}\int_{\mathbb{R}_+^3}(x_3+\sqrt{t} )^{-m_3}\big(|x'|+x_3+\sqrt{t} \big)^{-3-|m'|}\,dx'dx_3\nonumber\\
\leq & Ct^{-2-\frac{m}{2}-\frac{3}{2}(1-\frac{1}{r})}.
\end{align}
Set $D_1(x,t):=u_2\cdot\nabla d+2u_1\cdot\nabla d_1
+u\cdot\nabla d_2-|\nabla d|^2d_2-2|\nabla d_1|^2d-2(\nabla d:\nabla d_2) d-4(\nabla d:\nabla d_1) d_1$.
 Using $(\ref{u1})$, we obtain for $t\geq 2 t_2$,
\begin{align}\label{3.54}
&\|\partial_{tt}u(t)\|_{L^{r}(\mathbb{R}_+^3)}=\|\partial_t u_1(t)\|_{L^{r}(\mathbb{R}_+^3)}\nonumber\\
\leq&\|\mathbb{P}(u_1\cdot \nabla u+u\cdot \nabla u_1)(t)\|_{L^{r}(\mathbb{R}_+^3)}
+\|\mathbb{A} u_1(t)\|_{L^{r}(\mathbb{R}_+^3)}+\|  \mathbb{P}(\nabla \cdot(\nabla d_1\odot\nabla d)+\nabla \cdot(\nabla d\odot\nabla d_1))(t)\|_{L^{r}(\mathbb{R}_+^3)}\nonumber\\
\leq&\|\nabla^2 u_1(t)\|_{L^{r}(\mathbb{R}_+^3)}
+\|u_1(t)\|_{L^{2r}(\mathbb{R}_+^3)}\|\nabla u(t)\|_{L^{2r}(\mathbb{R}_+^3)}
+\|u(t)\|_{L^{2r}(\mathbb{R}_+^3)}\|\nabla u_1(t)\|_{L^{2r}(\mathbb{R}_+^3)}\nonumber\\
&+\|\nabla^2 d_1(t)\|_{L^{2r}(\mathbb{R}_+^3)}\|\nabla d(t)\|_{L^{2r}(\mathbb{R}_+^3)}
+\|\nabla d_1(t)\|_{L^{2r}(\mathbb{R}_+^3)}\|\nabla^2 d(t)\|_{L^{2r}(\mathbb{R}_+^3)}.
\end{align}And then from $(\ref{3.54})$, combined with Lemma $\ref{le4.1-}$ and Lemma $\ref{pro3.1}$, we obtain for $t\geq 2t_2$,
\begin{align}\label{4.11}
&\big\|\int_{\frac{t}{2}}^{t}\int_{\mathbb{R}_{+}^{3}}\nabla_{x}\mathcal{W}(x,y,t-s)D_1(y,s)\,dyds\big\|_{L^{r}(\mathbb{R}_{+}^{3})}\nonumber\\
\leq & C\int_{\frac{t}{2}}^t\|\nabla G_{t-s}(\cdot)\|_{L^1(\mathbb{R}_+^3)}\|D_1(s)\|_{L^r(\mathbb{R}_+^3)}\,ds\nonumber\\
\leq &C\int_{\frac t2}^t(t-s)^{-\frac12}\big(\|\nabla d(s)\|_{L^{\infty }(\mathbb{R}_+^3)}\|u_2(s)\|_{L^{r}(\mathbb{R}_+^3)}\nonumber\\
     &+\| u(s)\|_{L^{\infty}(\mathbb{R}_+^3)}\|\nabla d_2(s)\|_{L^{r}(\mathbb{R}_+^3)}
     +\|\nabla d(s)\|_{L^{\infty }(\mathbb{R}_+^3)}\|\nabla d_2(s)\|_{L^{r}(\mathbb{R}_+^3)}\big)\,ds\nonumber\\
 &+C\int_{\frac t2}^t(t-s)^{-\frac12}\big(\|u_1(s)\|_{L^{2r}(\mathbb{R}_+^3)}\|\nabla d_1(s)\|_{L^{2r}(\mathbb{R}_+^3)}
    +\|\nabla d(s)\|_{L^{3r}(\mathbb{R}_+^3)}\|\nabla d_1(s)\|_{L^{3r}(\mathbb{R}_+^3)}\|d_1(s)\|_{L^{3r}(\mathbb{R}_+^3)}\nonumber\\
    &+\|\nabla d(s)\|_{L^{4r}(\mathbb{R}_+^3)}^2\| d_2(s)\|_{L^{2r}(\mathbb{R}_+^3)}
    +\|\nabla d_1(s)\|_{L^{4r}(\mathbb{R}_+^3)}^2\| d(s)\|_{L^{2r}(\mathbb{R}_+^3)}\big)\,ds\nonumber\\
\leq&C\int_{\frac{t}{2}}^{t}(t-s)^{-\frac{1}{2}}s^{-4-\frac{3}{2}(1-\frac{1}{r})}\,ds
     +C\int_{\frac{t}{2}}^{t}(t-s)^{-\frac{1}{2}}s^{-\frac{3}{2}}\big(\|\nabla d_2(s)\|_{L^{r}(\mathbb{R}_{+}^{3})}+\|\nabla^2 u_1(s)\|_{L^{r}(\mathbb{R}_{+}^{3})}
      \big )\,ds.
\end{align}

\textbf{Step4:}\ The estimate of $\|\nabla^2 \partial_td_1\|_{L^r(\mathbb{R}_+^3)}$.

Using Lemma $\ref{pro3.1}$, $(\ref{3.24})$, $(\ref{3.28})$, $(\ref{3.37})$-$(\ref{3.38*})$, $(\ref{3.40})$, $(\ref{3.41})$ and Lemma $\ref{le4.1-}$, we find for $1\leq k\leq3$, $1\leq j\leq 2$, $0<\varepsilon<\frac12$ and $t\geq 2t_2$,
\begin{align}\label{4.14}
&\big\|\int_{\frac{t}{2}}^{t}\int_{\mathbb{R}_{+}^{3}}\partial_{x_k}\partial_{x_j}\mathcal{W}(x,y,t-s)D_1(y,s)\,dyds\big\|_{L^{r}(\mathbb{R}_{+}^{3})}\nonumber\\
&+\big\|\int_{\frac{t}{2}}^{t}\int_{\mathbb{R}_{+}^{3}}\partial_{x_k}\partial_{x_3}\mathcal{W}(x,y,t-s)D_1(y,s)\,dyds\big\|_{L^{r}(\mathbb{R}_{+}^{3})}\nonumber\\
\leq&C\int_{\frac{t}{2}}^{t}\|\partial_{x_k}G_{t-s}\|_{L^1(\mathbb{R}_+^3)}(\|\partial_3D_1(s)\|_{L^r(\mathbb{R}_+^3)}+\|\partial_jD_1(s)\|_{L^r(\mathbb{R}_+^3)})\,ds\nonumber\\
&+C\int_{\frac{t}{2}}^{t}\|x_{3}^{2\varepsilon}\partial_{x_{k}}\partial_{x_{3}}G_{t-s}(\cdot)\|_{L^{1}(\mathbb{R}_{+}^{3})}\|z_{3}^{-2\varepsilon}D_1(\cdot,s)\|_{L^{r}(\mathbb{R}_{+}^{3})}\,ds\nonumber\\
\leq&C\int_{\frac{t}{2}}^{t}(t-s)^{-1+\varepsilon}\big(\|\nabla d(s)\|_{L^{\infty }(\mathbb{R}_+^3)}\|u_2(s)\|_{L^{r}(\mathbb{R}_+^3)}
       +\|u_1(s)\|_{L^{2r}(\mathbb{R}_+^3)}\|\nabla d_1(s)\|_{L^{2r }(\mathbb{R}_+^3)}\nonumber\\
        &+\|u(s)\|_{L^{\infty }(\mathbb{R}_+^3)}\|\nabla d_2(s)\|_{L^{r }(\mathbb{R}_+^3)}
         +\|\nabla d(s)\|_{L^{4r}(\mathbb{R}_+^3)}^2\| d_2(s)\|_{L^{2r }(\mathbb{R}_+^3)}
         +\|\nabla d_1(s)\|_{L^{2r}(\mathbb{R}_+^3 )}^2\nonumber\\
           &+\|\nabla d(s)\|_{L^{\infty}(\mathbb{R}_+^3)}\|\nabla d_2(s)\|_{L^{r }(\mathbb{R}_+^3)}
            +\|\nabla d(s)\|_{L^{3r }(\mathbb{R}_+^3 )}\|\nabla d_1(s)\|_{L^{3r }(\mathbb{R}_+^3)}\| d_1(s)\|_{L^{3r }(\mathbb{R}_+^3)}\big)\,ds\nonumber\\
     &+C\int_{\frac{t}{2}}^{t}(t-s)^{-\frac12}\big(\|\nabla d(s)\|_{L^{\infty }(\mathbb{R}_+^3)}\|\nabla u_2(s)\|_{L^{r}(\mathbb{R}_+^3)}
        +\|\nabla^2 d(s)\|_{L^{\infty }(\mathbb{R}_+^3)}\|u_2(s)\|_{L^{r}(\mathbb{R}_+^3)}\nonumber\\
         &+\|\nabla u_1(s)\|_{L^{2r}(\mathbb{R}_+^3)}\|\nabla d_1(s)\|_{L^{2r }(\mathbb{R}_+^3)}
      +\| u_1(s)\|_{L^{2r}(\mathbb{R}_+^3)}\|\nabla^2 d_1(s)\|_{L^{2r }(\mathbb{R}_+^3)}\nonumber\\
        &+\|u(s)\|_{L^{\infty }(\mathbb{R}_+^3)}\|\nabla^2 d_2(s)\|_{L^{r}(\mathbb{R}_+^3)}
      +\|\nabla u(s)\|_{L^{\infty }(\mathbb{R}_+^3)}\|\nabla d_2(s)\|_{L^{r}(\mathbb{R}_+^3)}\nonumber\\
      &+\|\nabla^2 d(s)\|_{L^{\infty }(\mathbb{R}_+^3)}\|\nabla d_2(s)\|_{L^{r}(\mathbb{R}_+^3)}
      +\|\nabla d(s)\|_{L^{\infty }(\mathbb{R}_+^3)}\|\nabla^2 d_2(s)\|_{L^{r}(\mathbb{R}_+^3)}\nonumber\\
      &+\||\nabla d(s)|^2\|_{L^{\infty }(\mathbb{R}_+^3)}\|\nabla d_2(s)\|_{L^{r}(\mathbb{R}_+^3)}
      +\|\nabla d(s)\|_{L^{3r}(\mathbb{R}_+^3)}\|\nabla^2 d(s)\|_{L^{3r }(\mathbb{R}_+^3)}\| d_2(s)\|_{L^{3r}(\mathbb{R}_+^3)}\nonumber\\
      &+\|\nabla d_1(s)\|_{L^{3r}(\mathbb{R}_+^3)}\|\nabla^2 d(s)\|_{L^{3r }(\mathbb{R}_+^3)}\| d_1(s)\|_{L^{3r}(\mathbb{R}_+^3)}
      +\|\nabla d(s)\|_{L^{3r}(\mathbb{R}_+^3)}\|\nabla^2 d_1(s)\|_{L^{3r }(\mathbb{R}_+^3)}\| d_1(s)\|_{L^{3r}(\mathbb{R}_+^3)}\nonumber\\
      &+\|\nabla d_1(s)\|_{L^{4r}(\mathbb{R}_+^3)}^2\|\nabla d(s)\|_{L^{2r }(\mathbb{R}_+^3)}
      +\|\nabla d_1(s)\|_{L^{2r}(\mathbb{R}_+^3)}\|\nabla^2 d_1(s)\|_{L^{2r }(\mathbb{R}_+^3)}
      \big)\,ds\nonumber\\
\leq&C\int_{\frac{t}{2}}^{t}(t-s)^{-\frac{1}{2}}s^{-\frac{9}{2}-\frac{3}{2}(1-\frac{1}{r})}ds+C\int_{\frac{t}{2}}^t(t-s)^{-1+\varepsilon}s^{-4-\frac{3}{2}(1-\frac{1}{r})}\,ds\nonumber\\
     &+C\int_{\frac{t}{2}}^{t}(t-s)^{-\frac{1}{2}}s^{-\frac{3}{2}}\big(\|\nabla d_2(s)\|_{L^{r}(\mathbb{R}_{+}^{3})}+\|\nabla u_2(s)\|_{L^{r}(\mathbb{R}_{+}^{3})}+\|\nabla^2 u_1(s)\|_{L^{r}(\mathbb{R}_{+}^{3})}+
     \|\nabla^2 d_2(s)\|_{L^{r}(\mathbb{R}_{+}^{3})}  \big )\,ds,
\end{align}where we have applied the estimate $(\ref{3.54})$ to bound the final inequality.

\textbf{Step5:}\ The estimate of $\|\nabla \partial_tu_1\|_{L^r(\mathbb{R}_+^3)}$.

Now we try to get the estimate of $\|\nabla^5u(t)\|_{L^r(\mathbb{R}_+^3)}$. Using the regularity estimate $(\ref{3.44})$ with $\ell=3$, we get for $t>1$,
\begin{align*}
&\|\nabla^{5}u(t)\|_{L^r(\mathbb{R}_+^3)}+\|\nabla^{4}p(t)\|_{L^r(\mathbb{R}_+^3)}\nonumber\\
\leq &C(\|\nabla^3(u\cdot\nabla u)(t)\|_{L^r(\mathbb{R}_+^3)}+\|\nabla^3\partial_tu(t)\|_{L^r(\mathbb{R}_+^3)}+\|\nabla^3(\nabla \cdot(\nabla d\odot\nabla d))(t)\|_{L^r(\mathbb{R}_+^3)}).
\end{align*}
Applying the regularity estimate $(\ref{3.43})$ to problem $(\ref{u1})$, we obtain for $t>1$,
\begin{align}\label{3.62}
\|\nabla^3u_1(t)\|_{L^r(\mathbb{R}_+^3)}\leq& C\big(\|\nabla(u_1\cdot\nabla u+u\cdot\nabla u_1)(t)\|_{L^r(\mathbb{R}_+^3)}
+\|\nabla(\nabla\cdot(\nabla d_1\odot\nabla d))(t)\|_{L^r(\mathbb{R}_+^3)}\nonumber\\
&+\|\nabla(\nabla\cdot(\nabla d\odot\nabla d_1))(t)\|_{L^r(\mathbb{R}_+^3)}+\|\nabla\partial_tu_1(t)\|_{L^r(\mathbb{R}_+^3)}\big).
\end{align}
So it is sufficient to establish the decay of $\|\nabla\partial_tu_1(t)\|_{L^r(\mathbb{R}_+^3)}$ in $(\ref{3.62})$. From $(\ref{u1})$, we find for $t>1$,
\begin{align*}
	\left\{
	\begin{array}{lll}
		\partial_tu_2-\Delta u_2+u_2\cdot\nabla u+2u_1\cdot\nabla u_1+u\cdot\nabla u_2+\nabla \partial_{tt}p\\
\quad\quad+ \nabla \cdot(\nabla d_2\odot\nabla d)+\nabla \cdot(\nabla d\odot\nabla d_2)+2 \nabla \cdot(\nabla d_1\odot\nabla d_1)=0&\text{in}\quad\mathbb{R}_+^3\times(1,\infty),\\
       \nabla\cdot u_2=0&\text{in}\quad\mathbb{R}_+^3\times(1,\infty),\\
       u_2(x,t)=0&\text{on}\quad\partial\mathbb{R}_+^3\times(1,\infty),\\
		u_2(x,1)=\partial_{tt}u(1)&\text{in}\quad\mathbb{R}_+^3.\\
	\end{array}
	\right.
\end{align*}
Set $U_1(x,t):=u_2\cdot\nabla u+2u_1\cdot\nabla u_1+u\cdot\nabla u_2+\nabla \cdot(\nabla d\odot\nabla d_2)+ \nabla \cdot(\nabla d_2\odot\nabla d)+2 \nabla \cdot(\nabla d_1\odot\nabla d_1)$, then it holds for $t>2$,
\begin{equation}\label{u2-}
u_2(x,t)=\int_{\mathbb{R}_+^3}\mathcal{M}(x,y,\frac{t}{2})u_2(y,\frac{t}{2})dy-\int_{\frac{t}{2}}^t\int_{\mathbb{R}_+^3}\mathcal{M}(x,y,t-s)\mathbb{P}U_1(y,s)\,dyds.
\end{equation}
Let $1<r<\infty$. Using $(\ref{3.28})$, $(\ref{3.36*})$, $(\ref{3.33})$, $(\ref{3.40})$, $(\ref{3.41})$, Lemma $\ref{le4.1-}$ and Lemma $\ref{pro3.1}$, we find for each integer $m\geq1$ and $t\geq2t_2$,
\begin{align}\label{4.19-}
&\|\nabla_x^m\int_{\mathbb{R}_+^3}\mathcal{M}(\cdot,y,\frac{t}{2})u_2(y,\frac{t}{2})\,dy\|_{L^r(\mathbb{R}_+^3)}\nonumber\\
\leq &C\|\nabla^mG_t\|_{L^1(\mathbb{R}_+^3)}\|u_2(\frac t2)\|_{L^r(\mathbb{R}_+^3)}+C\|u_2(\frac t2)\|_{L^r(\mathbb{R}_+^3)}\times\nonumber\\
&\sum_{m=m_3+|m'|}\int_{\mathbb{R}_+^3}(x_3+\sqrt{t} )^{-m_3}\big(|x'|+x_3+\sqrt{t} \big)^{-3-|m'|}\,dx'dx_3\nonumber\\
\leq & Ct^{-\frac{m}{2}}\|u_2( t)\|_{L^r(\mathbb{R}_+^3)};
\end{align}
and
\begin{align}\label{4.20}
&\big\|\int_{\frac{t}{2}}^{t}\int_{\mathbb{R}_{+}^{3}}\nabla_{x}\mathcal{M}(x,y,t-s)\mathbb{P}U_1(y,s)\,dyds\big\|_{L^{r}(\mathbb{R}_{+}^{3})}\nonumber\\
\leq & C\int_{\frac{t}{2}}^t\|\nabla G_{t-s}\|_{L^1(\mathbb{R}_+^3)}\|\mathbb{P}U_1(s)\|_{L^r(\mathbb{R}_+^3)}\,ds\nonumber\\
&+C\int_{\frac{t}{2}}^t\big(\sup\limits_{\substack{x\in\mathbb{R}_+^3\\}}\|\nabla_x\mathcal{N}^{*}(x,\cdot,t-s)\|_{L^1(\mathbb{R}_+^3)}+\sup\limits_{\substack{y\in\mathbb{R}_+^3\\}}\|\nabla_x\mathcal{N}^{*}(\cdot,y,t-s)\|_{L^1(\mathbb{R}_+^3)}\big)
\times\|\mathbb{P}U_1(s)\|_{L^r(\mathbb{R}_+^3)}\,ds\nonumber\\
\leq &C\int_{\frac t2}^t(t-s)^{-\frac12}\big(\|u_1(s)\|_{L^{2r}(\mathbb{R}_+^3)}\|\nabla u_1(s)\|_{L^{2r}(\mathbb{R}_+^3)}+\|\nabla^2 d_1(s)\|_{L^{2r}(\mathbb{R}_+^3)}\|\nabla d_1(s)\|_{L^{2r}(\mathbb{R}_+^3)}\big)\,ds \nonumber\\
&+\int_{\frac t2}^t(t-s)^{-\frac12}\big(\|\nabla u(s)\|_{L^{\infty}(\mathbb{R}_+^3)}\|u_2\|_{L^{r}(\mathbb{R}_+^3)}
      +\|u(s)\|_{L^{\infty}(\mathbb{R}_+^3)}\|\nabla u_2(s)\|_{L^{r}(\mathbb{R}_+^3)}\nonumber\\
      &+\|\nabla d(s)\|_{L^{\infty}(\mathbb{R}_+^3)}\|\nabla^2 d_2(s)\|_{L^{r}(\mathbb{R}_+^3)}
      +\|\nabla^2 d(s)\|_{L^{\infty}(\mathbb{R}_+^3)}\|\nabla d_2(s)\|_{L^{r}(\mathbb{R}_+^3)}
      \big)\,ds     \nonumber\\
 \leq &C\int_{\frac t2}^t(t-s)^{-\frac12}s^{-4-\frac{3}{2}(1-\frac{1}{r})}\,ds
      +C\int_{\frac{t}{2}}^{t}(t-s)^{-\frac{1}{2}}s^{-\frac{3}{2}}\big(\|\nabla d_2(s)\|_{L^{r}(\mathbb{R}_{+}^{3})}\nonumber\\
      &+\|\nabla u_2(s)\|_{L^{r}(\mathbb{R}_{+}^{3})}
      +\|\nabla^2 u_1(s)\|_{L^{r}(\mathbb{R}_{+}^{3})}+
     \|\nabla^2 d_2(s)\|_{L^{r}(\mathbb{R}_{+}^{3})}  \big )\,ds,
\end{align}
where we have used the estimate $(\ref{3.54})$ in bounding the final inequality.

\textbf{Step6:}\ The estimates of $\|\nabla^4 u(t)\|_{L^{r}(\mathbb{R}_{+}^{3})}$, $\|\nabla^5 u(t)\|_{L^{r}(\mathbb{R}_{+}^{3})}$, $\|\nabla^5 d(t)\|_{L^{r}(\mathbb{R}_{+}^{3})}$ and $\|\nabla^6 d(t)\|_{L^{r}(\mathbb{R}_{+}^{3})}$.

Set $f_3(t)=\sup\limits_{0<s\leq t}[s^{2+\frac{3}{2}(1-\frac{1}{r})}(\|\nabla d_2(s)\|_{L^{r}(\mathbb{R}_{+}^{3})}+\|\nabla u_2(s)\|_{L^{r}(\mathbb{R}_{+}^{3})}+\|\nabla^2 u_1(s)\|_{L^{r}(\mathbb{R}_{+}^{3})}+
     \|\nabla^2 d_2(s)\|_{L^{r}(\mathbb{R}_{+}^{3})})] $.
By $(\ref{3.54})$, $(\ref{4.19-})$ and Lemma $\ref{le4.1-}$, we get for $t\geq 2t_2$,
\begin{align}\label{4.18}
&\|\nabla_x\int_{\mathbb{R}_+^3}\mathcal{M}(\cdot,y,\frac{t}{2})u_2(y,\frac{t}{2})\,dy\|_{L^r(\mathbb{R}_+^3)}\nonumber\\
\leq & Ct^{-\frac{1}{2}}\|u_2(t)\|_{L^r(\mathbb{R}_+^3)}\nonumber\\
\leq & Ct^{-\frac{1}{2}}(\|\nabla^2u_1(t)\|_{L^r(\mathbb{R}_+^3)}+t^{-3-\frac32(1-\frac1r)}).
\end{align}
Using $(\ref{4.2})$, $(\ref{d2-})$, $(\ref{4.9})$, $(\ref{4.11})$, $(\ref{4.14})$, $(\ref{u2-})$-$(\ref{4.18})$, $(\ref{u1-})$, $(\ref{3.35})$, we conclude for $t\geq 2t_2$,
\begin{equation}\label{4.19}
t^{2+\frac{3}{2}(1-\frac{1}{r})}\big(\|\nabla d_2(s)\|_{L^{r}(\mathbb{R}_{+}^{3})}+\|\nabla u_2(s)\|_{L^{r}(\mathbb{R}_{+}^{3})}+\|\nabla^2 u_1(s)\|_{L^{r}(\mathbb{R}_{+}^{3})}+
     \|\nabla^2 d_2(s)\|_{L^{r}(\mathbb{R}_{+}^{3})}  \big )\leq C+C_3t^{-1}f_3(t).
\end{equation}
Note that there exists $t_3\geq2t_2$ such that $C_3t_3^{-1}\leq \frac12$. Whence $(\ref{4.19})$ yields for $t\geq t_3$ that $f_3(t)\leq 2C$,
which implies for $1<r<\infty$ and $t\geq t_3$,
\begin{equation}\label{3.68}
\|\nabla d_2(s)\|_{L^{r}(\mathbb{R}_{+}^{3})}+\|\nabla u_2(s)\|_{L^{r}(\mathbb{R}_{+}^{3})}+\|\nabla^2 u_1(s)\|_{L^{r}(\mathbb{R}_{+}^{3})}+
     \|\nabla^2 d_2(s)\|_{L^{r}(\mathbb{R}_{+}^{3})}\leq Ct^{-{2}-\frac{3}{2}(1-\frac{1}{r})}.
\end{equation}
Using $(\ref{3.68})$ and running the same argument as in \textbf{Step1-6}, we obtain for $1<r<\infty$ and $t\geq t_3$,
\begin{align*}
&\|\nabla^4 u(t)\|_{L^{r}(\mathbb{R}_{+}^{3})}+\|\nabla^3p(t)\|_{L^r(\mathbb{R}_+^3)}\leq Ct^{-2-\frac{3}{2}(1-\frac{1}{r})},\\
&\|\nabla^5 u(t)\|_{L^{r}(\mathbb{R}_{+}^{3})}+\|\nabla^4p(t)\|_{L^r(\mathbb{R}_+^3)}\leq Ct^{-\frac52-\frac{3}{2}(1-\frac{1}{r})},\\
&\|\nabla^5 d(t)\|_{L^{r}(\mathbb{R}_{+}^{3})}\leq Ct^{-\frac52-\frac{3}{2}(1-\frac{1}{r})},\\
&\|\nabla^6 d(t)\|_{L^{r}(\mathbb{R}_{+}^{3})}\leq Ct^{-3-\frac{3}{2}(1-\frac{1}{r})}.
\end{align*}
\end{proof}

\begin{Lemma}\label{4.3*}
Assume $u_0, d_0-e_3\in L^1(\mathbb{R}_+^{3})$. Then for every $k\geq1$, there exists some number $t(k)>0$ such that for $t\geq t(k)$ and $1<r\leq\infty$, the strong solution $(u,d,p)$ of $(\ref{1.1-main})$-$(\ref{1.2})$ obtained in Proposition $\ref{solu1}$ satisfies
\begin{align*}
\|\nabla^{2+k}u(t)\|_{L^r(\mathbb{R}_+^3)}+\|\nabla^{1+k}&p(t)\|_{L^r(\mathbb{R}_+^3)}+\|\nabla^{2+k}d(t)\|_{L^r(\mathbb{R}_+^3)}\leq Ct^{-\frac{2+k}2-\frac 32(1-\frac1r)}.
\end{align*}
\end{Lemma}
\begin{proof}
Repeating the proofs of Lemma $\ref{le3.8}$ and Lemma $\ref{le3.9}$, it can be easily shown that for every integer $k\geq1$, there exists $t(k)>0$ independent of $t$, such that for $1<r<\infty$ and $t\geq t(k)$,
\begin{align*}
\|\nabla^{2+k}u(t)\|_{L^r(\mathbb{R}_+^3)}+\|\nabla^{1+k}&p(t)\|_{L^r(\mathbb{R}_+^3)}+\|\nabla^{2+k}d(t)\|_{L^r(\mathbb{R}_+^3)}\leq Ct^{-\frac{2+k}2-\frac 32(1-\frac1r)}.
\end{align*}
By virtue of $(\ref{gn-})$, we can obtain that for every $k\geq1$ and $t\geq t(k)$,
\begin{align*}
&\|(\nabla^{2+k}u,\nabla^{2+k}d)(t)\|_{L^\infty(\mathbb{R}_+^3)}+\|\nabla^{1+k}p(t)\|_{L^\infty(\mathbb{R}_+^3)}\\
\leq& C\|\nabla^{2+k}u(t)\|_{L^{6}(\mathbb{R}_+^3)}^{\frac{1}{2}}\|\nabla^{3+k}u(t)\|_{L^{6}(\mathbb{R}_+^3)}^{\frac{1}{2}}\\
&+ C\|\nabla^{2+k}d(t)\|_{L^{6}(\mathbb{R}_+^3)}^{\frac{1}{2}}\|\nabla^{3+k}d(t)\|_{L^{6}(\mathbb{R}_+^3)}^{\frac{1}{2}}\\
&+ C\|\nabla^{1+k}p(t)\|_{L^{6}(\mathbb{R}_+^3)}^{\frac{1}{2}}\|\nabla^{2+k}p(t)\|_{L^{6}(\mathbb{R}_+^3)}^{\frac{1}{2}}\\
\leq& C t^{-\frac{2+k}2-\frac32}.
\end{align*}
Hence, we finish the proof of Theorem $\ref{th1.3-main}$.
\end{proof}

\section{Proof of Theorem $\ref{th1.3}$ }\label{sec4}
In this section, we shall prove the decay rates of the solution stated in Proposition $\ref{solu1}$ under the additional assumption that the initial datum $u_0$ satisfies $\|x_3u_0\|_{L^1(\mathbb{R}_+^3)}<\infty$.
\begin{Lemma}\label{le4.1}
Under the assumptions of Theorem \ref{th1.3}, if $(u,d)$ is the global strong solution obtained by Proposition $\ref{solu1}$, then
\begin{align*}
&\|\nabla u(t)\|_{L^\infty(\mathbb{R}_+^3)}\leq Ct^{-\frac52},\\
&\|\nabla u(t)\|_{L^p(\mathbb{R}_+^3)}\leq Ct^{-1-\frac32(1-\frac1p)}
\end{align*}
hold for any $t>0$, $p\in(1,\infty)$.
\end{Lemma}
\begin{proof}
Using $(\ref{uB2})_1$, Lemma $\ref{le2.1}$, Lemma $\ref{le2.2}$, Lemma $\ref{le4.1-}$ and the second part of Lemma $\ref{le2.3}$, one has for any $t>0$,
\begin{align}\label{5.1}
&\|\nabla u(t)\|_{L^\infty(\mathbb{R}_+^3)}\nonumber\\
\leq&Ct^{-\frac54}\|u(\frac{t}{2})\|_{L^2(\mathbb{R}_+^3)}+C\int_{\frac{t}{2}}^{t}(t-s)^{-\frac12-\frac{3}{2r}}\|\big(u\cdot\nabla u+\nabla\cdot(\nabla d\odot\nabla d)\big)(s)\|_{L^r(\mathbb{R}_+^3)}\,{d}s\nonumber\\
\leq& C t^{-\frac52}+C\int_{\frac t2}^t(t-s)^{-\frac12-\frac{3}{2r}}\Big(\|u(s)\|_{L^{r_1}(\mathbb{R}_+^3)}\|\nabla u(s)\|_{L^{r_2}(\mathbb{R}_+^3)}+\|\nabla d(s)\|_{L^{r_1}(\mathbb{R}_+^3)}\|\nabla^2d(s)\|_{L^{r_2}(\mathbb{R}_+^3)}\Big)\,{d}s\nonumber\\
\leq&Ct^{-\frac52}+C\int_{\frac{t}{2}}^{t}(t-s)^{-\frac12-\frac{3}{2r}}\Big(s^{-\frac12-\frac 32(1-\frac1{r_1})}s^{-\frac12-\frac 32(1-\frac1{r_2})}+s^{-\frac12-\frac 32(1-\frac1{r_1})}s^{-1-\frac 32(1-\frac1{r_2})}\Big)\,{d}s\nonumber\\
\leq&Ct^{-\frac52},
\end{align}
where $\frac1r=\frac1{r_1}+\frac1{r_2}$, $r\in(3,\infty)$, $r_1,{r_2}\in(1,\infty)$.

The conclusion now follows from $(\ref{5.1})$, Lemma $\ref{le2.3}$ and the interpolation inequality.
\end{proof}

\begin{Lemma}\label{le5.4}
Under the assumptions of Theorem \ref{th1.3}, if $(u,d,p)$ is the global strong solution obtained by Proposition $\ref{solu1}$, then it holds for any $t>0$ and $1<r<\infty$ that
$$\|\nabla^{2}u(t)\|_{{L^{r}(\mathbb{R}_{+}^{3})}}+\|\partial_{t}u(t)\|_{{L^{r}(\mathbb{R}_{+}^{3})}}+\|\nabla p(t)\|_{{L^{r}(\mathbb{R}_{+}^{3})}}\leq Ct^{-\frac32-\frac 32(1-\frac1r)}.$$
\end{Lemma}
\begin{proof}
Following the proofs of $(\ref{3.9})$-$(\ref{3.14})$, combined with Lemma $\ref{le2.3}$ and Lemma $\ref{le4.1}$, we obtain for all $1<r<\infty$ and any $t>0$,
\begin{align}
&\|\mathbb{A}^\alpha u(t+h)-\mathbb{A}^\alpha u(t)\|_{L^r(\mathbb{R}_+^3)}+\|(-\Delta)^\alpha \nabla d(t+h)-(-\Delta)^\alpha  \nabla d(t)\|_{L^r(\mathbb{R}_+^3)}\nonumber\\
\leq& C\Big(h^\delta t^{-\alpha-\delta-2+\frac 3{2r}}+h^{1-\alpha}t^{-\frac92+\frac 3{2r}}+h^\delta t^{-\alpha-\delta-\frac72+\frac 3{2r}}\Big);\label{4.3}\\
&\|(-\Delta)^\alpha d(t+h)-(-\Delta)^\alpha  d(t)\|_{L^r(\mathbb{R}_+^3)}\leq C\Big(h^\delta t^{-\alpha-\delta-\frac 32+\frac 3{2r}}+h^{1-\alpha}t^{-4+\frac 3{2r}}+h^\delta t^{-\alpha-\delta-3+\frac 3{2r}}\Big)\label{5.4},
\end{align}
where $0<\alpha+\delta<1$, and $\alpha>0$, $\delta>0$.

From $(\ref{4.3})$, $(\ref{5.4})$, Lemma $\ref{le2.1}$, Lemma $\ref{le4.1-}$, Lemma $\ref{le4.1}$ and Lemmas $\ref{le2.2}$-$\ref{le2.3}$, we obtain for all $1<r<\infty$ and any $t>0$,
\begin{align*}
&\|\mathbb{A}u(t)\|_{{L^{r}(\mathbb{R}_{+}^{3})}}\nonumber\\
\leq& Ct^{-1}\|u{(\frac t4)}\|_{L^r(\mathbb{R}_+^3)}+C\|u(t)\|_{L^{r_1}(\mathbb{R}_+^3)}\|\nabla u(t)\|_{L^{r_2}(\mathbb{R}_+^3)}+C\|\nabla d(t)\|_{L^{r_1}(\mathbb{R}_+^3)}\|\nabla ^2d(t)\|_{L^{r_2}(\mathbb{R}_+^3)}\nonumber\\
 &+C\int_{\frac t4}^{\frac t2}(t-s)^{-1}\Big(\|u(s)\|_{L^{r_1}(\mathbb{R}_+^3)}\|\nabla u(s)\|_{L^{r_2}(\mathbb{R}_+^3)}+\|\nabla d(s)\|_{L^{r_1}(\mathbb{R}_+^3)}\|\nabla ^2d(s)\|_{L^{r_2}(\mathbb{R}_+^3)}\Big)\,ds\nonumber\\
&+C\int_{\frac t2}^t(t-s)^{-1}\|u(s)\|_{L^{r_1}(\mathbb{R}_+^3)}\|\nabla u(s)-\nabla u(t)\|_{L^{r_2}(\mathbb{R}_+^3)}\,ds\nonumber\\
&+C\int_{\frac t2}^t(t-s)^{-1-\frac 32(\frac1q-\frac1r)}\|(u(s)-u(t))\|_{L^{\frac{3q}{3-q}}(\mathbb{R}_+^3)}\|\nabla u(t)\|_{L^3(\mathbb{R}_+^3)}\,ds\nonumber\\
 &+ C\int_{\frac t2}^t(t-s)^{-1}\|\nabla^2 d(s)\|_{L^{r_1}(\mathbb{R}_+^3)}\|\nabla d(s)-\nabla d(t)\|_{L^{r_2}(\mathbb{R}_+^3)}\,ds\\
 &+C\int_{\frac t2}^t(t-s)^{-1}\|\nabla d(s)\|_{L^{r_1}(\mathbb{R}_+^3)}\|\nabla (\nabla d(s)-\nabla d(t))\|_{L^{r_2}(\mathbb{R}_+^3)}\,ds\\
\leq&Ct^{-\frac32-\frac 32(1-\frac1r)}+Ct^{-\frac12-\frac 32(1-\frac1{r_1})}t^{-1-\frac 32(1-\frac1{r_2})}\nonumber\\
&+  C\int_{\frac t4}^{\frac t2}(t-s)^{-1}(1+s)^{-\frac12-\frac 32(1-\frac1{r_1})}s^{-1-\frac 32(1-\frac1{r_2})}ds\nonumber\\
&+C\int_{\frac t2}^t\Big((t-s)^{\delta-1}s^{-\frac92-\delta+\frac 3{2r}}+(t-s)^{-\frac12}s^{-\frac{13}2+\frac 3{2r}}+(t-s)^{\delta-1}s^{-6-\delta+\frac 3{2r}}\Big)\,ds\nonumber\\
&+Ct^{-2}\int_{\frac t2}^t\Big((t-s)^{\delta-1-\frac 32(\frac1q-\frac1r)}s^{-\delta-\frac 52+\frac 3{2q}}\nonumber\\
&+(t-s)^{-\frac12-\frac 32(\frac1q-\frac1r)}s^{-\frac92+\frac 3{2q}}+(t-s)^{\delta-1-\frac 32(\frac1q-\frac1r)}s^{-4-\delta+\frac 3{2q}}\Big)\,ds\nonumber\\
\leq& Ct^{-\frac32-\frac 32(1-\frac1r)}.
\end{align*}
Here $\frac1r=\frac1{r_1}+\frac1{r_2}$, ${r_1}, {r_2}\in(1,\infty)$, and we take the number
\begin{align*}
q=
\left\{\begin{array}{ll}
\displaystyle		r,&\mathrm{~if~}r\in(1,3),\smallskip\\
\displaystyle 2,&\mathrm{~if~}r\in[3,\infty).  \smallskip
\end{array}\right.
\end{align*}

 Whence, for all $1<r<\infty$ and any $t>0$,
$$\|\nabla^2u(t)\|_{L^r(\mathbb{R}_+^3)}\leq C\|\mathbb{A}u(t)\|_{L^r(\mathbb{R}_+^3)}\leq Ct^{-\frac32-\frac 32(1-\frac1r)}.$$
Using the above estimates, combined with Lemma $\ref{le2.3}$, Lemma $\ref{le4.1-}$ and Lemma $\ref{le4.1}$, we obtain for $t>0$,
\begin{align*}
\|\partial_tu(t)\|_{L^r(\mathbb{R}_+^3)}\leq&\|\mathbb{A}u(t)\|_{L^r(\mathbb{R}_+^3)}+\|u(t)\|_{L^{r_1}(\mathbb{R}_+^3)}\|\nabla u(t)\|_{L^{r_2}(\mathbb{R}_+^3)}+\|\nabla d(t)\|_{L^{r_1}(\mathbb{R}_+^3)}\|\nabla^2 d(t)\|_{L^{r_2}(\mathbb{R}_+^3)}\\
\leq&Ct^{-\frac32-\frac 32(1-\frac1r)},
\end{align*}
and
\begin{align*}
\|\nabla p(t)\|_{L^r(\mathbb{R}_+^3)}\leq&\|\partial_tu(t)\|_{L^r(\mathbb{R}_+^3)}+\|\nabla^2u(t)\|_{L^r(\mathbb{R}_+^3)}+\|u(t)\|_{L^{r_1}(\mathbb{R}_+^3)}\|\nabla u(t)\|_{L^{r_2}(\mathbb{R}_+^3)}\\
&+\|\nabla d(t)\|_{L^{r_1}(\mathbb{R}_+^3)}\|\nabla^2 d(t)\|_{L^{r_2}(\mathbb{R}_+^3)}\\
\leq&Ct^{-\frac32-\frac 32(1-\frac1r)},
\end{align*}
where $\frac1r=\frac1{r_1}+\frac1{r_2}$, ${r_1}, {r_2}\in(1,\infty)$.
\end{proof}
\begin{Lemma}
Under the assumptions of Theorem \ref{th1.3}, if $(u,d)$ is the global strong solution obtained by Proposition $\ref{solu1}$, then
$$\|\nabla^2 u(t)\|_{L^\infty(\mathbb{R}_+^3)}\leq Ct^{-3}$$
holds for any $t>0$.
\end{Lemma}
\begin{proof}
Let $0<\theta<1$ and $\frac{3}{\theta}<r<\infty$. Following the proof of Lemma \ref{3.11}, we conclude for any $t>0$,
\begin{align*}
&\|\nabla^2 u(t)\|_{L^\infty(\mathbb{R}_+^3)}\\
\leq & \|\nabla^2 e^{-\frac t2\mathbb{A}}u(\frac t2)\|_{L^\infty(\mathbb{R}_+^3)}\\
&+\int_{\frac{t}{2}}^{t}\big\|\nabla^2e^{-(t-s)\mathbb{A}}\mathbb{P}\Big(u(s)\cdot\nabla u(s)+\nabla\cdot(\nabla d\odot\nabla d)(s)\Big)\big\|_{L^\infty(\mathbb{R}_+^3)}\,ds\\
\leq &Ct^{-1-\frac34}\| u(t)\|_{L^2(\mathbb{R}_+^3)}+C\int_{\frac t2}^{t}(t-s)^{-\frac12-\frac{3}{2r}}\Big(\|u\|_{L^{2r}(\mathbb{R}_+^3)}^2+\|\nabla ^2d\|_{L^{2r}(\mathbb{R}_+^3)}^2+\|\nabla u\|_{L^{2r}(\mathbb{R}_+^3)}^2+\|\nabla^2u\|_{L^{2r}(\mathbb{R}_+^3)}^2\\
&+\|\nabla d(t)\|_{L^{2r}(\mathbb{R}_+^3)}\|\nabla^4d(t)\|_{L^{2r}(\mathbb{R}_+^3)}+\|\nabla ^2d(t)\|_{L^{2r}(\mathbb{R}_+^3)}\|\nabla^3d(t)\|_{L^{2r}(\mathbb{R}_+^3)}\\
&+\|\nabla d(t)\|_{L^{2r}(\mathbb{R}_+^3)}\|\nabla^3d(t)\|_{L^{2r}(\mathbb{R}_+^3)}+\|\nabla d(t)\|_{L^{2r}(\mathbb{R}_+^3)}\|\nabla^2d(t)\|_{L^{2r}(\mathbb{R}_+^3)}\Big)\,{d}s\\
&+C\int_{\frac{t}{2}}^t(t-s)^{-1+\frac{\theta}{2}-\frac{3}{2r}}\Big(\|u(s)\|_{L^{2r}(\mathbb{R}_+^3)}^2+\|\nabla u(s)\|_{L^{2r}(\mathbb{R}_+^3)}^2+\|\nabla^2 u(s)\|_{L^{2r}(\mathbb{R}_+^3)}^2+\|\nabla d(s)\|_{L^{2r}(\mathbb{R}_+^3)}^2\\
&+\|\nabla ^2d(s)\|_{L^{2r}(\mathbb{R}_+^3)}^2+\|\nabla ^3d(s)\|_{L^{2r}(\mathbb{R}_+^3)}^2\Big)\,ds\\
\leq &Ct^{-3}+C\int_{\frac t2}^{t}(t-s)^{-\frac12-\frac{3}{2r}}s^{-1-3(1-\frac{1}{2r})}\,ds+C\int_{\frac{t}{2}}^t(t-s)^{-1+\frac{\theta}{2}-\frac{3}{2r}}s^{-1-3(1-\frac{1}{2r})}\,ds\\
\leq & Ct^{-3}.
\end{align*}
\end{proof}

\begin{Lemma}\label{le5.4-}
Under the assumptions of Theorem \ref{th1.3}, for every integer $k\geq1$, there exists $t(k)>0$ such that for $1<r\leq\infty$ and $t\geq t(k)$, the strong solution $(u,d,p)$ obtained in Proposition $\ref{solu1}$ satisfies
$$\|\nabla^{2+k}u(t)\|_{L^r(\mathbb{R}_+^3)}+\|\nabla^{1+k}p(t)\|_{L^r(\mathbb{R}_+^3)}\leq Ct^{-\frac12-\frac{2+k}2-\frac 32(1-\frac1r)}.$$
\end{Lemma}
\begin{proof}
Suppose $\|x_3u_0\|_{L^1(\mathbb{R}_+^3)}<\infty$. Using Lemma \ref{le2.3}, Lemma \ref{le4.1-}, Lemma \ref{4.3*} and Lemmas \ref{le4.1}-\ref{le5.4}, following the proofs of $(\ref{3.35})$, $(\ref{3.36})$ and $(\ref{3.38})$, we conclude for $m\geq1$ and $t>2$,
\begin{equation}\label{5.5}
\|\nabla_x^m\int_{\mathbb{R}_+^3}\mathcal{M}(\cdot,y,\frac{t}{2})u_1(y,\frac{t}{2})\,dy\|_{L^r(\mathbb{R}_+^3)}
\leq  Ct^{-\frac12-\frac{m+2}{2}-\frac{3}{2}(1-\frac{1}{r})};
\end{equation}

\begin{align}\label{5.6}
&\big\|\int_{\frac{t}{2}}^{t}\int_{\mathbb{R}_{+}^{3}}\nabla_{x}\mathcal{M}(x,y,t-s)\mathbb{P}U(y,s)\,dyds\big\|_{L^{r}(\mathbb{R}_{+}^{3})}\nonumber\\
\leq &C\int_{\frac{t}{2}}^{t}(t-s)^{-\frac{1}{2}}s^{-\frac{5}{2}-\frac{3}{2}-\frac{3}{2}(1-\frac{1}{r})}\,ds +C\tilde{f_2}(t)\int_{\frac{t}{2}}^t(t-s)^{-\frac{1}{2}}s^{-4-\frac{3}{2}(1-\frac{1}{r})}\,ds\nonumber\\
\leq &C t^{-\frac 72-\frac 32(1-\frac1r)}+C\tilde{f_2}(t)t^{-\frac{7}{2}-\frac{3}{2}(1-\frac{1}{r})};
\end{align}
and
\begin{align}\label{5.7}
&\big\|\int_{\frac{t}{2}}^{t}\int_{\mathbb{R}_{+}^{3}}\partial_{x_k}\partial_{x_j}\mathcal{W}(x,y,t-s)\big(u_1\cdot\nabla d+u\cdot\nabla d_1-|\nabla d|^2d_1-2(\nabla d:\nabla d_1) d\big)(y,s)\,dyds\big\|_{L^{r}(\mathbb{R}_{+}^{3})}\nonumber\\
&+\big\|\int_{\frac{t}{2}}^{t}\int_{\mathbb{R}_{+}^{3}}\partial_{x_k}\partial_{x_3}\mathcal{W}(x,y,t-s)\big(u_1\cdot\nabla d+u\cdot\nabla d_1-|\nabla d|^2d_1-2(\nabla d:\nabla d_1) d\big)(y,s)\,dyds\big\|_{L^{r}(\mathbb{R}_{+}^{3})}\nonumber\\
\leq&C\int_{\frac{t}{2}}^{t}(t-s)^{-\frac{1}{2}}s^{-\frac{5}{2}-\frac{3}{2}-\frac{3}{2}(1-\frac{1}{r})}ds+C\int_{\frac{t}{2}}^t(t-s)^{-1+\varepsilon}s^{-2-\frac{3}{2}-\frac{3}{2}(1-\frac{1}{r})}\,ds\nonumber\\
     &+C\int_{\frac{t}{2}}^{t}(t-s)^{-\frac{1}{2}}s^{-2}\big(\|\nabla u_1(s)\|_{L^{r}(\mathbb{R}_{+}^{3})}+\|\nabla^2 d_1(s)\|_{L^{r}(\mathbb{R}_{+}^{3})}  \big )\,ds\nonumber\\
\leq &C_\varepsilon t^{-3-\frac 32(1-\frac1r)}
+C\tilde{f_2}(t)t^{-\frac{7}{2}-\frac{3}{2}(1-\frac{1}{r})},
\end{align}
where $\tilde{f_2}(t)=\sup\limits_{0<s\leq t}[s^{2+\frac{3}{2}(1-\frac{1}{r})}(\|\nabla u_1(s)\|_{L^{r}(\mathbb{R}_{+}^{3})}+\|\nabla^2 d_1(s)\|_{L^{r}(\mathbb{R}_{+}^{3})})] $.

Putting $(\ref{5.5})$-$(\ref{5.7})$ together, following the proof of $(\ref{3.40})$, we have for $1<r<\infty$ and $t\geq\tilde{t_1}$ with some number $\tilde{t_1}>2$,
\begin{equation}\label{5.8}
\|\nabla \partial_tu(t)\|_{L^{r}(\mathbb{R}_{+}^{3})}=\|\nabla u_1(t)\|_{L^{r}(\mathbb{R}_{+}^{3})}\leq Ct^{-2-\frac{3}{2}(1-\frac{1}{r})}.
\end{equation}
Thanks to Lemma \ref{le2.3}, Lemma \ref{le4.1} and Lemma \ref{le5.4}, we find for $1<r<\infty$ and $t\geq\tilde{t_1}$,
\begin{align}\label{5.9}
\|\nabla(u\cdot\nabla u)(t)\|_{L^r(\mathbb{R}_+^3)}\leq& C\big(\|\nabla u(t)\|_{L^{2r}(\mathbb{R}_+^3)}^2+\|u(t)\|_{L^{2r}(\mathbb{R}_+^3)}\|\nabla^2u(t)\|_{L^{2r}(\mathbb{R}_+^3)}\big)\nonumber\\
\leq&Ct^{-\frac{7}{2}-\frac{3}{2}(1-\frac{1}{r})}.
\end{align}
Applying $(\ref{3.44})$ with $\ell=1$, together with $(\ref{5.8})$ and $(\ref{5.9})$, we obtain for $1<r<\infty$ and $t\geq\tilde{t_1}$,
\begin{align*}
&\|\nabla^3u(t)\|_{L^r(\mathbb{R}_+^3)}+\|\nabla^2p(t)\|_{L^r(\mathbb{R}_+^3)}\nonumber\\
\leq&C\big(\|\nabla(u\cdot\nabla u)(t)\|_{L^r(\mathbb{R}_+^3)}+\|\nabla\partial_tu(t)\|_{L^r(\mathbb{R}_+^3)}+\|\nabla(\nabla \cdot(\nabla d\odot\nabla d))(t)\|_{L^r(\mathbb{R}_+^3)}\big)\nonumber\\
\leq& Ct^{-2-\frac{3}{2}(1-\frac{1}{r})}.
\end{align*}

Using Lemma \ref{le2.3}, Lemma \ref{le4.1}, Lemma \ref{le5.4} and $(\ref{5.8})$, following the proof of $(\ref{4.2})$, we get for $1\leq k\leq3$, $1\leq j\leq 2$, $1<r<\infty$, $0<\varepsilon<\frac12$ and $t\geq 2\tilde{t_1}$,
\begin{align}\label{5.11}
&\big\|\int_{\frac{t}{2}}^{t}\int_{\mathbb{R}_{+}^{3}}\partial_{x_k}\partial_{x_j}\mathcal{M}(x,y,t-s)\mathbb{P}U(y,s)\,dyds\big\|_{L^{r}(\mathbb{R}_{+}^{3})}\nonumber\\
&+\big\|\int_{\frac{t}{2}}^{t}\int_{\mathbb{R}_{+}^{3}}\partial_{x_k}\partial_{x_3}\mathcal{M}(x,y,t-s)\mathbb{P}U(y,s)\,dyds\big\|_{L^{r}(\mathbb{R}_{+}^{3})}\nonumber\\
\leq&C\int_{\frac{t}{2}}^{t}(t-s)^{-\frac{1}{2}}s^{-\frac{5}{2}-\frac{3}{2}-\frac{3}{2}(1-\frac{1}{r})}ds+C\int_{\frac{t}{2}}^t(t-s)^{-1+\varepsilon}s^{-2-\frac{3}{2}-\frac{3}{2}(1-\frac{1}{r})}\,ds\nonumber\\
     &+C\int_{\frac{t}{2}}^{t}(t-s)^{-\frac{1}{2}}s^{-2}\big(\|\nabla d_2(s)\|_{L^{r}(\mathbb{R}_{+}^{3})}+\|\nabla^2 u_1(s)\|_{L^{r}(\mathbb{R}_{+}^{3})}+
     \|\nabla^2 d_2(s)\|_{L^{r}(\mathbb{R}_{+}^{3})}  \big )\,ds\nonumber\\
\leq &C_\varepsilon t^{-3-\frac 32(1-\frac1r)}
+C\tilde{f_3}(t)t^{-4-\frac{3}{2}(1-\frac{1}{r})},
\end{align}
where $\tilde{f_3}(t)=\sup\limits_{0<s\leq t}[s^{\frac{5}{2}+\frac{3}{2}(1-\frac{1}{r})}(\|\nabla d_2(s)\|_{L^{r}(\mathbb{R}_{+}^{3})}+\|\nabla u_2(s)\|_{L^{r}(\mathbb{R}_{+}^{3})}+\|\nabla^2 u_1(s)\|_{L^{r}(\mathbb{R}_{+}^{3})}+
     \|\nabla^2 d_2(s)\|_{L^{r}(\mathbb{R}_{+}^{3})})] $.

Following the proof of $(\ref{4.11})$, we obtain that
\begin{align}\label{5.12}
&\big\|\int_{\frac{t}{2}}^{t}\int_{\mathbb{R}_{+}^{3}}\nabla_{x}\mathcal{W}(x,y,t-s)D_1(y,s)\,dyds\big\|_{L^{r}(\mathbb{R}_{+}^{3})}\nonumber\\
\leq&C\int_{\frac{t}{2}}^{t}(t-s)^{-\frac{1}{2}}s^{-3-\frac{3}{2}-\frac{3}{2}(1-\frac{1}{r})}\,ds
     +C\int_{\frac{t}{2}}^{t}(t-s)^{-\frac{1}{2}}s^{-2}\big(\|\nabla d_2(s)\|_{L^{r}(\mathbb{R}_{+}^{3})}+\|\nabla^2 u_1(s)\|_{L^{r}(\mathbb{R}_{+}^{3})}
      \big )\,ds\nonumber\\
\leq &Ct^{-4-\frac{3}{2}(1-\frac{1}{r})}
+C\tilde{f_3}(t)t^{-4-\frac{3}{2}(1-\frac{1}{r})}.
\end{align}
Following the proofs of $(\ref{4.14})$ and $(\ref{4.20})$, we get for $1\leq k\leq3$, $1\leq j\leq 2$, $1<r<\infty$ and $t\geq 2\tilde{t_1}$,
\begin{align}\label{5.13}
&\big\|\int_{\frac{t}{2}}^{t}\int_{\mathbb{R}_{+}^{3}}\partial_{x_k}\partial_{x_j}\mathcal{W}(x,y,t-s)D_1(y,s)\,dyds\big\|_{L^{r}(\mathbb{R}_{+}^{3})}\nonumber\\
&+\big\|\int_{\frac{t}{2}}^{t}\int_{\mathbb{R}_{+}^{3}}\partial_{x_k}\partial_{x_3}\mathcal{W}(x,y,t-s)D_1(y,s)\,dyds\big\|_{L^{r}(\mathbb{R}_{+}^{3})}\nonumber\\
\leq&C\int_{\frac{t}{2}}^{t}(t-s)^{-\frac{1}{2}}s^{-5-\frac{3}{2}(1-\frac{1}{r})}ds+C\int_{\frac{t}{2}}^t(t-s)^{-1+\varepsilon}s^{-\frac{9}{2}-\frac{3}{2}(1-\frac{1}{r})}\,ds\nonumber\\
     &+C\int_{\frac{t}{2}}^{t}(t-s)^{-\frac{1}{2}}s^{-2}\big(\|\nabla d_2(s)\|_{L^{r}(\mathbb{R}_{+}^{3})}+\|\nabla u_2(s)\|_{L^{r}(\mathbb{R}_{+}^{3})}+\|\nabla^2 u_1(s)\|_{L^{r}(\mathbb{R}_{+}^{3})}+
     \|\nabla^2 d_2(s)\|_{L^{r}(\mathbb{R}_{+}^{3})}  \big )\,ds\nonumber\\
\leq &C_\varepsilon t^{-4-\frac{3}{2}(1-\frac{1}{r})}
+C\tilde{f_3}(t)t^{-4-\frac{3}{2}(1-\frac{1}{r})},
\end{align}
and
\begin{align}\label{5.14}
&\big\|\int_{\frac{t}{2}}^{t}\int_{\mathbb{R}_{+}^{3}}\nabla_{x}\mathcal{M}(x,y,t-s)\mathbb{P}U_1(y,s)\,dyds\big\|_{L^{r}(\mathbb{R}_{+}^{3})}\nonumber\\
 \leq &C\int_{\frac t2}^t(t-s)^{-\frac12}s^{-\frac{7}{2}-\frac{3}{2}-\frac{3}{2}(1-\frac{1}{r})}\,ds
      +C\tilde{f}_3(t)\int_{\frac{t}{2}}^t(t-s)^{-\frac{1}{2}}s^{-\frac92-\frac{3}{2}(1-\frac{1}{r})}\,ds\nonumber\\
      \leq &Ct^{-\frac{9}{2}-\frac{3}{2}(1-\frac{1}{r})}+C\tilde{f}_3(t)t^{-4-\frac{3}{2}(1-\frac{1}{r})}.
\end{align}
From $(\ref{4.19-})$, $(\ref{5.11})$-$(\ref{5.14})$, following the proof of $(\ref{3.68})$, we conclude that for $1<r<\infty$ and $t\geq \tilde{t_2}$ with some number $\tilde{t_2}\geq 2\tilde{t_1}$,
\begin{equation}\label{5.15}
\|\nabla d_2(s)\|_{L^{r}(\mathbb{R}_{+}^{3})}+\|\nabla u_2(s)\|_{L^{r}(\mathbb{R}_{+}^{3})}+\|\nabla^2 u_1(s)\|_{L^{r}(\mathbb{R}_{+}^{3})}+
     \|\nabla^2 d_2(s)\|_{L^{r}(\mathbb{R}_{+}^{3})}\leq Ct^{-\frac{5}{2}-\frac{3}{2}(1-\frac{1}{r})}.
\end{equation}
Using $(\ref{3.44})$ and $(\ref{5.15})$, running the same argument as in $(\ref{5.11})$-$(\ref{5.14})$, we obtain that for $1<r<\infty$ and $t\geq \tilde{t_2}$,
\begin{align*}
\|\nabla^4 u(t)\|_{L^{r}(\mathbb{R}_{+}^{3})}+\|\nabla^3p&(t)\|_{L^r(\mathbb{R}_+^3)}\leq Ct^{-\frac52-\frac{3}{2}(1-\frac{1}{r})},\\
\|\nabla^5 u(t)\|_{L^{r}(\mathbb{R}_{+}^{3})}+\|\nabla^4p&(t)\|_{L^r(\mathbb{R}_+^3)}\leq Ct^{-3-\frac{3}{2}(1-\frac{1}{r})}.
\end{align*}
Repeating the proofs of $(\ref{5.5})$-$(\ref{5.14})$, we can get for every integer $k\geq1$,
\begin{equation}\label{5.16}
\|\nabla^{2+k}u(t)\|_{L^r(\mathbb{R}_+^3)}+\|\nabla^{1+k}p(t)\|_{L^r(\mathbb{R}_+^3)}\leq \tilde{C}t^{-\frac12-\frac{2+k}2-\frac 32(1-\frac1r)}
\end{equation}
holds for $1<r<\infty$ and $t\geq\tilde{t_k}$ with some number $\tilde{t_k}>0$.
Using the Gagliardo-Nirenberg inequality $(\ref{gn-})$, we conclude that for any $k\geq1$ and $t\geq\tilde{t_k}$, $(\ref{5.16})$ is valid for $r=\infty$. Thus, we complete the proof of Lemma \ref{le5.4-}.
\end{proof}
\begin{proof}[Proof of Theorem $\ref{th1.3}$]

Theorem $\ref{th1.3}$ directly follows from Lemmas $\ref{le4.1}$-$\ref{le5.4-}$.
\end{proof}
\section*{Acknowledgements}

This work was partially supported by National Key R\&D Program of China (No. 2021YFA1002900), Guangzhou City Basic and Applied Basic Research Fund (No. 2024A04J6336), and the Scientific Research Innovation Project of Graduate School of South China Normal University (No. 2024KYLX084).

\bigskip

{\bf Data Availability:} Data sharing is not applicable to this article.

\bigskip

{\bf Conflict of Interest:} The authors declare that they have no conflict of interest.


\end{document}